\documentclass{article}

\usepackage{arxiv}

\usepackage[utf8]{inputenc} % allow utf-8 input
\usepackage[T1]{fontenc}    % use 8-bit T1 fonts
\usepackage{hyperref}       % hyperlinks
\usepackage{url}            % simple URL typesetting
\usepackage{booktabs}       % professional-quality tables
\usepackage{amsfonts}       % blackboard math symbols
\usepackage{nicefrac}       % compact symbols for 1/2, etc.
\usepackage{microtype}      % microtypography
\usepackage{lipsum}
\usepackage{graphicx}
\graphicspath{ {./images/} }
\usepackage[english]{babel}
\usepackage[utf8]{inputenc}

\usepackage{amsthm}
\usepackage{latexsym} \RequirePackage{amsthm}
\usepackage{amsmath}

\usepackage{amssymb}
\RequirePackage{makeidx}
\RequirePackage{graphicx}
\usepackage{txfonts}
\usepackage{kerkis,kmath}

\numberwithin{equation}{section}
\newcommand{\R}{\mathbb{R}}
\newcommand{\C}{\mathbb{C}}

\DeclareMathOperator{\tr}{tr}

\newcommand{\rank}{\mathop{\operator@font rank}}

\newcommand{\vertiii}[1]{{\left\vert\kern-0.25ex\left\vert\kern-0.25ex\left\vert #1
    \right\vert\kern-0.25ex\right\vert\kern-0.25ex\right\vert}}
\makeatother
\newenvironment{acks}[1][\acknowledgementsname]{\section*{#1}}{\par}
 
\if@balayout
    \renewenvironment{acks}[1][\acknowledgementsname]%
        {%
            \vskip0.5\baselineskip
            \small
            {\noindent\normalfont\sffamily\bfseries\acknowledgementsname}\par
            \begingroup\parindent 0pt\parskip 0.5\baselineskip
        }%
        {\endgroup}
\fi
\newcommand{\Z}{\mathbb{Z}} % integers
\newcommand{\N}{\mathbb{N}} %natural
\newcommand{\E}{\mathbb{E}}%mesitimi

\newtheorem{thm}{Theorem}[section]
\newtheorem{lem}[thm]{Lemma}
\newtheorem{prop}[thm]{Proposition}
\newtheorem{cor}[thm]{Corollary}
\newtheorem{assumption}[thm]{Assumption}
\theoremstyle{definition}
\newtheorem{defn}[thm]{Definition}

\theoremstyle{remark}
\newtheorem{rem}[thm]{Remark}

\title{Universality of the least singular value and singular vector delocalisation for L\'evy non-symmetric random matrices}

\author{
 Michail Louvaris \thanks{Supported by the Hellenic Foundation for Research and Innovation (H. F. R. I.)
under the “First Call for H. F. R. I. Research Projects to support Facutly members and Researchers and the
procurement of high-cost research equipment grant”, (Project Number: 1034).
} \\ Department of Mathematics
 \\National and Kapodistrian University of Athens \\Panepistimiopolis, Athens
15784, Greece.\\
 \texttt{louvarismixalis@gmail.com} \\
}

\begin{document}
\maketitle
\begin{abstract}
In this paper we consider $N \times N $ matrices $D_{N}$ with i.i.d. entries all following an $a-$stable law divided by $N^{1/a}$. We prove that the least singular value of $D_{N}$, multiplied by $N$, tends to the same law as in the Gaussian case, for almost all $a \in (0,2)$. This is proven by considering the symmetrization of the matrix $D_{N}$ and using a version of the three step strategy, a well known strategy in the random matrix theory literature. In order to apply the three step strategy, we also prove an isotropic local law for the symmetrization of matrices after slightly perturbing them by a Gaussian matrix with a similar structure. The isotropic local law is proven for a general class of matrices that satisfy some regularity assumption. We also prove the complete delocalization for the left and right singular vectors of $D_{N}$ at small energy, i.e., for energies at a small interval around $0$.
\end{abstract}

\tableofcontents

\section{Introduction and main results}
\subsection{Introduction}
The asymptotic behavior of the spectrum of random matrices has been a crucial topic of studies since Wigner's semicircle law, first proven in \cite{wigner1958distribution}. The study of the asymptotic spectral behavior of Wishart matrices was the next important result, firstly investigated in \cite{marvcenko1967distribution}, although Wishart matrices actually preceded Wigner matrices.

 The Wishart matrices, and more generally the covariance matrices, play a significant role in various scientific fields. See for example \cite{dieng2011application} and \cite{johnstone2001distribution} for applications in statistics, \cite{onatski2009testing} for application in economics and \cite{patterson2006population} for application in population genetics.  Several spectral properties of these matrices have been investigated. We focus on the case that the entries of the matrix are identically distributed, independent random variables (i.i.d.). For those matrices some significant results concern the limit of the largest eigenvalue, the asymptotic behavior of the correlation functions and the asymptotic bulk and edge behavior. For example, see \cite{pillai2014universality} or the lecture notes concerning the singular values \cite{chafai2009singular}. These results are proven for matrices whose entries have finite variance.

Besides those results, an important problem in random matrix theory is the asymptotic behavior of the least eigenvalue of covariance matrices, when the matrices' dimensions are equal. Note that the inverse of the least singular value of a matrix is equal to the operator
norm of its inverse, so an estimate of the least singular value gives control to the probability that the inverse has large norm and also gives control to its condition number. To name an illustrative application, this estimate of the least singular value for various random matrix models, plays an important role in the analysis of the performance of algorithms, see \cite{spielman2004smoothed}. In the case that the entries of the matrices are normally distributed, the limiting distribution has been described in Theorem 4.2 of \cite{edelman1988eigenvalues}, by directly computing the density of the smallest singular value multiplied by $N$.  In the general i.i.d. case, under the assumption of finite moments of sufficiently large order, the least singular value is proven to tend to the same law as the least singular value of a Gaussian random matrix, in Theorem 1.3 of \cite{tao2010random}. This phenomenon, the same asymptotic distribution for the least singular value of a matrix as in the Gaussian case, will be called universality of the least singular value for the matrix. Lastly, in the most recent papers \cite{che2019universality} and \cite{che2021universality} the authors proved that universality of the least singular values holds for more general classes of matrices.  

The above results have been focused on the finite variance cases. In the case of infinite second moment, and more specifically in the case of stable entries, there are not so many results concerning the behavior of the spectrum of covariance matrices. There are some results, mostly concerning the limit of the E.S.D. of such matrices (\cite{belinschi2009spectral},\cite{bordenave2017spectrum}) and the limit of the largest eigenvalues \cite{auffinger2009poisson} and \cite{heiny2017eigenvalues}. Moreover there are also some generalizations, which concern the limit of the largest eigenvalue of heavy tailed  autocovariance matrices in \cite{heiny2021large} and covariance matrices with heavy tailed $m$-dependent entries in \cite{basrak2021extreme}. Despite that, progress has been made concerning the symmetric matrices with heavy tailed entries. In \cite{arous2008spectrum}, the authors found the limit of the empirical spectral distribution of such matrices. Next, in \cite{bordenave2013localization} and \cite{bordenave2017delocalization} the authors proved some version of local law and examined the localization and delocalization of the eigenvectors in each of these cases. Moreover, in \cite{bordenave2017spectrum} and \cite{bordenave2011spectrum}, the authors gave a better understanding of the limiting distribution of the empirical spectral distribution by proving the convergence of resolvent of the matrix to the root of a Poisson weighted infinite tree in some operator space. Recently, in \cite{aggarwal2021goe} and \cite{aggarwal2020eigenvector}, the authors showed complete delocalization of the eigenvectors whose eigenvalues belong in some interval around 0, GOE  statistics for the correlation function and described the precise limit of the eigenvectors respectively.\footnote{GOE denotes the Gaussian Orthogonal Ensemble, i.e., symmetric matrix with independent entries (up to symmetry) where the non-diagonal entries have law $N(0,1)$ and the diagonal $N(0,2)$.}   

In this paper we prove universality for the least singular value of random matrices with i.i.d. $a-$stable entries. The methods we use also imply the complete singular vector delocalization for such matrices at small energies. We prove these results using a version of the three step strategy, a strategy developed in the last decade, which is suitable in order to obtain universality results for random matrix models, see \cite{erdos2017dynamical}.

The basic inspiration for this paper is Theorem 2.5 in \cite{aggarwal2021goe}, which proves universality of the correlation functions for symmetric L\'evy random matrices at small energies.
 Both the intermediate local law, Theorem \ref{local law}, and the theorem concerning the comparison of the entries of the resolvent, Theorem \ref{theorima 3.15 goe}, are similar to Theorem 3.5 and Theorem 3.15 of \cite{aggarwal2021goe} respectively, adjusted to our set of matrices. For the intermediate local law we also use a lot of results from \cite{bordenave2013localization} and \cite{bordenave2017delocalization}.

 Results and methods from \cite{bourgade2017eigenvector} and \cite{che2019universality} had significant influence to this paper as well. In particular the isotropic local law in Sections \ref{section gia pertubed local law} is an analogue of Theorem 2.1 in \cite{bourgade2017eigenvector}, proven for a different class of matrices. Moreover universality for the least singular value of random matrices after perturbing them by a Brownian motion Matrix can be found in Theorem 3.2 of \cite{che2019universality}. So several results from Sections \ref{section universality for the least singular value} and \ref{section pou apodiknietai to theorima} are based or influenced by results of \cite{che2019universality}. 
  
 \subsection{Main results} 
  Fix a parameter $a \in (0,2)$. A random variable $Z$ is called $(0,\sigma)$ $a$-stable law if
\begin{equation}\label{eystathis katanomes}
 \E(e^{itZ})=\exp(-\sigma^{a}|t|^{a}) , \text{ for all } t \in \R.
\end{equation}

\begin{defn}\label{o orismos toy pinaka}
Set
\begin{align}\label{sigma}
    \sigma:=\left(\frac{\pi}{2\sin(\frac{\pi a}{2})\Gamma(a)}\right)^{1/a}>0,
    \end{align}
and let $J$ be a symmetric random variable with finite variance and let $Z$ be a $(0,\sigma)$ $a$-stable random variable, independent from $J$. Then, define the matrix $D_{N}(a)=\{d_{i,j}\}_{1\leq i,j\leq N}$ to be random matrix with i.i.d. entries, all having the same law as $ N^{-1/a}(J+Z)$. In
what follows, we may omit explicitly indicating the dependence of the matrices $D_N$ on
the parameters $a$ and $N$, and use the notation $D$.

Lastly, fix parameters $C_{1},C_{2}$ such that 
\begin{align}\label{tailbound}
    \frac{C_{1}}{Nt^{a}+1}\leq \mathbb{P}\bigg(|d_{i,j}|\geq t\bigg)\leq \frac{C_{2}}{Nt^{a}+1}.
    \end{align}
Such parameters exist due to the tail properties of the stable distribution. See \cite{samorodnitsky2017stable}, Property 1.2.8.
\end{defn}
The parameter $\sigma$ is chosen in \eqref{sigma} like so, in order to keep our notation consistent with previous works such as \cite{aggarwal2021goe},\cite{aggarwal2020eigenvector},\cite{bordenave2017delocalization} and \cite{bordenave2013localization}. This parameter can be altered by a rescaling.\\
Moreover, denote $\mathrm{\rho_{sc}}$ the probability density function of the semicircle law, i.e.,

$$\mathrm{\rho_{sc}}(x)=\mathbf{1}\left\{|x|\leq 2\right\}\frac{1}{2\pi}\sqrt{4-x^{2}}.$$
Furthermore, set
\begin{equation}\label{semicirle law kai logos}
\xi:=\frac{\rho_{a}(0)}{\mathrm{\rho_{sc}(0)}},
\end{equation}
where $\rho_{a}$ is the density of the limiting distribution of the empirical measure of the singular values of $D$ and their negative ones and is described in Proposition \ref{orio stieltjes prot}.

In what follows we will use the standard Big $\mathcal{O}$ notation. Specifically given two functions $f,g$, we will say $f=\mathcal{O}(g)$ if and only if there exists a constant $C>0$ independent of any other parameter such that
\begin{align}
    \limsup_{x\to \infty}\left|\frac{f(x)}{g(x)}\right|=C<\infty,
\end{align}
where the constant $C>0$ will be independent of any other parameter. If the constant $C$ depends on some parameter(s) $c$ defined earlier, we will write $f=\mathcal{O}_{c}(g)$. Moreover if the constant $C=0$ then we will write $f=o(g)$.

 Our main result shows that the least singular values of $D_{N}$ are universal as N tends to infinity. The analogous result for matrices with finite variance entries was proven in \cite{tao2010random}. We also prove that the left and right singular vectors of $D_{N}$ are completely delocalized for small energies, in the following sense.
\begin{thm}\label{TO THEORIMA}
There exists a countable set $\mathcal{A}$, subset of $(0,2) $, with no accumulation points in (0,2) such that the following holds.
Let $\{D_{N}(a)\}_{N \in \N}$ be sequences of  matrices, where $D_{N}(a) \in \R^{N \times N}$  with i.i.d. entries all following $N^{-1/a}(Z+J)$, where $Z,J$ as in Definition \ref{o orismos toy pinaka}. Then for every $a \in (0,2)\setminus\mathcal{A}$:
\begin{enumerate}

  \item Let $s_{1}(D_{N}(a))$ denote the least singular value of $D_{N}(a)$. Then, there exists  $c>0$ such that for all $r \geq 0$   
  \begin{equation}
  \mathbb{P}\bigg(N\xi s_{1}(D_{N}(a))\leq r\bigg)=1-\exp\left(-\frac{r^{2}}{2}-r\right)+\mathcal{O}_r(N^{-c}).\end{equation}
  \item
  For each $\delta>0$ and $D>0$ there exist constants $C=C(a,\delta,D)>0$ and $c=c(a)$ such that:
  \begin{equation}
\mathbb{P}\bigg(\max\bigg\{\lVert|u\lVert|_{\infty}: u \in \mathcal{B}_{N}\bigg\}>N^{\delta-\frac{1}{2}}\bigg)\leq C N^{-D}.
\end{equation}
where $\mathcal{B}_{N}$ is the set of eigenvectors of $D_{N}D^{T}_{N}$ or $D^{T}_{N}D_{N}$, normalized with the Euclidean norm, whose corresponding eigenvalues belong to the set $[-c,c]$. 
  \end{enumerate}
\end{thm}
The proof of Theorem \ref{TO THEORIMA} can be found in Subsection \ref{subsection pou deixnoume to theorima}. 
\begin{rem}
    The set $\mathcal{A}$ for which Theorem \ref{TO THEORIMA} cannot be applied is conjectured to be empty. Its presence is due to some $a-$dependent fixed point equations in \cite{bordenave2017delocalization}, which we use and can be inverted only if $a\notin \mathcal{A}$. 
\end{rem}

Moreover, we can generalize the proof of Theorem \ref{TO THEORIMA} to the joint distribution of the bottom $k $ singular values in the following sense.
\begin{thm}\label{GIA TIS MIKRES IDIOTIMES}
Fix a positive integer $k$. Let $\mathcal{A} \subseteq (0,2)$ be the countable set of Theorem \ref{TO THEORIMA}.
Then define, as in Definition \ref{o orismos toy pinaka}, $\{D_{N}\}_{N \in \N}$ with i.i.d. entries all following $N^{-1/a}(Z+J)$, where $Z$ is $(0,\sigma)$ $a-$stable for $a \in (0,2)\setminus \mathcal{A} $. Also let $\{L_{N}\}_{N}$ be a sequence of $N \times N$ i.i.d. matrices, with entries following the same law as a centered normal random variable with variance $\frac{1}{N}$.
\\Also for any matrix A define
\[\Lambda_{k}(A):=(Ns_{1}(A),\cdots,Ns_{k}(A)),\]
where $\{s_{i}(A)\}_{i \in [N]}$ are the singular values of $A$ arranged in increasing order. Also denote $1_{k}=(1,\cdots,1)$ and for all $E \in \R^{k}$ 
\[\Omega(E):=\{x \in \R^{k}:x_{i}\leq E_{i}  \text{ for all } i \in [k]\}.\]
Then there exists $c>0$ such that for all $E \in \R^k$
\begin{equation}\mathbb{P}\bigg(\Lambda_{k}(L_{N}) \in \Omega(E-N^{-c}1_{k})\bigg)-\mathcal{O}_E(N^{-c})\leq \mathbb{P}\bigg(\Lambda_{K}(\xi D_{N})\in \Omega(E)\bigg) \leq \mathbb{P}\bigg(\Lambda_{K}(L_{N}) \in \Omega(E+N^{-c}1_{k})\bigg)+\mathcal{O}_E(N^{-c}).\end{equation}
\end{thm}
The proof of Theorem \ref{GIA TIS MIKRES IDIOTIMES} is similar to that of Theorem \ref{TO THEORIMA} and therefore is omitted.
\\Note that the universal limiting distribution of $\Lambda_{k}(L_{N})$ is explicitly given in \cite{tao2010random}.
\\Moreover, by the way that we will prove Theorem \ref{TO THEORIMA}, we can prove a similar result for the gap probability in the symmetric case. The proof of the following corollary will again be omitted due to its similarity to the proof of Theorem \ref{TO THEORIMA}. 
\begin{cor}\label{Mikroteri idiotimi symetrikon}
Let $M_{N}$ be an $N \times N$ symmetric matrix with i.i.d. entries (with respect to symmetry) and let all entries follow the same law as $N^{-1/a}(Z+J)$, where $Z,J$ are defined in Definition \ref{o orismos toy pinaka} for $\alpha \in (0,2)\setminus \mathcal{A}$ . Here $\mathcal{A}$ is the set of Theorem \ref{TO THEORIMA}. Also let $W_{N}$ be a GOE matrix ($N \times N $ symmetric, with i.i.d. centered Gaussian entries, with variance $N^{-1}$). Arrange the eigenvalues of $M_{N}$ and $W_{N}$ in increasing order. Then there exists $\delta>0$ such that for any $r>0$,
\begin{equation}
    \left|\mathbb{P}\left(\#\left\{i \in [N]:N\lambda_{i}(M_{N})\in \left(-\frac{r}{2},\frac{r}{2}\right)\right\}=0\right)-\mathbb{P}\left(\#\left\{i \in [N]:N\lambda_{i}(W_{N})\in \left(-\frac{r}{2},\frac{r}{2}\right)\right\}=0\right)\right|\leq \mathcal{O}_r(N^{-\delta}).
    \end{equation}
 For the Gaussian case, the limiting distribution of the gap probability is given in Theorem 3.12 of \cite{anderson2010introduction}.
\end{cor}
\begin{rem}
Note that by Theorem \ref{TO THEORIMA}, the least singular value of a random matrix with i.i.d. entries, all following an $a-$stable distribution, are of order $\mathcal{O}(N^{\frac{1}{a}-1})$ for $a \in (0,2)\setminus \mathcal{A}$. So for $a \in (0,1)\cap \mathcal{A}^{c}$ the least singular value, without normalization, tends to $\infty $, which is different from the finite variance case. 
\end{rem}
\section{Preliminaries and sketch of the proof}
\subsection{Preliminaries}
In this subsection we present some necessary definitions and lemmas.

Firstly fix parameters $a,b, \rho,\nu$ such that
\begin{equation}\label{statheres}
  a \in (0,2), \ \ \ \nu=\frac{1}{a}-b>0, \ \ \ 0<\rho<\nu, \ \ \ \frac{1}{4-a}<\nu < \frac{1}{4-2a}, \ \ \  a\rho<(2-a)\nu .
\end{equation}
Note that given $a\in (0,2)$, such parameters will exists. Moreover $\nu>0$ is the level on which we will truncate the matrix $D_N$ in \eqref{truncate}. This truncation is crucial to our analysis as is explained later in Subsection \ref{sketch of the proof}. The rest of the restrictions for the parameters in \eqref{statheres}, will be explained later in the choice of $\epsilon_0$ in \eqref{(4.25)}, in the proof of Theorem \ref{theorima 3.15 goe}. 

Next we give some preliminaries definitions and lemmas.
\begin{defn}
For each $a \in (0,\infty)$ and $u \in \C^{N}$ we will use the notation \[\|u\|_{a}=\left(\sum_{i=1}^{N}|u_{i}|^{a}\right)^{1/a}.\]
Moreover if $N=1$ and $a=2$, we will use the notation $|u|$ for the Euclidean norm.
\end{defn}

\begin{defn}\label{orismos E.S.D.}
Fix an $N \times N$ matrix $Y$. Then the empirical spectral distribution of $Y$ is the measure
\[\mu_{Y}:=\frac{1}{N} \sum_{i=1}^{N}\delta_{\lambda_{i}(Y)},\]
where $\delta_{x}$ is the Dirac measure for $x \in \R$ and $\{\lambda_{i}(Y)\}_{i \in [N]}$ are the eigenvalues of Y.
We will also use the notation $\lambda_{\max}(Y)$ for the largest eigenvalue of Y. 
\end{defn}
\begin{defn}
Let $M$ be an $ N \times N$ real matrix. Then the $2N \times 2N$ matrix
\[\begin{bmatrix} 0 & M^{T} \\M & 0 
\end{bmatrix}\]
 is called the symmetrization of $M$.
\end{defn}
\begin{defn}\label{orismos b-removals}
Let $H_{N}$ be the symmetrization of $D_{N}$, i.e.,
\[H_{N}=\begin{bmatrix} 0 & D_{N}^{T} \\ D_{N} & 0
\end{bmatrix}.\]
Then define the matrix $X_{N}=\{x_{i,j}\}_{1\leq i,j\leq 2N}$ such that \begin{align}\label{truncate}
    x_{i,j}:=h_{i,j}\mathbf{1}\left\{N^{1/a}|h_{i,j}|\geq N^{b}\right\}.
    \end{align}
 The elements of $X_{N}$ (in the non-diagonal blocks) are called the $b-$removals of a deformed $(0,\sigma)$ $a$-stable law.
We also define the matrices $A_{N}:=H_{N}-X_{N}$, the matrix $E_{N}$ whose symmetrization is $A_{N}$ and the matrix $K_{N}$ whose symmetrization is $X_{N}$, i.e.,
\[X_{N}= \begin{bmatrix} 0 & K_{N}^{T} \\ K_{N} & 0
\end{bmatrix}, \ \ \ A_{N}=\begin{bmatrix} 0 & E_{N}^{T} \\ E_{N} & 0 \end{bmatrix} \]
Furthermore, define the matrix $L_{N}$ to be an $N\times N$ matrix with i.i.d. entries all following the law of a normal, centered random variable with variance $\frac{1}{N}$, and its symmetrization $W_{N}$. 
In what follows, we may omit the dependence of the matrices defined here on $N$, for notational convenience.
\end{defn}
\begin{rem}\label{Singular values kai idiotimes sym}
Note that the eigenvalues of $H$ are exactly the singular values of $D$ and their respective negative ones since
\[\det(\lambda\boldsymbol{\cdot} \mathbb{I}_{2N}-H)=\det(\lambda^{2}\boldsymbol{\cdot} \mathbb{I}_{N}-D^{T}D)
.\]

\end{rem}
Moreover, note that if we prove delocalization for the eigenvectors of $H$ in the sense of the second part of Theorem \ref{TO THEORIMA}, then we will have an understanding over the delocalization of the left and right singular vectors of $D$, because of the following remark.
\begin{rem}
If $J_{1},J_{2}$ are the matrices with columns the normalized left and right singular vectors of $D$, which by the singular value decomposition gives us that $J_{1} DJ_{2}=\text{diag}(s_{1},s_{2}\cdots,s_{N})$, then one can compute that the matrix 
\[\frac{1}{\sqrt{2}}\begin{bmatrix} J_{2}^{*} & J_{1} \\ J_{2}^{*} & -J_{1} \end{bmatrix},\]
has columns the normalized eigenvectors of $H$.
\end{rem}
So in what follows, we will focus on proving delocalization for the eigenvectors and universality of the least positive eigenvalue for $H$.
\\We will use the notation $\operatorname{Im}(z)$ for the imaginary part of any $z \in \C$ and $\C^{+}:=\{z \in \C: \operatorname{Im}(z)>0\}$.
\\ Furthermore we need the following definitions. 
\begin{defn}
Let $M$ be an $N \times N$ matrix. The matrix $Y=(M-zI)^{-1}$ for $z \in \C^{+}$ is called the resolvent of $M$ at $z$.
\end{defn}
\begin{defn}\label{orismos stieltjes}(Stieltjes transform)
Let $M$ be an $N \times N$ matrix and let $\mu_{M}$ be its empirical spectral distribution. Then for each $z \in \C^{+}$, we define its Stieltjes transform as the normalized trace of its resolvent, i.e.,
\[m_{M}^{N}(z):=\int \frac{1}{x-z} d \mu_{M}(x)=\frac{1}{N} \operatorname{tr}(M-z\mathbb{I})^{-1}. \] 
In what follows, we might omit the dependence on the dimension of the Stieltjes transform or on the matrix, when it is clear to which matrix we refer. 
\end{defn}
 \begin{defn}\label{orismos t}
     In what follows we will use the following notation. 
     \begin{align}
         t:= N \operatorname{Var}(E_{1,1}).
     \end{align}
        Moreover, in Corollary \ref{var(E)} we prove that $t \rightarrow
 0$ as $N \to \infty$.
 \end{defn}

In the next Lemma we give an estimate for the entries of $A$.

\begin{lem}[\cite{aggarwal2021goe}, Lemma 4.1]\label{kommenes t.m.}
Let $R\geq N^{-1/a}$ and $p>a$. Then there exist a small constant  $c=c(a,p,C_1)$ and a large constant $C=C(a,p,C_2)$ such that
 \[cN^{-1}R^{p-a}\leq \E |D_{1,1}|^{p} \mathbf{1}\left\{|D_{1,1}|\leq R\right\}\leq C N^{-1}R^{p-a}.\]
 Here $D_{1,1}$ is the (1,1)-entry of $D_{N}$.
 Here $C_{1},C_{2}$ are the parameters from \eqref{tailbound}.
 \end{lem}
 A direct application of the previous result for $R=N^{-\nu}$ and $p=2$ implies the following.
 \begin{cor}\label{var(E)}
 The entries of $E_{N}$ satisfy the following
 \[cN^{\nu(a-2)}\leq N \operatorname{Var}(E_{1,1})\leq C N^{\nu(a-2)}.\]
 \end{cor}
 \begin{rem}\label{typical scale}
     Note that the convergence of the E.S.D. of a sequence of random matrices, implies that the typical scale of an eigenvalue is $\frac{1}{N}$ (at least in the bulk of the spectrum) of the limiting distribution of the  E.S.D.   
 \end{rem}
\begin{defn}
Let $\mathop{F}(u)$ be a family of events indexed by  some parameter(s) $u$. We will say that $\mathop{F}(u)$ holds with overwhelming probability, if for any $D>0$ there exists an $N(D,u)$ such that for all $N \geq N(D,u)$
\[\mathbb{P}(F(u))\geq1- N^{-D}.\]
uniformly in $u$.
\end{defn} 
Next we present a measure, for which in Theorem \ref{local law} we will prove that it is the limiting distribution of the E.S.D. of $X_N$.
\begin{defn}
  Let $M_N$ be a sequence of symmetric $N \times N$ matrices with i.i.d. entries (up to symmetry) and for each $N\in \N$ let all the entries follow the same law $N^{-1/a}(Z+J)$, where $Z,J$ are defined in Definition \ref{o orismos toy pinaka}. In what follows for any $z\in \C^+$, we will use the notation 
    \begin{align}
            m_{a}(z):= \frac{1}{N} \lim_{N \to \infty } \operatorname{tr}( M_{N}-z\mathbb{I})^{-1}.
    \end{align}
So $m_{a}$ is the Stieltjes transform of the limiting distribution of the E.S.D. of the sequence of matrices $M_{N}$, see Theorem 1.4 of \cite{arous2008spectrum}. The properties of $m_{a}$ are described next in Proposition  \ref{orio stieltjes prot}.
\end{defn}
\begin{prop}\label{orio stieltjes prot}
  The Stieltjes transform $m_{a}(z)$ of the limiting distribution of the E.S.D. of the matrices $M_{N}$ satisfies the following equation
  \[m_{a}(z)=i \psi_{a,z}(y(z)),\]
where 
\begin{align}
&\phi_{a,z}(x)=\frac{1}{\Gamma(\frac{a}{2})}\int_{0}^{\infty}t^{\frac{a}{2}-1}e^{itz}e^{-\Gamma(1-\frac{a}{2})t^{\frac{a}{2}}x}dt,\\
&\psi_{a,z}(x)=\int_{0}^{\infty}e^{itz}e^{-\Gamma(1-\frac{a}{2})t^{a/2}x}dt,\\ \label{eksisosi stieltjes prop}&
  y(z)=\phi_{a,z}(y(z)),
\end{align}
where \eqref{eksisosi stieltjes prop} is proven to have a unique solution on $\C^{+}$. Moreover the limiting probability density function $\rho_{a}$ is bounded, absolutely continuous, analytic except at a possible finite set and with density at $0$ given by $$\rho_{a}(0)=\frac{1}{\pi}\Gamma(1+\frac{2}{a})\left(\frac{\Gamma(1-\frac{a}{2})}{\Gamma(1+\frac{a}{2})}\right)^{1/a}.$$  These results are proven in Proposition 1.1 of \cite{belinschi2009spectral} and Theorem 1.6 of \cite{bordenave2011spectrum}.  
\end{prop}
\begin{rem}
    Later in Theorem \ref{local law}, we will prove that the Stieltjes transform of $X_N$ also converges to $m_a$. So we will refer to the measure whose Stieltjes transform is $m_a$, as the limiting measure of the E.S.D. of $X_N$.
\end{rem}
\subsection{Sketch of the proof}\label{sketch of the proof}
 Now we are ready to present a sketch of the proof. At this point we will try to avoid as much technicalities as possible. In order to prove universality, meaning the same asymptotic distribution for the least singular value of $D_{N}$ as in the Gaussian case, we are going to follow the three step strategy, a well known strategy in random matrix theory literature. Some of the most fundamental results concerning this method can be found in \cite{erdos2017dynamical} and in \cite{landon2019fixed}, which focus on proving universality of the correlation function for symmetric matrices. The key idea is that after a slight perturbation of a random matrix by a Brownian Motion matrix, the resulting matrix should behave as a Gaussian one, given that the initial matrix satisfies some mild assumption concerning its Stieltjes transform. This idea is exploited in the study of the evolution of the eigenvalues and the eigenvectors via stochastic differential equations. This method was crucial to the proof of the Wigner-Dyson-Mehta conjecture, see for example \cite{erdos2017dynamical}. The three step strategy has also been used in establishing universality of the least singular value for random matrices, see for example \cite{che2019universality}, \cite{che2021universality}. Specifically in our case: 
\begin{itemize}
\item \textbf{First step:} We investigate the asymptotic spectral behavior of $X$ at an "intermediate" scale. At this step we prove that the matrix $X$ satisfies the necessary conditions, which insure that after a slight perturbation by a Brownian motion matrix universality will hold. This is done in Section \ref{section local}. Note that by definition, $X$ contains the "big" elements of $H$. So the first step involves proving two estimates. One comparison of the Stieltjes transform of the E.S.D. of $X$ with the Stieltjes transform of its limiting measure, and one bound for resolvent entries of $X$. Set $m_{X}(z)$ the Stieltjes transform of $X$ and $R_{i,j}(z)$ the resolvent of $X$ at $z$. In particular we wish to show that the following events
\begin{align}
    &\left|m_{a}(z)-m_{X}(z)\right|=o(1),
    \\&\max_{j \in [2N]}\left| R_{j,j}(z)\right|= \mathcal{O}\left(\log^C(N)\right), \text{ for some } C>0,
\end{align}
hold with overwhelming probability for any $z: \operatorname{Im}(z)\geq N^{\delta-\frac{1}{2}}$ for any small enough $\delta>0$ and $\operatorname{Re}(z)$  in some $N-$independent interval. These results are called intermediate because the natural scale would be $\operatorname{Im}(z)\geq N^{-1+\delta}$, as is explained Remark \ref{typical scale}.   
 \item \textbf{Second step}: We consider the perturbed matrix $X+\sqrt{t}W$, where $W$ is the symmetrization of a full centered Gaussian matrix with i.i.d. entries with variance $\frac{1}{N}$, and $t$ is chosen so that the variances of the entries of $\sqrt{t}W$ and of $A$ match. It can be computed that $t \sim N^{\nu (a-2)}$. 
 
 \quad The level of the intermediate scale local law in the previous step, is justified in this part of the proof. In order to apply universality Theorems for the matrices after slightly perturbing by Brownian motion matrices,  see for example Theorem 3.2 of \cite{che2019universality}, we wish the variances of the $\sqrt{t}W$ to be above the intermediate scale of the local law. Since $N^{\delta-\frac{1}{2}}=o(t)$, for small enough $\delta>0$, this is implied.    
 
\quad Roughly, what we need to prove at this step is that the desired properties, delocalization of the eigenvectors and universality of the least singular value, hold for the matrix $X+\sqrt{t}W$. 

\quad So for the first part of the second step, we prove universality of the least positive eigenvalue for $X+\sqrt{t}W$. This is based on the regularity of the Stietljes transform of $X$, proven in the previous step, and some results from \cite{che2019universality}. More precisely at the first part of the second step we prove that
\begin{align}
    \lim_{N\to \infty} \mathbb{P}\left( N \xi \lambda_{N}(X+\sqrt{t}W) \geq r\right)=\lim_{N \to \infty} \mathbb{P}\left( N \lambda_N (W)\geq r \right) , \text{ for any } r\in \R^+.
\end{align}

This result is proven in Section \ref{section universality for the least singular value}.

\quad  In most of the universality-type theorems the fact that the entries have finite variances play a significant role, see for example Lemma 15.4 in \cite{erdos2017dynamical}. In \cite{che2019universality} the authors showed universality of the least singular value for sparse random matrix models. Firstly, they prove universality of the least singular value for the sparse models after slightly perturbing them by a Brownian motion matrix and then they remove the Brownian motion matrix. This is done by using results which take advantage of the fact that the entries have finite variance, for example see Lemma 5.14 in \cite{che2019universality}. Since our model does not have entries with finite variance, we will need to compare the matrices $X+\sqrt{t}W$ and $X+A$ with a different method. Essential to that method is the fact that the resolvent entries of $X+\sqrt{t}W$ do not grow very fast. Specifically set $T_{i,j}(z)$ to be the resolvent of $X+\sqrt{t}W$ at $z$.
In particular in Section \ref{section gia pertubed local law} we prove that for any small $\delta>0$ the $\delta-$dependent events
\begin{align}\label{overwhelming section 5}
\sup_{i,j}|T_{i,j}(z)|\leq N^{\delta}
\end{align}
hold with overwhelming probability, and for all $z:\operatorname{Im}(z)\geq N^{\epsilon-1}$ for any small $\epsilon>0$, very close to the natural scale in Remark \ref{typical scale}. It is known that bounds as the one in \eqref{overwhelming section 5} imply the complete eigenvector delocalization for the matrix $X+\sqrt{t}W$.

  In order to establish \eqref{overwhelming section 5}, we prove something better. A \textbf{universal} result which compares the entries of the resolvent of any matrix, which satisfies some mild regularity assumption, Assumption \ref{Assumption for local pertubed }, with the additive free convolution of the matrix with the semicircle law. Thus, the largest part of 
Section \ref{section gia pertubed local law} is mostly independent for the rest of the paper. 
\item \textbf{Third step:} We first compare the resolvent of $X+A$ and $X+\sqrt{t}W$. During the second step we have proven the desired properties, eigenvector delocalization and universality of the least eigenvalue for the matrix $X+\sqrt{t}W$, so we need to find a way to quantify the transition from the matrix $X+\sqrt{t}W$ to $X+A$ in order to prove the same properties for $H$. This is done by introducing the matrices
\[H^{\gamma}:= X+ \sqrt{t}(1-\gamma^{2})^{1/2}W+ \gamma A, \text{ for all } \gamma \in [0,1].\]
We manage to prove that the resolvent entries of $H^{\gamma}$ are asymptotically close for all $\gamma \in [0,1]$, in Theorem \ref{theorima 3.15 goe}. 
Similarly we study the continuity properties for $\gamma \in [0,1]$ of the functions 
\begin{equation}\label{orismos q sto introduction}
q\left(\frac{N}{\pi} \int_{\frac{-r}{N}}^{\frac{r}{N}} \operatorname{Im}(m_{\gamma}(E+i\eta) dE \right),
 \end{equation}
where $m_{\gamma}$ is the Stieltjes transform of the matrix $H^{\gamma}$ and $\eta$ is of order $N^{-\delta-1}$, below the natural scale. Eventually in \eqref{anisotita Capeiro sinartiseon gia kathe gamma gia mikro h} we prove that the functions defined in \eqref{orismos q sto introduction} are asymptotically close for any $\gamma \in [0,1]$. 
\\Next we introduce the functions 
\[\iota_{N}(Y,r):=\#\left\{i\in [N]:\lambda_{i}(Y)\in (-r,r)\right\},\]
where $Y$ is a symmetric $N \times N$ matrix, $\{\lambda_{i}\}_{i \in [N]}$ are the eigenvalues of $Y$ and $r$ is any positive number.
So, it suffices to prove that there exists $c>0$ such that for any $r>0$
\begin{equation}\left|\mathbb{P}\left[ \iota_{2N}\left(X+\sqrt{t}G,\frac{r}{N}\right)=0\right]-\mathbb{P} \left[\iota_{2N}\left(X+A,\frac{r}{N}\right)=0\right]\right|\leq \mathcal{O}_r(N^{-c}).\end{equation}
In order to prove the latter, we approximate the quantities $\mathbb{P}(i_{2N}(H^{\gamma},\frac{r}{N})=0)$ by appropriately choosing functions of the form \eqref{orismos q sto introduction}.
This is done in Lemma \eqref{anisotita poy sigrini to gap probability me Capeiro}. After combining the results above, we conclude the proof in Subsection \ref{subsection pou deixnoume to theorima}.
\end{itemize}
\section{Intermediate local law for X}\label{section local} Consider the matrices $H_{N}$ and $X_{N}$ as they are defined in Definition \ref{orismos b-removals}. In this section we are going to establish the local law (Theorem \ref{local law}) for the $b$-removals of the matrix $H$, i.e., the matrix $X$. What we mean by local law is convergence of the Stieltjes transform of $X$ to its asymptotic limit, for complex numbers $z$ that depend on the dimension $N$ in some sense. 
 \\We will also use the notation 
 \begin{align}\label{Resolvent of X}
R(z)=(X-(E+i\eta)I)^{-1},
\end{align}
for $z=E+i\eta$. In what follows we might abbreviate the dependence from the parameter $z$.
 
 A precise formulation of this result is the following. There exists $C=C(a,b,\delta)$ such that
\begin{equation}\mathbb{P}\left(\sup_{E \in (-\frac{1}{C},\frac{1}{C})}\sup_{\eta \geq N^{\delta-\frac{1}{2}}}|m_{a}(E+i\eta)-m_{X}(E+i\eta)|\geq \frac{1}{N^{a \delta/8}}\right)\leq \exp\left(-\frac{(\log(N))^{2}}{C}\right),
\end{equation}
where the properties of  $m_{a}(z)$ are described in Proposition \ref{orio stieltjes prot}.

 We also prove that for all $z$ for which the local law holds, the diagonal entries of the resolvent of $X$ are almost bounded. More specifically for any large enough $N \in \N$ it is true that,
\begin{equation}
\mathbb{P}\left(\sup_{E \in (-\frac{1}{C},\frac{1}{C})}\sup_{\eta \geq N^{\delta-\frac{1}{2}}}\max_{j \in [2N]}|R_{j,j}|>C \log^{C}(N)\right)\leq C \exp\left(-\frac{(\log(N))^{2}}{C}\right).
\end{equation}
In order to establish those results we will need to analyze the resolvent of $X$, in order for us to compare it with $m_a$. 
The main influence for this step is Theorem 3.5 of \cite{aggarwal2021goe}, where an intermediate local law is proven for symmetric heavy tailed random matrices.
The main difference of the proof of the intermediate local law for our set of matrices from the symmetric case is that, by construction, only half of the minors of the resolvent will participate in the sum of the reductive formula from Schur complement formula. This difference is not crucial since we also prove that each of the diagonal entries of the resolvent of the matrix is identically distributed. Moreover, by Corollary \ref{Anisotites sigentrosis}, the sum of half of the diagonal entries of the resolvent is concentrated around its mean, like the sum of all its diagonal entries. The rest of the proof remains almost the same, but we will include most of the proofs for completeness of the paper.

 Firstly we need to give a more detailed description of the limiting distribution. 

\subsection{Preliminaries for the intermediate local law}\label{subsection gia periergous xorous}
    In \cite{arous2008spectrum} the authors proved that the E.S.D. of symmetric matrices with heavy tailed entries, converge in distribution to a deterministic measure and they described it. Next in \cite{belinschi2009spectral}, the authors described the limit of the sample covariance matrices. Next the authors in \cite{bordenave2017delocalization} and \cite{bordenave2013localization} proved local laws for symmetric heavy tailed matrices at an intermediate scale larger than $N^{\delta-\frac{1}{2}}$. Lastly in \cite{aggarwal2021goe}, the authors proved a local law at the intermediate scale $N^{\delta-\frac{1}{2}}$. All the previously mentioned results, are based on solving a fixed point equation. In the most recent results, these fixed equations are solved more generally, in a metric space which we are going to present in this subsection.

Next we present the metric space in which we will work with in order to prove an intermediate local law for $X$. The results we present here can be also found in \cite{bordenave2017delocalization}.
\begin{defn}
For any $u,v \in \C$ define the following "inner product"
\[(u|v):=u\operatorname{Re}(v)+\bar{u}\operatorname{Im}(v)=\operatorname{Re}(u)(\operatorname{Re}(v)+\operatorname{Im}(v))+i\operatorname{Im}(u)(\operatorname{Re}(v)-\operatorname{Im}(v)).\]
One may compute the following 
\begin{equation}\label{anisotita (u|v)<|u||v|}
(u|1)=u, \ \ (-iu|e^{\pi i/4})=\operatorname{Im}(u)\sqrt{2}, \ \ |(u|v)|\leq 2|u||v|. 
\end{equation}
\end{defn}
\begin{defn}\label{perierges normes}
Set $\mathbb{K}=\C^{+}\cap \{z \in \C: \operatorname{Re}(z)>0\}$ and $\mathbb{K}^{+}=\bar{\mathbb{K}}$.
Let $H_{w}$  be the space of the $C^{1}$, $g:\mathbb{K}^{+}\rightarrow \C$ such that $g(\lambda u)= \lambda^{w}u$ for each $\lambda>0$. Set also 
$S^{1}_{+}=\overline{S^{1} \cap \mathbb{K}^{+}}$ where $S^{1}$ is the unit sphere on $\C$ with respect to the Euclidean norm. Following equation (10) of \cite{bordenave2017delocalization}, define for each $r \in [0,1)$ a norm on $H_{r}$
\[|g|_{\infty}=\sup_{u \in S^{1}_{+}}|g(u)|,\]
\[|g|_{r}=|g|_{\infty}+\sup_{u \in S_{+}^{1}}\sqrt{|(i|u)^{r}\partial_{1}g(u)|^{2}+|(i|u)^{r}\partial_{2}g(u)|^2}.\]
Here,
\[\partial_{1}g(x+iy)=\frac{d g(x+iy)}{dx}\]
and likewise
\[\partial_{2}g(x+iy)=\frac{d g(x+iy)}{dy}.\]
Next, define the spaces $H_{w,r}$ the completion of $H_{w}$ with respect to the $|g|_{r}$ norm. Further define $H_{w,r}^{\delta}\subseteq H_{w,r}$ to be the set $\{g \in H_{w,r}:\inf_{u \in S^{1}_{+}}|\operatorname{Re}(g(u))|>\delta\}$. Also define the set \[H_{w,r}^{0}=\cup_{\delta>0}H_{w,r}^{\delta}.\]
Further, abbreviate $H_{w}^{\delta}:=H_{w,0}^{\delta}$
\end{defn}
\begin{rem}\label{sigrisi sup. norm gia periergous xorous}
For any $g\in H_{r}$, by construction it is true that
\[|g|_{\infty}\leq |g|_{r}.\]
\end{rem}
Next, we present some lemmas concerning the metric spaces we presented and the fixed point equation we wish to solve.
\begin{lem}[\cite{bordenave2017delocalization},Lemma 5.2]\label{Frasontas F1-F2 se periergous xorous}
Let $r \in (0,1) $ \ and $u \in S_{+}^{1}$ and $x_{1},x_{2}\in \mathbb{K}^{+}$ and let $\eta \in (0,1)$ such that $|x_{1}|,|x_{2}|\leq \eta^{-1}$. Set $F_{k}(u)=(x_{k}|u)^{r}$ for $k\in \{1,2\}$. Then there exists a constant $C(r)$ such that for any $s \in (0,r)$ we have that
\begin{equation}\label{sigrisi esoterikon ginomenon}
|F_{k}|_{1-r+s}\leq C |x_{k}|^{r}, \ \ \ |F_{1}-F_{2}|_{1-r+s}\leq C \eta^{-r}(|x_{1}-x_{2}|^{r}+\eta^{s}|x_{1}-x_{2}|^{s}).
\end{equation}
Furthermore, if we further assume that $\operatorname{Re}(x_{1}),\operatorname{Re}(x_{2})\geq t$ and set $G_{k}(u)=(x_{k}^{-1}|u)^{r}, \ k \in \{1,2\}$, there exists a constant $C=C(r)$ such that,
\begin{equation}\label{sigrisi antistrofou esoterikon ginomenon}
|G_{1}-G_{2}|_{1-r+s}\leq C t^{r-2}\eta^{2r-1}|x_{1}-x_{2}|.
\end{equation}
\end{lem}
\begin{defn}\label{orismos tou Y}
Following Section 3.2 of \cite{bordenave2017delocalization}, for any numbers $h \in \bar{K},u \in S_{+}^{1}$ and $g \in H_{a/2}$ define the  functions,
\begin{multline}
F_{h,g}(u)=\int_{0}^{\pi/2}\int_{0}^{\infty}\int_{0}^{\infty}[\exp(-r^{a/2}g(e^{i \theta})-(rh|e^{i\theta}))-\exp(-r^{a/2}g(e^{i\theta}+uy)-(urh|u)-(rh|e^{i\theta}))]r^{a/2-1}dr \boldsymbol{\cdot} \\ \boldsymbol{\cdot} y^{-a/2-1}dy  (\sin(\theta))^{a/2-1} d\theta 
\end{multline}
and
\[Y_{f}(u)=Y_{z,f}(u)=c_{a}F_{-iz,f}(\tilde{u}),\]
where
\[c_{a}=\frac{a}{2^{a/2}\Gamma(a/2)^{2}}.\]
\end{defn}
\begin{lem}[\cite{bordenave2017delocalization},Lemma 4.1]
If $g \in H_{a/2,r}^{0}$ then $F_{\eta,g} \in H_{a/2,r}$. Also if $g \in \bar{H_{a/2,r}^{0}}$ and $\operatorname{Re}(h)>0$ then $F_{g,h} \in \bar{H_{a/2,r}^{0}}.$ 

\end{lem}
Next for any $f \in H_{a/2}$ and $p>0$ define the functions,

\begin{itemize}\label{orizontas ta terata}
  \item $r_{p,z}(f)=\frac{2^{1-p/2}}{\Gamma(p/2)}\int_{0}^{\pi/2}\int_{0}^{\infty}y^{p-1}\exp((iyz|e^{i\theta})-y^{a/2}f(e^{i\theta}))\sin(2\theta)^{p/2-1} dy d\theta$
 \item $s_{p,z}(x)=\frac{1}{\Gamma(p)}\int_{0}^{\infty} y^{p-1} \exp(-iyz-x y^{a/2}) dy.$
\end{itemize}
\begin{lem}[\cite{bordenave2017delocalization},Proposition 3.3]\label{Orismos kai idiotites omega}
There exists a countable subset $\mathcal{A} \subseteq (0,2)$ with no accumulation points such for any $r \in (0,1]$ and $a \in (0,2)\setminus \mathcal{A}$, there exists a constant $c=c(a,r)$ with the property that:
\\There exists a unique function $\Omega_{0} \in H_{a/2}$ such that $\Omega_{0}=Y_{0,\Omega_{0}}$. Additionally for $\operatorname{Im}(z)>0$ and $|z|\leq c$, there exists a unique function $f=\Omega_{z}\in H_{a/2,r}$ that solves $f=Y_{z,f}$ with $|f-\Omega_{0}|_{r}\leq c$. Moreover the function satisfies $\Omega_{z}(e^{i\pi/4})\geq c$ and for any $p>0$ there exists a constant $C=C(a,p) $ such that $|r_{p,z}(\Omega_{z})|\leq C$.
\end{lem}
\begin{lem}\label{Lemma 7.3}(\cite{bordenave2017delocalization},Proposition 3.4)
Adopt the notation of the previous lemma. After decreasing $c$ if necessary, there exists a constant $C>0$ such that the following holds. 
\\If $\operatorname{Im}(z)>0,|z|\leq c$ and $|f-\Omega_{z}|_{r}\leq c$, then
\[|f-\Omega_{z}|_{r}\leq C |f-Y_{z,f}|_{r}.\]
\end{lem}
\begin{lem}\label{F_h(g)}(\cite{bordenave2017delocalization},Lemma 4.1) Let $r\in (0,1)$ and $p>0$. There exists a constant $C=C(a,p,r)>0$ such that, for any $g \in \bar{H}_{a/2,r}^{0}$ and $h \in \mathbb{K}$, we have that
\[|F_{\eta}(g)|_{r}\leq C (\operatorname{Re} \eta)^{-a/2}+C |g|_{r} (\operatorname{Re}(\eta))^{-a/2}, \] 
\[|r_{p.i\eta}(g)| \leq C (\operatorname{Re}\eta)^{-p}, \ \ |s_{p,i\eta}(g(1))| \leq C ( \operatorname{Re}(\eta) )^{-p}.\]
\end{lem}
\begin{lem}\label{Y_f,Y_g}(\cite{bordenave2017delocalization}, Lemma 4.3)
For any $\alpha,r>0$ there exists a constant $C=C(\alpha,a,r)>0$ such that for any $f,g \in \bar{H_{a/2,r}^{0}},$ and $z \in \C$
\[|Y_{f}-Y_{g}|_{r}\leq C |f-g|_{r}+ |f-g|_{\infty}(|f|_{r}+|g|_{r}).\]
Furthermore, for any $p>0$ there exists a constant $C'=C'(a,\alpha,r,p)$ such that for any $f,g \in H_{a/2,r}^{\alpha}$ and for any $z \in \C$ and $x,y \in \mathbb{K}$ with $\operatorname{Re}(x),\operatorname{Re}(y)\geq \alpha$ we have that
\begin{align}\label{r_pz comparison}
|r_{p,z}(f)-r_{p,z}(g)|\leq C' |f-g|_{\infty}, \ \ |s_{p,z}(x)-s_{p,z}(y)| \leq C' |x-y|.
\end{align}
\end{lem}
The reason to present all the tools in this subsection is explained in the following Remark.
\begin{rem}
Due to Lemma 4.4. of \cite{bordenave2017delocalization}, $is_{1,z}(\Omega_{z}(1))$, which is defined in Lemma \ref{Orismos kai idiotites omega}, is exactly the limiting Stieltjes transform in Proposition \ref{orio stieltjes prot}.
\end{rem}
 \subsection{Statement of the intermediate local law}
In this subsection, we state the local law for the matrix $X$ and state a stronger theorem which will imply the local law. 
\\ Before we present the theorem, we give some definitions. Recall the notation from Subsection \ref{subsection gia periergous xorous}.
\begin{defn}\label{orismos theta kai gamma(z)}
Define the following quantities 
\[\iota_{z}(u):=\Gamma\left(1-\frac{a}{2}\right)(-iR_{j,j}|u)^{a/2} ,\ \ \ \gamma_{z}:=\E(\iota_{z}(u)),\]
\[I_{p}:=\E(-iR_{j,j})^{p}, \ \ \ J_{p}:=\E(|iR_{j,j}|^{p}),\]
where we have omitted the dependence from the dimension $N$ in the notation we used.
\end{defn} 
In what follows, keep in mind the definition of the functions $r_{p,z}$ and $s_{p,z}$ in Subsection \ref{subsection gia periergous xorous}. Next, we present the theorem which will imply the intermediate local law proved at Subsection \ref{subsection apodiksi aplou local law}. 
\begin{thm}\label{To theorima gia to local}
Let $a \in (0,2)$, $b \in (0,\frac{1}{a})$, $s \in (0,\frac{a}{2})$, $p>0$, $\epsilon \in (0,1]$ and $N \in \N$. Set $\theta=(\frac{1}{a}-b)(2-a)/10$. Suppose $z=E+i\eta \in \C^{+}$ with $E,\eta \in \R$. Assume that:
\begin{equation}\label{Ypothesi gia local law}
  z=E+i\eta, \ \ \ |z|\leq\frac{1}{\epsilon}, \ \ \ \eta \geq N^{\epsilon-s/a}, \ \ \ \E(\operatorname{Im}(R_{i,i})^{a/2})\geq \epsilon ,\ \ \ \E |R_{i,i}|^{2}\leq \epsilon^{-1},
  \text{ for all } i \in [2N].
\end{equation}

Then, there exists a constant $C=C(a,\epsilon,b,s,p)>0$ such that
\begin{align}
\label{fixed point equation}
&\big|\gamma_{z}-Y_{\gamma_{z}}\big|_{1-a/2+s}\leq C \log^{C}(N)\left(\frac{1}{(\eta^{2}N)^{a/8}}+\frac{1}{N^{\theta}}+\frac{1}{N^{s}\eta^{a/2}}\right),
\\&\label{fixed point Ropes}
  \left|I_{p}-s_{p,z}\left( \gamma_{z}\left(1\right)\right)\right|\leq C \log^{C}(N)\left(\frac{1}{(\eta^{2}N)^{a/8}}+\frac{1}{N^{\theta}}+\frac{1}{N^{s}\eta^{a/2}}\right),
  \\&\label{fixed point ropes v.2}
   |J_{p}-r_{p,z}(\gamma_{z})|\leq C \log^{C}(N)\left(\frac{1}{(\eta^{2}N)^{a/8}}+\frac{1}{N^{\theta}}+\frac{1}{N^{s}\eta^{a/2}}\right).
\end{align}
Moreover,
\begin{equation}\label{kato fragma sto theorima 7.8}
\inf_{u \in S^{1}_{+}} \operatorname{Re}(\gamma_{z}(u))>\frac{1}{C}
\end{equation}
and 
\begin{equation}\label{Kato Fragma ton Rjj mesa sto theorima 7.8}
\mathbb{P}\left(\max_{j \in [2N]}|R_{j,j}|>C \log^{C}(N)\right)\leq C \exp\left(-\frac{(\log(N))^{2}}{C}\right). \end{equation}
\end{thm}
The proof of Theorem \ref{To theorima gia to local} can be found in Subsection \ref{subsection apodiksi aplou local law}.

Next we present the local law.
\begin{thm}[Local law]\label{local law}
There exists a countable set $\mathcal{A} \subseteq (0,2)$ with no accumulation points in $(0,2)$ such that for each $a \in (0,2)\setminus \mathcal{A}$ the following holds. Fix $b \in (0,\frac{1}{a})$, $\theta=(\frac{1}{a}-b)(2-a)/10$ and $\delta \in (0,\min\{\theta,\frac{1}{2}\})$. Then there exists a constant $C=C(a,b,\delta,p)>0$ such that 
\begin{equation}\label{local law mesa sto theorima}
\mathbb{P}\left(\sup_{z \in D_{C,\delta}}\left|m_{N}(z)-is_{1,z}(\Omega_{z}(1))\right|>\frac{1}{N^{a\delta/8}}\right)\leq C \exp\left(-\frac{\log^{2}(N)}{C}\right).
\end{equation}
Furthermore,
\begin{equation}\label{Ropes kai meses times}
\sup_{u \in S_{+}^{1}}|\gamma_{z}(u)-\Omega_{z}(u)|\leq \frac{C}{N^{a\delta/8}}
\end{equation}
and
\begin{equation}\label{fragma gia Rii pano se ola ta z}
\mathbb{P}\left(\sup_{z \in D_{C,\delta}}\max_{j \in [2N]}|R_{j,j}|>C \log^{C}(N)\right)\leq C \exp\left(-\frac{(\log(N))^{2}}{C}\right).
\end{equation}
Where $D_{C,\delta}=\{z=E+i\eta: E \in (-\frac{1}{C},\frac{1}{C}), \frac{1}{C}\geq \eta \geq N^{\delta-1/2}\}$, $m_{N}(z)$ is the Stieltjes transform of $X$ and $\Omega_{z}(u)$ is defined in Lemma \ref{Orismos kai idiotites omega}.
\end{thm}
\begin{proof}[Proof Of Theorem \ref{local law} given Theorem \ref{To theorima gia to local}] 
The proof is similar to the proof of Theorem 7.6 given Theorem 7.8 and Lemmas \ref{Orismos kai idiotites omega}, \ref{Lemma 7.3} (there called Lemma 7.2 and Lemma 7.3) in \cite{aggarwal2021goe}, so it will be omitted.
\end{proof}

\subsection{General results concerning the resolvent and the eigenvalues of a matrix} 

Firstly we present a well-known result that compares the eigenvalues of a matrix with the eigenvalues of its minors.
\begin{lem}[Weyl's inequality]\label{Weyl's inequality} 
 Let $R,M,Q \in \R^{N^{2}}$ some symmetric matrices such that
 \[M=R+Q.\]
 Let $\mu_{i},\rho_{i},q_{i}$ be the eigenvalues of $M,R,Q$ respectively arranged in decreasing order. Then
 \[q_{j}+\rho_{k}\leq \mu_{i}\leq q_{r}+\rho_{s}\]
 for any indices such that $j+k-n\geq i \geq r+s-1$.
 
 \end{lem}
In the rest of this subsections we present some general results concerning the resolvent of a matrix. Most of them are known results, but we include them because they will be useful in the proof of Theorem \ref{local law}.
\begin{lem}\label{BASIKES ANISOTITES GIA RESOVLENT} 
Let $M_{1},M_{2}$ be two invertible, $N \times N$ matrices then the following identity is true
\begin{align}\label{resolventeq1}
    M_{1}^{-1}-M_{2}^{-1}=M_{1}^{-1}(M_{2}-M_{1})M_{2}^{-1}.\end{align}
Moreover if $Y=(M_{1}-z\mathbb{I})^{-1}$ such that $z \in \C^{+}=\{z \in \C: \operatorname{Im}(z)>0\}$ then
\begin{align}\label{resolventwq2}
    |Y_{i,j}|\leq \frac{1}{\operatorname{Im}(z)} , \ \ i,j \in [N].
    \end{align}
\end{lem}
\begin{proof}
    The identity \eqref{resolventeq1} follows trivially by a right multiplication on both sides by the element $M_1$ and a left multiplication on both sides by the element $M_2$.
    Moreover \eqref{resolventwq2} follows trivially from the spectral theorem.
\end{proof}
\begin{defn}\label{notation of the minor}In what follows in this section we will use the following notation. Consider $M$ to be any $N \times N$ matrix.
Let $J \subseteq [N]$. We will use the notation $(M^{(J)}-z\mathbb{I})^{-1}$ for the resolvent of the matrix $M^{(J)}$, where $M^{(J)}$ is the matrix $M$ with the $i-th$ row and column being replaced by zero vectors, for each $i \in J$.
\end{defn} 
\begin{lem}\label{taytotites gia resolvent}
Let $M$ be an $N\times N$ matrix and $z \in \C^{+}$.
 \\ Then we have the following complements formula
\begin{equation}\label{efarmogishcur}
\frac{1}{(M-z\mathbb{I})_{i,i}^{-1}}=M_{i,i}-z-\sum_{k,j \in [N]\setminus\{i\}}M_{i,j}(M^{(i)}-z\mathbb{I})^{-1}_{jk}M_{k,i} .\end{equation}
Next, we present the Ward identity. That is, for each $J \subseteq [N]$ and $j \in [N]\setminus J$  it is true that
\begin{equation}\label{wardident}\sum_{k \in [N]\setminus J}|(M^{(J)}-z\mathbb{I})^{-1}_{jk}|^{2}=\frac{\operatorname{Im} ((M^{(J)}-z\mathbb{I})^{-1}_{j,j})}{\operatorname{Im}(z)}.\end{equation}
\end{lem}
\begin{proof}
    The estimates \eqref{efarmogishcur} and 
\eqref{wardident}  can be found in (8.8) and (8.3) in \cite{erdos2017dynamical}, respectively. 
\end{proof}
\begin{lem}[\cite{bordenave2017delocalization},Lemma 5.5]
Let $M$ be an $N \times N$ matrix. For any $r \in (0,1], z\in \C^{+}, \text{ } \eta=\operatorname{Im}(z) $ and $i \in [N]$ we have the following deterministic bound
\[\frac{1}{N}\sum_{i=1}^{N}\left|(M-z\mathbb{I})_{i,i}^{-1}-(M^{(i)}-z\mathbb{I})^{-1}_{i,i}\right|^{r} \leq \frac{4}{(N\eta)^{r}}.\]
\end{lem}
\begin{cor}\label{Nteterministiko fragma ton minor resolvents}(\cite{aggarwal2021goe},Cor 5.7)
Let $M$ be an $N \times N$ matrix. For any $r \in [1,2], z\in \C^{+},\text{ 
}\eta=\operatorname{Im}(z) $ and $i\in [N]$ we have the deterministic estimate,
\[\frac{1}{N}\sum_{i=1}^{N}\left|(M-z\mathbb{I})_{i,i}^{-1}-(M^{(i)}-z\mathbb{I})^{-1}_{i,i}\right|^{r} \leq \frac{4}{(N\eta)^{r}}\leq \frac{8}{N \eta^{r}}.\]
 \end{cor}
\subsection{Concentration results for the resolvent of a matrix}
In this subsection, we present various identities and inequalities concerning the resolvent and the eigenvalues of a matrix.

Next we present some concentration inequalities.
\begin{lem}\label{Misa resovlent}Let $N$ be an even positive integer and let $A=(a_{i,j})_{1\leq i,j\leq N}$ such that the rows $A_{i}=(a_{i1},a_{i2},\cdots,a_{ii})$ are mutually independent for each $i \in [N]$. Let $B=(A-z\mathbb{I})^{-1}$ and $z=E+i\eta$ where $\eta>0$. Then for any Lipchitz function $f$ with Lipchitz norm $L_{f}$ and any $x>0$,we have that,
\begin{align}
&\mathbb{P}\left[\left|\frac{2}{N}\sum_{i=1}^{\frac{N}{2}}f(B_{i,i})-\E\frac{2}{N}\sum_{i=1}^{\frac{N}{2}}f(B_{i,i})\right|\geq x\right]\leq 2\exp\left(-\frac{N\eta^{2}x^{2}}{8L^{2}_{f}}\right), 
\\&\mathbb{P}\left[\left|\frac{2}{N}\sum_{i=1}^{\frac{N}{2}}f(B_{i+\frac{N}{2},i+\frac{N}{2}})-\E\frac{2}{N}\sum_{i=1}^{\frac{N}{2}}f(B_{i+\frac{N}{2},i+\frac{N}{2}})\right|\geq x\right]\leq 2\exp\left(-\frac{N\eta^{2}x^{2}}{8L_{f}^{2}}\right) .
\end{align}
\end{lem}
\begin{proof}
The proof is similar to the respective proof for the Stieltjes transform in Lemma C.4 of \cite{bordenave2013localization}. We will sketch the proof for the first $N 2^{-1}$ diagonal entries. The proof for the remaining $N 2^{-1}$ entries is similar. More precisely, 
for any two deterministic Hermitian matrices $C$ and $B$, let $R(C)$ and $R(B)$ be their resolvents at $z$. Then it is proven in equation (91) of Lemma C.4 of \cite{bordenave2013localization} that :
\begin{equation}
\frac{1}{N}\left|\sum_{k=1}^{N/2}R_{k,k}(B)-R_{k,k}(C)\right| \leq  \frac{1}{N}\sum_{k=1}^{N}\left|R_{k,k}(C)-R_{k,k}(B)\right|\leq  \operatorname{rank}(C-B)2 (\operatorname{Im}(z)N)^{-1}.
\end{equation}
So if one considers the function $$F(\{x_{i}\}_{i=1}^{N})= \sum_{k=1}^{N/2}\frac{f(R_{k,k}(X))}{N}, \ \ \ \ \ \ \{x_{i}\}_{i=1}^{N}: x_{i} \in \C^{i-1} \times \R ,$$
 where $X$ is a Hermitian matrix with the i-th row of $X$ being $x_{i}$. Note that it suffices to describe the entries of the i-th row until the i-th column since the remaining elements will be filled by the properties of the Hermitian matrices. So if we consider two elements $X,X' \in \cup_{i=1}^{N} \C^{i-1}\times \R$ with only the i-th vector of $X$ and $X'$ different, then one has:
\[|F(X)-F(X')|\leq \operatorname{rank}(X-X') 2 (\operatorname{Im}(z)N)^{-1}\leq 4 (\operatorname{Im}(z)N)^{-1},\]
since by construction, one has that $\operatorname{rank}(X-X') \leq 2$. Now the desired inequality comes from Azuma–Hoeffding inequality, see Lemma 1.2 in \cite{mcdiarmid1989method}.
\end{proof}
\begin{cor}\label{Anisotites sigentrosis} 
One can apply the previous Lemma to get the following concentration results. Fix an $N \times N$ symmetric random matrix $Y$ with i.i.d. entries (up to symmetry) , where $N$ is an even integer, with resolvent matrix $B=(Y-z\mathbb{I})^{-1}$ for $z=E+i\eta$. Then the following bounds are true: 
\begin{equation}
\begin{aligned}
\mathbb{P}\left[\frac{2}{N}\left|\sum_{k=1}^{N/2}B_{k,k}-\E B_{k,k}\right|\geq \frac{4 \log(N)}{(N \eta^{2})^{1/2} } \right]\leq 2 \exp\left(-(\log(N))^{2}\right),
\\ \label{concineqhalfim} \mathbb{P}\left[\frac{2}{N}\left|\sum_{k=1}^{N/2}\operatorname{Im}(B_{k,k})-\E \operatorname{Im}(B_{k,k})\right|\geq \frac{4 \log(N)}{(N \eta^{2})^{1/2} } \right]\leq 2 \exp(-(\log(N))^{2}).
 \end{aligned}
\end{equation}
Moreover for any $a \in (0,2)$ there exists a constant $C=C(a)$ such that,
\begin{equation}\label{concentration  (-ix)^{a/2}}
  \mathbb{P}\left[\frac{2}{N}\left|\sum_{k=1}^{N/2}(-i B_{k,k})^{a/2}-\E  (-i B_{k,k})^{a/2}\right|\geq x \right]\leq 2 \exp\left(-\frac{N (\eta^{a/2}x)^{4/a}}{C}\right).
\end{equation}

The same results hold for the remaining $N/2$ diagonal entries of $R$.

\end{cor}
\begin{proof}
The first two inequalities are true by direct application of Lemma \ref{Misa resovlent} for the functions $f(x)=x$ and $f(x)=\operatorname{Im}(x)$ respectively.

For the third inequality let $c>0$ and fix $\phi_{c}:\C\to \R^{+}$, such that
\begin{align}
   \phi_{c}(z)=\begin{cases}
       0 & |z|\leq c,
       \\ \frac{1}{c}(|z|-c)& |z| \in (c,2c),
       \\ 1 & |z|\geq 2c.
   \end{cases} 
\end{align}
Note that the function $\phi_{c}$ is Lipschitz with  Lipschitz constant bounded by $\frac{1}{c}$. Then define the function
\begin{align}\label{f_c(z) Lipsc}
    f_{c}(z)=(-iz)^{a/2}\phi_{c}(z).
\end{align}
Since $|(1-\phi_{c}(z))(-iz)^{a/2}|\leq (2c)^{a/2}$ for all $z\in \C^{+}$, it is clear that $|(-iz)^{a/2}|\leq f_{c}(z)+ (2c)^{a/2}$. Moreover note that the function $f_{c}(z)$ is Lipschitz with constant bounded by $2c^{\frac{a}{2}-1}$. 

So for any $x\geq 0$ fix $c$ such that $(2c)^{a/2}= x/4.$ Then
 \begin{align}
 &\mathbb{P}\left[\frac{2}{N}\left|\sum_{k=1}^{N/2}(-i B_{k,k})^{a/2}-\E (- (-i B_{k,k})^{a/2}\right|\geq x \right] 
 \\\label{upboundlips()^a/2}&\leq \mathbb{P}\left[\frac{2}{N}\left|\sum_{k=1}^{N/2} f_{c}\left( B_{k,k}\right)-\E f_{c}\left( B_{k,k}\right)\right| \geq \frac{x}{2} \right]. 
 \end{align}
 Now the proof is completed after a direct application of Lemma \ref{Misa resovlent}, after noticing that $2c^{\frac{a}{2}-1}=2^{\frac{4}{a}-\frac{a}{2}}x^{1-\frac{2}{a}}$.
\end{proof}
The following result is an analogue of Lemma 5.3 in \cite{bordenave2017delocalization} for concentration of only half of the resolvent diagonal entries. The proof is analogous.
\begin{lem}\label{Anisotita sigkentrosis se periergous xorous}
Let $N$ be an even and positive integer, $A=\{a_{i,j}\}_{1\leq i,j\leq N}$ a symmetric matrix with independent entries (up to symmetry). Fix $u \in S_{+}^{1},a \in (0,2)$ and $s \in (0,\frac{a}{2}) $. Moreover define the resolvent matrix $B=(A-z_0\mathbb{I})^{-1}$ for $z_0=E+i\eta \in \C^+$.
\\Then if we denote $f_{u}:\C \rightarrow \C$ such that $f_{u}(z)=(iz|u)^{a/2}$, there exists constant $C=C(a)>0$ such that 
\[\mathbb{P}\left[\left|\frac{2}{N}\sum_{i=1}^{\frac{N}{2}}f_u(B_{i,i})-\E\frac{2}{N}\sum_{k=1}^{\frac{N}{2}}f_u(B_{k,k})\right|_{1-a/2+s}
\geq x\right]\leq C(\eta^{a/2}x)^{-1/s}\exp \left(-\frac{
N(x \eta^{\frac{a}{2}}
)^
{\frac{2}{s
}
}}{C}\right).\]
A similar estimate is true for the concentration of the second half of the diagonal entries of the resolvent.
\end{lem}
\begin{proof}
By definition of the norms in Definition \ref{perierges normes} we need to bound the following quantities  
\begin{align}\label{sigentrosi se periergous xorous 1}
&\mathbb{P}\left[\sup_{u\in S_{+}^{1}}\left|\frac{2}{N}\sum_{k=1}^{\frac{N}{2}}f_u(B_{k,k})-\E\frac{2}{N}\sum_{k=1}^{\frac{N}{2}}f_u(B_{k,k})\right|\geq x\right] \text{ for any } x>0,
    \\\label{concentration for derivative}&\mathbb{P}\left[\sup_{u\in S_{+}^{1}}\max_{j\in \{1,2\}}\left|(i|u)^{1-\frac{a}{2}+s}\partial_{j}\left(\frac{2}{N}\sum_{k=1}^{\frac{N}{2}}f_u(B_{k,k})-\E\frac{2}{N}\sum_{k=1}^{\frac{N}{2}}f_u(B_{k,k})\right)\right| \geq x \right] \text{ for any } x>0.
\end{align}
Fix $u\in S_{+}^{1}$ and $c>0$. Then, similarly to the proof of \eqref{concentration  (-ix)^{a/2}} in Corollary \ref{Anisotites sigentrosis}, we can construct a function $\phi_{c} :\C\to \R^+$, which is $\frac{1}{c}-$ Lipchitz  function and for which it is true that if we decompose $f_u$ in the following sense,
\begin{align}\label{f_u(z)}
    f_{u}(z)= \phi_{c} f_{u}(z) + (1-\phi_{c}) f_{u}(z)= f_{1,u}(z)+ f_{2,u}, 
\end{align}
then $f_{2,u}(z)$ is bounded by $(2c)^{1-\frac{a}{2}+s}$ and $f_{1,u}(z)$ is  Lipschitz with constant bounded by $c' c^{s-\frac{a}{2}}$, for some other absolute constant $c'$. So for any $x>0$, if one fixes $c$ such that $(2c)^{1-\frac{a}{2}+s}=\frac{x}{4}$ it is implied that
\begin{align}
&\mathbb{P}\left[\left|\frac{2}{N}\sum_{k=1}^{\frac{N}{2}}f_u(B_{k,k})-\E\frac{2}{N}\sum_{k=1}^{\frac{N}{2}}f_u(B_{k,k})\right|\geq x\right] 
  \leq \mathbb{P}\left[\left|\frac{2}{N}\sum_{k=1}^{\frac{N}{2}}f_{1,u}(B_{k,k})-\E\frac{2}{N}\sum_{k=1}^{\frac{N}{2}}f_{1,u}(B_{k,k})\right|\geq \frac{x}{2}\right].     
\end{align}
So by a direct application of Lemma \ref{Misa resovlent} for the function $f_{1,u}$ one can conclude that
\begin{align}\label{sigentrosi se periergous xorous 1, gia ena u}
\mathbb{P}\left[\left|\frac{2}{N}\sum_{k=1}^{\frac{N}{2}}f_u(B_{k,k})-\E\frac{2}{N}\sum_{k=1}^{\frac{N}{2}}f_u(B_{k,k})\right|\geq x\right] \leq \exp \left(-\frac{
N(x \eta^{\frac{a}{2}}
)^
{\frac{2}{s
}
}}{C}\right),   
\end{align}
for some constant $C=C(a)$. Moreover due to the deterministic bounds in \eqref{anisotita (u|v)<|u||v|} and \eqref{resolventwq2}, we can restrict to the case that $x\leq 4 \eta^{-\frac{a}{2}}$. 
 Furthermore, by (4.6) in \cite{vershynin2018high} for any $c\in(0,1)$, any $c-$net of the sphere $S_{+}^{1}$ has cardinality at most $\frac{3}{c}$. Set $\mathcal{F}$ one $c-$net of the sphere. So for any $x\in (0,4\eta^{-\frac{a}{2}})$, fix $c$ such that $(2^{\frac{1}{2}}c\eta^{-1})^{1-\frac{a}{2}+s}=\frac{x}{4}$. Thus 
 by \eqref{sigrisi esoterikon ginomenon}, we conclude that
\begin{align}\label{c-net aplo}
\mathbb{P}\left[\sup_{u\in S_{+}^{1}}\left|\frac{2}{N}\sum_{i=1}^{\frac{N}{2}}f_u(B_{k,k})-\E\frac{2}{N}\sum_{k=1}^{\frac{N}{2}}f_u(B_{k,k})\right|\geq x\right] \leq     \mathbb{P}\left[\sup_{u\in \mathcal{F}}\left|\frac{2}{N}\sum_{k=1}^{\frac{N}{2}}f_u(B_{k,k})-\E\frac{2}{N}\sum_{k=1}^{\frac{N}{2}}f_u(B_{k,k})\right|\geq \frac{x}{2}\right].
\end{align}
So we get \eqref{sigentrosi se periergous xorous 1}  after using the union bound and \eqref{sigentrosi se periergous xorous 1, gia ena u} to bound \eqref{c-net aplo}.

It remains to prove \eqref{concentration for derivative}. For this note that
\begin{align}
    \frac{2}{N}\partial_{j}\sum_{k=1}^{\frac{N}{2}}f_{u}(B_{k,k})=\frac{2}{N}\sum_{k=1}^{\frac{N}{2}}(1-\frac{a}{2}+s)(iB_{k,k}|u)^{\frac{a}{2}-1}(iB_{k,k}|j).
\end{align}
So we can treat the function $g_u(z)=(iz|u)(iz|j)$ analogously $f_u(z)$ in \eqref{f_u(z)}. In particular we have the following decomposition
\begin{align}
    g_{u}(z)= \phi_{c} g_{u}(z) + (1-\phi_{c}) g_{u}(z)= g_{1,u}(z)+ g_{2,u},
\end{align}
where $g_{1,u}$ is Lipschitz with constant bounded by $c_{0}c^{s-\frac{a}{2}}/|ui|$ and $g_{2,u}$ is bounded by $c_{0}c^{1+s-\frac{a}{2}}/|ui|$, for some absolute constant $c_0$. So for any $x>0$ let $c$ be a number such that $c_0 c^{1+s-\frac{a}{2}}=x/4$, we get that
\begin{align}
&\mathbb{P}\left[\left|(i|u)^{1-\frac{a}{2}+s}\partial_{j}\left(\frac{2}{N}\sum_{k=1}^{\frac{N}{2}}f_u(B_{k,k})-\E\frac{2}{N}\sum_{k=1}^{\frac{N}{2}}f_u(B_{k,k})\right)\right| \geq x \right] \\&\label{g_2,u}\leq \mathbb{P}\left[|(i|u)|^{1-\frac{a}{2}+s}\left|\left(\frac{2}{N}\sum_{k=1}^{\frac{N}{2}}g_{1,u}(B_{k,k})-\E\frac{2}{N}\sum_{k=1}^{\frac{N}{2}}g_{1,u}(B_{k,k})\right)\right| \geq x/2 \right].
\end{align}
By a direct application of Lemma \ref{Misa resovlent}, we can bound \eqref{g_2,u}.
The last part of the proof is completed by a $c-$net argument, completely analogously to \eqref{c-net aplo}.  
\end{proof}
\begin{lem}[\cite{bordenave2017delocalization},Lemma 5.4]\label{gausianes kai resovlent} 
Let $(y_{1},y_{2},\cdots,y_{N})$ be a Gaussian random vector whose covariance matrix is the $Id$. Fix $a \in (0,2)$, $s \in (0,a/2)$. Moreover, let $\{h_{k}\}_{k\in [N]} \in (\C^{+})^{N}$ such that $|h_{k}|\leq \eta^{-1}$, for some $\eta>0$. Then for each $j \in \N$ define the following quantities  
\[f_{j}(u)=\left(h_{j}|u\right)^{a/2}|y_{j}|^{a}, \ \ g_{j}(u)=\left( h_{j}|u\right)^{a/2}\E|y_{j}|^{a}.\]
Then there exists a constant $C=C(a)$ such that
\[\mathbb{P}\left[\left|\frac{1}{N}\left(\sum_{j=1}^{N}f_{j+N}-g_{j+N}\right)\right|_{1-a/2+s}\geq x\right]\leq C (\eta^{a/2}x)^{-1/s}\exp\left(-\frac{N(\eta^{a/2}x)^{2/s}}{C}\right).\]
\end{lem}
\begin{rem}
    Due to the deterministic bound \eqref{resolventwq2}, we can apply Lemma \ref{gausianes kai resovlent} for any number of the diagonal entries of the resolvent of a matrix. 
\end{rem}
\subsection{Gaussian and stable random variables}
In this subsection we present several results concerning Gaussian random variables and their interaction with the quantities we study.
\begin{lem}\label{gausianes me pinaka}(\cite{aggarwal2021goe},Lemma 6.4) Let $N \in \N$ and $x$ be a $b-$removal of a $(0,\sigma)$ a-stable distribution, as is defined in Definition \ref{orismos b-removals}. Then let $\hat{X}$ be an $N-$dimensional vector with independent entries all with law $N^{-1/a}x$. Then for any $u \in \R$ and for $A$ a non-negative symmetric matrix and $Y$ an $N-$dimensional centered 
Gaussian vector with covariance matrix the $Id$ it is true that, 
\begin{align}
\E \left[\exp\left(-\frac{u^{2}}{2}\langle A\hat{X},\hat{X}\rangle\right)\right]=\E \exp\left(\frac{-\sigma^{a}|u|^{a}\|A^{1/2}Y\|_{a}^{a}}{N}\right)\exp\left(\mathcal{O}(u^{2}N^{(2-a)(b-1/a)-1}\log(N)\tr(A))\right)
\\ +N \exp\left(-\frac{\log^{2}(N)}{2}\right).
\end{align}
\end{lem}
\begin{lem}\label{Gaussianes kaieystathis gia mh-diagonious cov pinakes}(\cite{aggarwal2021goe},Lemma 6.5)
Let $N$ be a positive integer and let $r,d$ be positive real numbers such that $0<r<2<d \leq 4$. Denote $w=(w_{1},w_{2}\cdots,w_{N})$ to be a centered $N-$dimensional Gaussian random variable with covariance matrix $U_{i,j}=\E (w_{i}w_{j})$ for $i,j \in [N]$. Denote $V_{j}=\E (w_{j,j}^{2})$ for each $j \in [N]$ and define
\[U=\frac{1}{N^{2}} \sum_{i,j \in [N]}U^{2}_{i,j}, \ \ V=\frac{1}{N}\sum_{j=1}^{N}V_{j}, \ \ \mathop{X}=\sum_{i=1}^{N}V_{j}^{d/2}, \ \ p=\frac{d-r}{d-2}, \ \ q=\frac{d-r}{2-r} .\]
Then if $V >100 \log^{10}(N)U^{1/2}$ there exists a constant $C=C(\alpha,r)$:
\[\mathbb{P}\left( \frac{|w|_{r}^{r}}{N}<\frac{V^{p}}{C ((\mathop{X}(\log(N))^{8})^{p/q}}\right)\leq C \exp\left(-\frac{(\log(N))^{2}}{2}\right).\]
\end{lem}
\subsection{Bounds for the resolvent of X.}
Recall the notation $R$ for the resolvent of $X$ and let $X^{(i)}$ is the matrix $X$ with its $i-$th row and column replaced by $0$ vector, as in Definition \ref{notation of the minor}. 

In what follows we will use the following notation $R^{(i)}=(X^{(i)}-z\mathbb{I})^{-1}$ and
\begin{align}\label{orismos T_i}
S_{i}(z)=\sum_{j \in [2N]\setminus{i}}X_{i,j}^{2}R^{(i)}_{j,j}(z)  \text{ and }  T_{i}(z)=X_{i,i}-U_{i}(z)  \text{ where }  U_{i}(z)=\sum_{j \in 2N]\setminus\{i\}} \sum_{k \in  [2N]\setminus\{i,j\}}X_{ij}R_{jk}^{(i)}(z)X_{k,i}, \ \ \ i \in [2N].
\end{align}
For notational convenience, we will omit the dependence of $S_{i}(z),T_{i}(z)$ and $U_{i}(z)$ from $z$ and $N$, the dimension of the matrix. By the resolvent equality in Lemma \ref{taytotites gia resolvent} one has that 
\begin{equation}\label{Taytotita gia Rii}
  R_{i,i}=\frac{1}{T_{i}-z-S_{i}}.
\end{equation}
Moreover for each $i \in [2N]$, one has that $\operatorname{Im}(R^{(i)})$ is positive definite, since it is symmetric and by the spectral theorem its eigenvalues are 
\begin{align}
    \frac{\eta}{\left(\lambda_{j}(X^{(i)})+E\right)^2+\eta^2}>0 \text{ , } j\in [2N],
\end{align}
where $\lambda_{j}(X^{(i)})$ are the eigenvalues of $X^{(i)}$.
 So it is true that 
\begin{align}\label{Im(S_i)>Im(T_{i})}
\operatorname{Im}(S_{i})\geq 0 \ \ \text{and} \ \ \operatorname{Im}(S_{i}-T_{i})\geq 0.
\end{align}
In addition, the diagonal entries of the resolvent $R_{i,i}$ are identically distributed. This is proven in the following Lemma.
\begin{lem}\label{iid R_i, Schur}
 The random variables $R_{i,i}$, for each $i \in [2N]$, are identically distributed. 
     \begin{proof}
Note that due to Schur's complement formula it is true that for any $N \times N$ matrices $ A,B,C,D $ , if $A,D$ are invertible then
\[\begin{bmatrix}A & B \\ C & D \end{bmatrix}^{-1}= \begin{bmatrix} (A-BD^{-1}C)^{-1} & * \\ * & * \end{bmatrix}=\begin{bmatrix} * & * \\ * & (D-CA^{-1}B)^{-1} \end{bmatrix} .\]

So if one sets $A=D=-z\mathbb{I}$, $C=K$ and $B=K^{T}$ it is true that $ R_{i,i}= z (K^{T}K-z^{2}\mathbb{I})^{-1}_{i,i} $ for $i \in [N]$ and $R_{i,i}= z (KK^{T}-z^{2}\mathbb{I})^{-1}_{i,i}$ for $i \in [2N]\setminus [N]$. Thus we can conclude that for each $i \in [N]$ the diagonal term $R_{i,i}$ has the same law as $R_{i+N,i+N}$. Moreover, for $i,j \in [N]$ or $i,j \in [2N]\setminus[N] $ it is easy to see that the matrix $X$ retains its law after the permutation of $i-th$ column and row to the $j-th$. All these imply that the diagonal terms $R_{i,i}$ have the same law for each $i \in [2N]$. 
\end{proof}
\end{lem}
Note that the Lemma above would not be true if the dimensions of the matrix, whose symmetrization is $X$, were not equal.

Moreover since the matrix $X$ has 0 at its diagonal blocks, one may compute that
\begin{equation}\label{Orismos S1-T1}S_{1}=\sum_{j=1}^{N}X_{1,N+j}^{2}R^{(i)}_{N+j,N+j}, \ \ \ T_{1}= -\sum_{j,k \in [2N]\setminus[N]:j \neq k}X_{1,j}X_{k,1}R_{j,k}^{(1)}.\end{equation}

Keep in mind that we want to prove Theorem \ref{To theorima gia to local}, so in what follows in this section we will operate under the assumption that \eqref{Ypothesi gia local law} holds.

The following is the analogue of Proposition 7.9 in \cite{aggarwal2021goe}, adjusted to our set of matrices.
\begin{prop}\label{Fragma Si}
For each $i \in [2N]$ there exists a constant $C=C(a,\epsilon,b)>1$ such that
\[\mathbb{P}\left(\operatorname{Im}(S_{i})<\frac{1}{C(\log(N))^{C}}\right)\leq C \exp\left(-\frac{(\log(N))^{2}}{C}\right).\]
\end{prop}
\begin{proof}
We will prove the estimate for $S_{1}$, since $R_{i,i}$ are identically distributed for $i \in [2N]$ due to Lemma \ref{iid R_i, Schur}.
\\Set the event:
\[\mathcal{E}=\left\{\left|\frac{1}{N}\sum_{i=1}^{N}R_{N+j,N+j}^{(1)}-\frac{1}{2}\E R_{2N,2N}\right|\leq \frac{8 \log(N)}{(N \eta^{2})^{1/2}}+\frac{16}{N\eta}\right\}.\]
By Corollary \ref{Anisotites sigentrosis} and Lemma \ref{Nteterministiko fragma ton minor resolvents}, one has that $\mathbb{P}(\mathcal{E}^{c})\leq 2 \exp(-(\log(N))^{2})$.

Observe that $\operatorname{Im}(S_{1})=\langle \mathop{A} \tilde{X},\tilde{X} \rangle$, where $\mathop{A}$ is an $N-$dimensional diagonal matrix with entries $\mathop{A}_{j,j}=\operatorname{Im}(R^{(1)}_{N+j,N+j})$ and $\tilde{X}$ is an $N-$dimensional vector with entries $\tilde{X_{j}}=X_{1,N+j}$. So we can apply Markov inequality for $u=(\log(N)^{2/a}(2\log(2))^{1/2}$ to get that:
\[\mathbb{P}(\operatorname{Im}(S_{1})<\mathbf{1}(\mathcal{E})\log(N)^{-4/a})\leq 2 \E (\mathbf{1}(\mathcal{E})\exp\left(-\frac{u^{2}}{2}\langle \mathop{A} \tilde{X},\tilde{X}\rangle\right) .\]

Next, we can apply Lemma \ref{gausianes me pinaka} and after bounding $\tr(A)$ by $C= C'(a,b,\epsilon)$, which we can do since we work on the set $\mathcal{E}$ and since it is true that $\E \operatorname{Im} (R_{1,1}) \leq (\E  |R_{1,1}|^2)^{\frac{1}{2}} \leq \epsilon^{-\frac{1}{2}}$ due to our assumption that $\E |R_{1,1}|^2\leq \epsilon^{-1}$ in  \eqref{Ypothesi gia local law}. We conclude that
\[\mathbb{P}\left(\operatorname{Im}(S_{1})\leq \frac{1}{\log(N)^{4/a}}\right)\leq C'\E \exp\left(-\frac{-\log^{2}(N)\lVert|A^{1/2}Y\rVert|_{a}^{a}}{C'N}\right) + C'\exp\left(-\frac{\log(N)^{2}}{C'}\right), \]
where $Y$ is a Gaussian vector with covariance matrix the identical, as is mentioned in Lemma \ref{gausianes me pinaka}.
Thus it remains to prove a lower bound for
\begin{align}\label{ineqAY}
\frac{\lVert|\mathop{A}^{1/2}Y\rVert|_{a}^{a}}{N}=\frac{1}{N}\sum_{j=1}^{N}|\operatorname{Im}(R_{N+j,N+j}^{(1)})|^{a/2}|y_{j}|^{a}.
\end{align}
Note that for $s \in (0,\frac{a}{2})$ and by Remark \ref{sigrisi sup. norm gia periergous xorous} 
\begin{align}
&\left|\frac{1}{N}\sum_{j=1}^{N}|\operatorname{Im}(R_{N+j,N+j}^{(1)})|^{a/2}|y_{j}|^{a}-\frac{1}{N}\sum_{j=1}^{N}|\operatorname{Im}(R_{N+j,N+j}^{(1)})|^{a/2}\E |y_{j}|^{a}\right| \\& \leq \sup_{u \in S_{+}^{1}}\left| \frac{1}{N}\sum_{j=1}^{N}|(-iR_{N+j,N+j}^{(1)}|u)|^{a/2}|y_{j}|^{a}-\frac{1}{N}\sum_{j=1}^{N}|(-i R_{N+j,N+j}^{(1)}|u)|^{a/2}\E |y_{j}|^{a}\right| \\& \leq \left|\frac{1}{N}\sum_{j=1}^{N}|(-iR_{N+j,N+j}^{(1)}|u)|^{a/2}|y_{j}|^{a}-\frac{1}{N}\sum_{j=1}^{N}|(-i R_{N+j,N+j}^{(1)}|u)|^{a/2}\E |y_{j}|^{a}\right|_{1-a/2+s}  
\end{align}
So we can apply Lemma \ref{gausianes kai resovlent} for $$x=C \frac{\log^{4} N}{N^{a/4}\eta^{a/2}},$$ and $s=\frac{a}{4}$ to get that the inequality
\begin{align}
    \left|\frac{1}{N}\sum_{j=1}^{N}|\operatorname{Im}(R_{N+j,N+j}^{(1)})|^{a/2}|y_{j}|^{a}-\frac{1}{N}\sum_{j=1}^{N}|\operatorname{Im}(R_{N+j,N+j}^{(1)})|^{a/2}\E |y_{j}|^{a}\right| \leq  x
\end{align}
  holds with probability at least $1-C \exp(-\frac{\log^{2}N}{C})$. Thus, it is sufficient to give a lower bound to  
  \begin{align}
      \frac{1}{N}\sum_{j=1}^{N}|\operatorname{Im}(R_{N+j,N+j}^{(1)})|^{a/2}\E |y_{j}|^{a}.
  \end{align}
in order to obtain a lower bound for the quantity in \eqref{ineqAY}.

Next we apply again Lemma \ref{Nteterministiko fragma ton minor resolvents} for $r=\frac{a}{2}$, and since for any $u_1,u_2 \in \R^{+}$ and $r \in (0,1]$ it is true that $|u_1^{r}-u_2^{r}|\leq |u_1-u_2|^{r}$, we obtain that
\[\E|y_{1}|^{a}\frac{1}{N}\sum_{j=1}^{N}|\operatorname{Im} (R_{N+J,N+J})^{a/2}-\operatorname{Im}(R^{(1)}_{N+j,N+j})^{a/2}|\leq \frac{4\E|y_{1}|^{a}}{(\eta N)^{a/2}}.\]
So we have concluded that the event that
\[\frac{\lVert|\mathop{A}^{1/2}Y\rVert|_{a}^{a}}{N}\geq \mathcal{O}\left(\frac{1}{(\eta N)^{a/2}}\right)+\mathcal{O}\left(\frac{1}{(\eta N^{1/2})^{a/2}}\right)+ C''\frac{1}{N}\sum_{j=1}^{N}\operatorname{Im}(R_{N+j})^{a/2}, \]
holds with probability at least $1-C \exp\left(-\frac{\log^{2}(N)}{N}\right)$. Restricting again on the set $\mathcal{E}$ and using the concentration inequality \eqref{concineqhalfim} and our hypothesis \ref{Ypothesi gia local law}, one can conclude that there exists $C=C(a,\epsilon,b)$ such that
\[\mathbb{P}\left(\frac{\lVert\mathop{A}^{1/2}Y\rVert|_{a}^{a}}{C N}\leq \epsilon\right)\leq C \exp\left(-\frac{-\log^{2}(N)}{C}\right),\]
which finishes the proof.
\end{proof}
The following is the analogue of Proposition 7.10 in \cite{aggarwal2021goe}, adjusted to our set of matrices.
\begin{prop}\label{Fragma Si-Ti}
For each $i \in [2N]$ there exists a constant $C=C(a,\epsilon,b)>1$ such that
\begin{equation}\label{fragma Si-Ti mesa sto lhmma}
  \mathbb{P}\left(\operatorname{Im}(S_{i}-T_{i})<\frac{1}{C(\log(N))^{C}}\right)\leq C \exp\left(-\frac{(\log(N))^{2}}{C}\right).\end{equation}
Moreover 
\begin{equation}\label{Kato Fragma ton Rjj}
\mathbb{P}\left(\max_{j \in [2N]}|R_{j,j}|>C \log^{C}(N)\right)\leq C \exp\left(-\frac{(\log(N))^{2}}{C}\right). \end{equation}
\end{prop} 
\begin{proof}
By construction, one can prove that for $A=\{\operatorname{Im}(R^{(1)}_{i,j})\}_{i,j \in [2N]\setminus[N]}$ and $\tilde{X}=\{X_{1,N+j}\}_{j \in [N]}$ it is true that,
\[\operatorname{Im}(S_{1}-T_{1})=\langle A\tilde{X},\tilde{X}\rangle.\]
So after applying Lemma \ref{gausianes me pinaka}, like in Proposition \ref{Fragma Si}, one has that
\[\mathbb{P}\left(\operatorname{Im}(S_{1}-T_{i})<\frac{1}{\log^{4/a}(N)}\right)\leq C \E\exp\left(-\frac{C\log^{2}(N)\|A^{1/2}Y\|_{a}^{a}}{N}\right)+C \exp\left(-\frac{\log^{2}(N)}{C}\right), \]
where $Y$ is again a centered $N-$dimensional Gaussian random variable with covariance matrix equal to the identical.
\\Next, we want to apply Lemma \ref{Gaussianes kaieystathis gia mh-diagonious cov pinakes} in order to establish a lower bound for $\frac{1}{N}\|A^{1/2}Y\|_{a}^{a}$. Following the notation of Lemma \ref{Gaussianes kaieystathis gia mh-diagonious cov pinakes} set
\[w_{i}=(A^{1/2}Y)_{i}, \ \ V_{j}=\operatorname{Im}(R_{j,j}^{(1)}), \ \ U_{j,k}=\operatorname{Im} R_{j,k}^{(1)}(z), \]

\begin{align}\label{X'}
X'=\frac{1}{N}\sum_{i=1}^{N}V_{N+j,N+j}^{a/2}, \ \ \ U=\frac{1}{N^{2}}\sum_{i,j \in [2N]\setminus[N]U_{i,j}}U_{i,j}, \ \ \ r=a, \ \ d=2+\epsilon .\end{align}
So one may apply Lemma \ref{BASIKES ANISOTITES GIA RESOVLENT} and Lemma \ref{taytotites gia resolvent} to get that
\begin{equation}\label{Fragma gia to U prop 7.10}
U\leq \frac{4}{N^{2}} \sum_{i,j \in [2N]}U_{i,j}^{2} \leq \frac{4}{N^{2}} \sum_{i,j \in [2N]}|\operatorname{Im}(R_{ij}^{(1)})|^2\leq \frac{4}{N^{2}\eta}\sum_{j=1}^{2N}\operatorname{Im}(R_{j,j}^{(1)})\leq \frac{4}{N \eta^{2}}.
\end{equation}

Next, we can approximate $V=\frac{1}{N} \sum_{j=1}^{N}V_{N+j}$ by $\frac{1}{N}\sum_{j=1}^{N}\operatorname{Im}(R_{N+j,N+j})$ due to the deterministic bound in Lemma \ref{Nteterministiko fragma ton minor resolvents} and then approximate $\frac{1}{N}\sum_{j=1}^{N}\operatorname{Im}(R_{N+j,N+j})$ by $\frac{1}{2}\E \operatorname{Im}(R_{1,1})$ due to Corollary \ref{BASIKES ANISOTITES GIA RESOVLENT}, on an event which holds with probability at least $1-2\exp\left(-\frac{\log^2 (N)}{8}\right)$. The approximation procedure described above is identical to the similar approximation described in Proposition \ref{Fragma Si}. So after taking into account the Hypothesis \eqref{Ypothesi gia local law}, we have that
\[\E \operatorname{Im}(R_{1,1})\geq \left(\E [\operatorname{Im} (R_{1,1})]^{a/2}\right)^{2/a}\geq \epsilon^{2/a}.\]
Thus, it is implied that 
\begin{equation}\label{fragma gia V prop 7.10}
\mathbb{P}\left(\frac{|V|}{C}<1\right) < C \exp\left(-\frac{\log^{2}N}{C}\right).
\end{equation}
So after combining \eqref{Fragma gia to U prop 7.10} and \eqref{fragma gia V prop 7.10}, we get that for sufficient large $N$ it is true that
\[\mathbb{P}\left(|V|\leq 100 \log^{10}(N)U^{1/2}\right)\leq C' \exp\left(-\frac{\log^{2}(N)}{C}\right).\]
Next we need to bound $\mathop{X}'$ from \eqref{X'}. Note that again we can apply Lemma \ref{Nteterministiko fragma ton minor resolvents} to get that
\begin{equation}\label{Prosegisi X apo stieltjes prop 7.10}
\begin{aligned}
 \left|\frac{1}{N}\sum_{j=1}^{N}\operatorname{Im}(R_{j+N,j+N})-\mathop{X'}\right|\leq 
\frac{1}{N}\sum_{j=1}^{N}|\operatorname{Im}(R_{j+N,j+N})^{a/2}-\operatorname{Im}(R^{(1)}_{N+j,N+j})^{a/2}|\leq \\ \leq \frac{1}{N}\sum_{j=1}^{2N}|\operatorname{Im}(R_{j,j})-\operatorname{Im} (R^{(1)}_{j,j})|^{a/2}\leq \sum_{j=1}^{2N}|R_{j,j}-R_{j,j}^{(1)}|^{a/2}\leq \frac{4}{(N\eta)^{a/2}}.
\end{aligned}
\end{equation}
Moreover, since the function $f(y)=\mathbf{1}\left\{|\operatorname{Im}(y)|\leq \eta \right\}\operatorname{Im}(y)^{a/2}+\mathbf{1}\left\{|\operatorname{Im}(y)|\geq \eta\right\}\eta^{a/2}$ is Lipschitz with Lipschitz-constant $L=a\eta^{1-a/2}$, we can apply Lemma \ref{Misa resovlent} for $x=N^{-1/2}\eta^{a/2}\log(N)$ to get that
\begin{equation}\label{Prosegisi X apo mesi timi prop 7.10}
\mathbb{P}\left(\left|\frac{1}{N}\sum_{j=1}^{N}|\operatorname{Im}(R_{N+j,N+j})|^{a/2}-\frac{1}{N}\E\ |\operatorname{Im}(R_{N+j,N+j})|^{a/2}\right|\geq \frac{\log(N)}{N^{1/2}\eta^{a/2}}\right)\leq 2\exp\left(-\frac{\log^{2}(N)}{8a^{2}}\right).
\end{equation}
So after combining \eqref{Prosegisi X apo mesi timi prop 7.10} ,\eqref{Prosegisi X apo stieltjes prop 7.10} with \eqref{Ypothesi gia local law} and specifically with the fact that $\E (|R_{j,j}|^{a/2})\leq \E( |R_{j,j}|^{2})^{a/4}\leq \frac{1}{\epsilon^{a/4}} $ we get that,
\begin{equation}\label{Fragma gia to X prop 7.10}
  \mathbb{P}\left(|\mathop{X'}|>C\right)\leq C \exp\left(-\frac{\log^{2}(N)}{C}\right),
\end{equation}
for sufficient large universal constant $C$. So the bounding for $\frac{1}{N}\|A^{1/2}Y\|^{a}_{a}$ comes from a direct application of Lemma \ref{Gaussianes kaieystathis gia mh-diagonious cov pinakes} with the bounding for $V$ and $\mathop{X}'$ proven in \eqref{Fragma gia to U prop 7.10} and \eqref{Fragma gia to X prop 7.10} 
  \\Note that \eqref{Kato Fragma ton Rjj} is a corollary of \eqref{fragma Si-Ti mesa sto lhmma} and \eqref{Taytotita gia Rii}.
\end{proof}
The following is the analogue of Proposition 5.9 in \cite{aggarwal2021goe}, adjusted to our set of matrices.
\begin{lem}\label{Fragma T_i}
    There exists some constant $C=C(a)$ such that for any $x\geq 1$ and for any $i\in [2N]$, it is true that
    \begin{align}
     \mathbb{P}\left[|T_{i}|\geq \frac{Cx}{(N\eta^2)^{1/2}} \right] \leq \frac{C}{x^{a/2}} 
    \end{align}
\end{lem}
\begin{proof}
It is sufficient to prove it only for $i=1$. Recall the definition of $T_1$ in \eqref{orismos T_i}. So we need to prove that
    \begin{align}\label{u_i,tail bound for T_i} \mathbb{P}\left(|U_{1}|\geq \frac{x}{(N\eta^2)^{1/2}} \right)\leq \frac{C}{x^{a/2}}.
  \end{align}

   Set the event 
  \begin{align}\label{omega_1 bound}
  \Omega_{1}(s)=\cap_{N+1\leq j\leq 2N}\{|X_{1,j}|\leq s\}.
\end{align}
  Then 
  \begin{align}\label{omega, U_i}
\mathbb{P}\left(|U_{1}|\geq \frac{x}{(N\eta^2)^{1/2}} \right)\leq \mathbb{P}\left(\mathbf{1}\left(\Omega_{1}(s)\right)|U_{1}|\geq \frac{x}{(N\eta^2)^{1/2}}  \right)+ \mathbb{P}(\Omega^{c}_{1}(s))
  \end{align}
  For the second summand in the right-hand-side of \eqref{omega, U_i} by the union bound and \eqref{tailbound} one has that
  \begin{align}\label{omera simpliroma}
\mathbb{P}(\Omega^c_{1}(s))\leq \sum_{j=N+1}^{2N} \mathbb{P}(|X_{1,j}|\geq s)\leq \frac{C}{s^a}  \end{align}
For the first term on the right hand-side of \eqref{omega, U_i} note that by Markov's inequality, the independence of $\{X_{1,j}\}_{j \in [2N]\setminus [N]}$ and $R^{(1)}$ and the symmetry of the random variables $X_{j,1}$
\begin{align}
&\mathbb{P}\left(\mathbf{1}\left(\Omega_{1}(s)\right)|U_{1}|\geq \frac{x}{(N\eta^2)^{1/2}}   \right)\leq \frac{N\eta^2}{x^2}\E \left|\sum_{k,j \in 2N]\setminus[N]:k\neq j}X_{1j}R_{jk}^{(1)}(z)X_{k,1}\right|^2 \mathbf{1}\left(\Omega_{1}(s)\right)
    \\\label{ineq U_i R X Omega}&= \frac{ 2 N\eta^2}{x^2} \sum_{k,j\in [2N]\setminus[N]: k
    \neq j} \E |R^{(1)}_{j,k}|^2 \E (X^2_{1,j}\mathbf{1}\left(\Omega_{1}(s)\right))^2.
\end{align}
We will bound each of the terms inside the sum in \eqref{ineq U_i R X Omega}  individually. Firstly 
\begin{align}\label{X_{1,j}<s}
    \E X^{2}_{1,j} \mathbf{1}\left(\Omega_{1}(s)\right) \leq \E X^{2}_{1,j} \mathbf{1}\left\{|X_{1,j}|\leq s\right\} \leq \frac{2 C s^{2-a}}{(2-a)N}.
\end{align}
The last inequality in \eqref{X_{1,j}<s} can be found in the proof of Proposition 5.9 \cite{aggarwal2021goe}.

Moreover, due to \eqref{wardident} and \eqref{resolventwq2} one has that
\begin{align}\label{ward R}
    \sum_{k,j \in [2N]\setminus[N]:k\neq j} \E |R^{(1)}_{j,k}|^2 < \sum_{j=1}^{N}\frac{\operatorname{Im}( R^{(1)}_{j+N,j+N})}{\eta}\leq \frac{N}{\eta^2}.
\end{align}
Thus combining \eqref{ward R}, \eqref{X_{1,j}<s} and \eqref{omera simpliroma} we get that for some absolute constant $C=C(a)$ it is true that
\begin{align}\label{U_i telos}
\mathbb{P}\left(|U_{1}|\geq \frac{x}{(N\eta^2)^{1/2}} \right) \leq \frac{C s^{4-2a}}{x^2}+ \frac{C}{s^a}.
\end{align}
Setting $s=x^{1/2}$, we get \eqref{u_i,tail bound for T_i}.
\end{proof}

\subsection{Proof of Theorem \ref{To theorima gia to local}}\label{subsection apodiksi aplou local law}

In order to prove Theorem \ref{To theorima gia to local}, we wish to replace the entries of $X$ by $a-$stable entries in several quantities, for example in quantities defined in \eqref{Orismos S1-T1}, in order to use the properties of the $a-$stable distribution.

Firstly consider the following 
\begin{defn}\label{Orismoi gia antikatastasi me Zi}
Define the following quantities:
\[\mathop{\omega_{z}}(u)^{(i)}=\Gamma\left(1-\frac{a}{2}\right)(iz-i\mathop{S_{i}}|u)^{a/2}, \ \ \ \bar{\mathop{\omega_{z}}}(u)=\E\mathop{\omega_{z}}(u)^{(i)},\]
\[G_{i}=\sum_{j:|j-i|\geq N}Z_{j,j}R^{(i)}_{j,j}, \ \ \ \ \mathop{\Psi_{z}^{(i)}}(u)=\Gamma\left(1-\frac{a}{2}\right)(iz-i\mathop{G_{i}}|u)^{a/2}, \ \ \ \mathop{\psi_{z}}(u)=\E(\Psi_{z}(u)).\]
Here $Z_{j,j}$ are i.i.d. random variables from the definition of the matrix $D_{N}$, all with law $N^{-1/a}Z$ where $Z$ is a $(0,\sigma)$ a-stable random variable as in Definition \ref{o orismos toy pinaka} .

\end{defn}
We start this subsection with a comparison between $(-z-S_{i})^{-1}$ and $R_{i,i}$.
\begin{lem}\label{R_i,S_i compare}
    For any $p>0$ there exists a constant $C=C(a,\epsilon,b,s,p)$ such that
    \begin{align}\label{R_ii-(-S_i-z)}
        &\left|\E |R_{i,i}|^{p}-|(-S_{i}-z)^{-1}|^p \right| \leq  \frac{C \log^{C}(N)}{(N\eta^{2})^{a/8}}  \text{ ,
 }   \\\label{iR_ii-i(S_i-z)}& \left|\E |(-iR_{i,i})|^{p}-\E|(-iz-iS_{i})|^{-p}\right|\leq \frac{C \log^{C}(N)}{(N\eta^{2})^{a/8}},
 \\\label{omega-gamma}& |\gamma_{z}-\bar\omega_{z}|_{1-a/2+s}\leq \frac{C \log^{C}(N)}{(N\eta^{2})^{a/8}}.
    \end{align}
\end{lem}
\begin{proof}
Let $C_1, C_2, C_3$ the constants from Propositions \ref{Fragma Si}, \ref{Fragma Si-Ti} and Lemma \ref{Fragma T_i} respectively and set $C=\max\{C_1, C_2, C_3\}$. Moreover let $E_{1},E_2$ the events  whose probability we bound in Proposition \ref{Fragma Si} and \ref{Fragma Si-Ti} respectively and set $E=E_{1}\cup E_2$. 

Note that due to our assumptions in \eqref{Ypothesi gia local law}, \eqref{resolventwq2} and \eqref{Im(S_i)>Im(T_{i})} it is true that
\begin{align}\label{anisotita 1/Im()}
    \frac{1}{\operatorname{Im} (S_i-T_i+z)} \leq N^{1/2} \text{ , }     \frac{1}{\operatorname{Im}(S_i+z)} \leq N^{1/2}.
\end{align}
 Furthermore by (5.5) in \cite{aggarwal2021goe} one has that for any $u>0$
    \begin{align}
      &\left| |R_{i,i}|^{p}-|(-S_{i}-z)^{-1}|^p \right| 
      \\& \leq\mathbf{1}\left\{|T_i|<u\right\} (p-1)u \left( \left|\frac{1}{\operatorname{Im}(S_i -T_i+z)} \right|^{p+1}+ \left| \frac{1}{\operatorname{Im}(S_i+z)}\right|^{p+1} \right)
      \\&+ \mathbf{1}\left\{|T_i|\geq u\right\} \left( \left|\frac{1}{\operatorname{Im}(S_i -T_i+z)} \right|^{p}+ \left| \frac{1}{\operatorname{Im}(S_i+z)}\right|^{p} \right) 
      \end{align}
So one by Propositions \ref{Fragma Si} and \ref{Fragma Si-Ti} one has that
\begin{align}
    &\E \mathbf{1}\left(E^{c}\right)  \left| |R_{i,i}|^{p}-|(-S_{i}-z)^{-1}|^p \right| \leq 2 u (p-1) C^{p+1} \log^{C(p+1)} N+ 2 \mathbb{P}(|T_i|\geq u) C^{p} \log^{Cp}(N)
    \\& \E \mathbf{1}\left(E\right) \left| |R_{i,i}|^{p}-|(-S_{i}-z)^{-1}|^p \right| \leq 2 u N^{(p+1)/2} \exp \left( \frac{-\log^{2} N}{C} \right)
    \end{align}
So after setting $u=(N\eta^{2})^{-1/4}$ and applying Lemma \ref{Fragma T_i}, we get \eqref{R_ii-(-S_i-z)}.

The proof of \eqref{iR_ii-i(S_i-z)} is analogous and therefore it is omitted.

For the proof of \eqref{omega-gamma} note that 
\begin{itemize}
    \item By \eqref{sigrisi esoterikon ginomenon} applied for $x_{1}=  (iT_i- iS_i-iz)^{-1}, x_{2}=  (-iS_i-iz)^{-1}$ and for $r=\frac{a}{2}$ and $\eta=(2C \log^{2C} N)^{-1}$, we get that there exists a constant $C'=C'(a)>0$ such that for any $u>0$ it is true that
    \begin{align}\label{omega-gamma, <u}
    &\mathbf{1}\left(E^c\right) \mathbf{1}\left\{|T_{i}|<u\right\} |i_{z}-\omega_{z}|_{1-a/2+s} \leq C' \left(2C \log^{2C}(N)\right)^{\frac{a}{2}} \mathbf{1}\left(E^c\right) \mathbf{1}\left\{|T_{i}|<u\right\}\left| |z-T_i+S_i|^{-1}- |z+S_{i}|^{-1}\right|^{\frac{a}{2}} 
    \\&\leq u  C' \left(2C \log^{2C}(N)\right)^{\frac{3a}{2}} \mathbf{1}\left(E^c\right),
    \\&\label{omega-gamma,>u}  \mathbf{1}\left(E^c\right) \mathbf{1}\left\{|T_{i}|\geq u\right\} |i_{z}-\omega_{z}|_{1-a/2+s} \leq C' \left(2C \log^{2C}(N)\right)^{\frac{a}{2}} \mathbf{1}\left(E^c\right) \mathbf{1}\left\{|T_{i}|\geq u\right\} \left| |z-T_i+S_i|^{-1}- |z+S_{i}|^{-\frac{a}{2}}\right|
    \\& \leq  2 C' \left(2C \log^{2C}(N)\right)^{\frac{3a}{2}} \mathbf{1}\left\{|T_{i}|\geq u\right\}. 
    \end{align}
   \item Moreover again by \eqref{sigrisi esoterikon ginomenon} for the same $x_1,x_2$ and $r$ as before and for  $\eta=N^{-1/2}$ there exists a constant $C'=C(a)$ such that
    \begin{align}\label{omega-gamma, crude}
          \mathbf{1}\left(E\right) |i_{z}-\omega_{z}|_{1-a/2+s} \leq 2C' \mathbf{1}\left(E\right) N^{a/4}
    \end{align}
\end{itemize}
Note that by definition $\E \omega_{z}=\bar \omega_z$ and $\E i_{z}=\gamma_z$.
So after summing \eqref{omega-gamma,>u}, \eqref{omega-gamma, <u} and \eqref{omega-gamma, crude}, taking expectation and applying Propositions \ref{Fragma Si} and \ref{Fragma Si-Ti} and Lemma \ref{Fragma T_i} for $x=(N\eta^2)^{1/4}$, we get \eqref{omega-gamma}.  

\end{proof}
\subsubsection{Fixed point equation}
In this subsection we establish the asymptotic fixed point equation. Firstly, we show that the quantities in Definition \ref{Orismoi gia antikatastasi me Zi} are approximately equal to the respective quantities of the Stieltjes transform, i.e., the quantities defined in Definition \ref{orismos theta kai gamma(z)}. The latter is proven in the following proposition.
\begin{prop}\label{Sigrisi Gi me Ri} 
It is true that for any $p \in \N$,
  \begin{equation}\label{ropes Ri kai Gi}
  \left|\E |R_{i,i}|^{p}-\E|(-z-G_{i})|^{-p}\right|\leq \frac{C \log^{C}(N)}{(N\eta^{2})^{a/8}}+\frac{C \log^{C}(N)}{N^{4\theta}} ,
  \end{equation}
  \begin{equation}\label{migadikous ropes Ri kai Gi}
  \bigg|\E |(-iR_{i,i})|^{p}-\E|(-iz-iG_{i})|^{-p}\bigg|\leq \frac{C \log^{C}(N)}{(N\eta^{2})^{a/8}}+ \frac{C \log^{C}(N)}{N^{4\theta}} 
  \end{equation}
  and
   \begin{equation}\label{migadikous ropes Ri kai Gi gia gamma kai y(z)}
  |\gamma_{z}-\mathop{\psi_{z}}|_{1-a/2+s}\leq \frac{C \log^{C}(N)}{(N\eta^{2})^{a/8}} + \frac{C \log^{C}(N)}{N^{4\theta}}. 
  \end{equation}
\end{prop}
\begin{proof}
We first present two facts.
\begin{itemize}
    \item One can show that there exists $C=C(a)>0$ such that
\[\mathbb{P}\left(|S_{1}-G_{1}|\geq N^{-4\theta}\right)\leq C (1+\E(R_{1,1}))N^{-4\theta} ,\]
similarly to the proof of Lemma 6.8 in \cite{aggarwal2021goe}. As a result, by Assumption \eqref{Ypothesi gia local law} we have that \begin{equation}\label{fragma gia Si-Gi}
\mathbb{P}(|S_{1}-G_{1}|\geq N^{-4\theta})\leq C N^{-4\theta},
\end{equation}
for some constant $C=C(a,\epsilon)$.
\item For each $i \in [2N]$ there exists a constant $C=C(a,\epsilon,b)>1$ such that
\begin{align}\label{Im(G_i) bound}
\mathbb{P}\left(\operatorname{Im}(G_{i})<\frac{1}{C(\log(N))^{C}}\right)\leq C \exp\left(-\frac{(\log(N))^{2}}{C}\right).
\end{align}
The proof of \eqref{Im(G_i) bound} is completely analogous to the proof of Proposition \ref{Fragma Si}, after replacing the usage of Lemma \ref{gausianes me pinaka} with Lemma B.1 in \cite{bordenave2013localization}. Therefore it is omitted.  
\end{itemize}
Moreover note that due to Lemma \ref{R_i,S_i compare}, it is sufficient to prove that 
for any $p \in \N$,
  \begin{align}\label{ropes G_i S_i}
  &\left|\E |-z-S_{i,i}|^{-p}-\E|(-z-G_{i})|^{-p}\right|\leq \frac{C \log^{C}(N)}{N^{4\theta}},
 \\&\bigg|\E |(-iz-iS)|^{-p}-\E|(-iz-iG_{i})|^{-p}\bigg|\leq  \frac{C \log^{C}(N)}{N^{4\theta}} 
  \\&|\bar\omega_{z}-\mathop{\psi_{z}}|_{1-a/2+s}\leq  \frac{C \log^{C}(N)}{N^{4\theta}}. 
  \end{align}
  Given \eqref{fragma gia Si-Gi} and \eqref{Im(G_i) bound}, the proof of \eqref{ropes G_i S_i}, is completely analogous to the proof of Lemma \ref{R_i,S_i compare}, therefore it is omitted.
\end{proof}
Moreover, we have the following results which will be used in order to establish the limiting fixed point equation. The following Lemma will be the basis for the approximation of the fixed point equation.
\begin{lem}\label{fixed point gia Gi}
Recall Definition \ref{orismos tou Y}. It is true that,
\begin{equation}\label{fixed point mesa sto limma}
\mathop{\Psi_{z}}(u)=\E_{\mathcal{D}}(Y_{\zeta}(u)),
\end{equation}
where $Y_{\zeta}$ is as in Definition \ref{orismos tou Y}, $\mathcal{D}=\{y_{i}\}_{i \in [N]}$ is an $N-$dimensional Gaussian random variable independent from any other quantity with covariance matrix being the identical, $\E_{\mathcal{D}}$ denotes the expectation with respect to the random variable $D$ and
\[\zeta(u)=\frac{1}{N}\sum_{j=1}^{N}\left(-iR^{(1)}_{N+j,N+j}|u\right)^{a/2}\frac{|y_{j}|^{a}}{\E |y_{j}|^{a}}.\]
Also, 
\[\E (-iz-iG_{1})^{-p}=\E_{\mathcal{D}}s_{p,z}(\zeta(1)), \ \ \ \E |-z-G_{1}|^{-p}=\E_{\mathcal{D}} r_{p,z}(\zeta(1)) .\]
\end{lem}
\begin{proof}
This Lemma is a corollary of [\cite{bordenave2017delocalization} Corollary 5.8]
\end{proof}
So in Proposition \ref{Sigrisi Gi me Ri}, we manage to approximate the quantities involving $G_{1}$, such as $\mathop{y_{z}}(u)$, by the analogous quantities involving $R_{1,1}$, such as $\gamma_{z}(u)$. In order to establish the asymptotic fixed point equation, we will need to approximate the function $\zeta(u)$ mentioned in Lemma \ref{fixed point gia Gi} by $\gamma_{z}(u)$ and then take advantage of \eqref{fixed point mesa sto limma}. This approximation is done via the following Lemma.
\begin{lem}\label{7.17}
There exists a constant $C=C(a,\epsilon,s)>1$ such that
\begin{equation}\label{Zeta kai gamma(z)}
  \mathbb{P}\left(|\zeta-\gamma_{z}|_{1-a/2+s}>\frac{C\log^{C}(N)}{N^{s/2}\eta^{a/2}}\right)\leq C \exp\left(-\frac{\log^{2}(N)}{C}\right).
  \end{equation}
\end{lem}
\begin{proof}
Firstly note that $\zeta $ is close to $\E_{D}\zeta$ with high probability due to Lemma \ref{gausianes kai resovlent} for appropriate $x$, i.e.,
\begin{equation}\label{Konta stin proti mesi timi Prop 7.17}
\mathbb{P}\left(|\zeta-\E_{D}\zeta|_{1-a/2+s}\geq \frac{\log^{s}(N)}{N^{s/2}\eta^{a/2}}\right)\leq C \exp\left(-\frac{\log^{2}N}{C}\right),
\end{equation}
for some constant $C=C(a)$. Next, note that by Lemma \ref{Anisotita sigkentrosis se periergous xorous} applied for the matrix $X^{(1)}$ and for appropriate $x$ one has that
\begin{equation}\label{Konta stin deyteri mesi timi prop 7.17}
\mathbb{P}\left(|\E_{X^{(1)}}\E_{D}\zeta-\E_{D}\zeta|_{1-a/2+s}\geq \frac{\log^{s}(N)}{N^{s/2}\eta^{a/2}}\right)\leq C \exp\left(-\frac{\log^{2}N}{C}\right),
\end{equation}
for some appropriately chosen constant $C=C(a)$. Here $\E_{X^{(1)}}$ denotes the mean value with respect to the law of the matrix $X^{(1)}$. Next by Lemma \ref{Frasontas F1-F2 se periergous xorous} one has that 
\begin{equation}
\sum_{i=1}^{N}\frac{1}{N}|R_{N+i,N+i}-R^{(1)}_{N+i,N+i}|_{1-a/2+s}\leq C \eta^{-a/2}\frac{1}{N}\sum_{i=1}^{N}\left(|R_{N+i,N+i}-R^{(1)}_{N+i,N+i|}|^{a/2}+\eta^{s}| R_{N+i,N+i}-R^{(1)}_{N+i,N+i}|^{s}\right).
\end{equation}
So after applying Lemma \ref{Nteterministiko fragma ton minor resolvents} and since $R_{j,j}$ are identical distributed, one has the deterministic bound
\begin{equation}\label{nteterministiko fragma minor prop 7.17}
|\E_{X^{(1)}}\E_{D}\zeta-\gamma_{z}|\leq C'\left(\frac{1}{\eta^{a}N^{a/2}}+\frac{1}{N^{s}\eta^{a/2}}\right).
\end{equation}
So after combining \eqref{Konta stin proti mesi timi Prop 7.17} ,\eqref{Konta stin deyteri mesi timi prop 7.17} and \eqref{nteterministiko fragma minor prop 7.17}, we get the desired inequality.
\end{proof}
Next, we give some more approximating results.
\begin{cor}\label{7.18}
There exists a constant $C=C(a,\epsilon,s)>0$ such that
\begin{equation}\label{prota 2 limma 7.18}
|\gamma_{z}|_{1-a/2+s}<C, \ \ \ \inf_{u \in S^{1}_{+}}\operatorname{Re}(\gamma_{z}(u))>\frac{1}{C}, 
\end{equation}
\begin{equation}\label{apotelesma ton proton 2 limma 7.18}
\mathbb{P}\left(\inf_{u \in S_{+}^{1}}\zeta(u)<\frac{1}{C}\right)<C\exp\left(-\frac{\log^{2}(N)}{C}\right),
\end{equation}
\end{cor}
\begin{proof}
By \eqref{Zeta kai gamma(z)}, the estimate in \eqref{apotelesma ton proton 2 limma 7.18} is a consequence of \eqref{prota 2 limma 7.18}.

For \eqref{Zeta kai gamma(z)} note that due to the first estimate in \eqref{sigrisi esoterikon ginomenon} one has that there exists a constant $C=C(s)$ such that
\begin{align}\label{mean value (-iR_ii)}
   \left|(-iR_{i,i}|u)^{a/2}\right|_{1-\frac{a}{2}+s} \leq C |R_{i,i}|^{a/2}. 
\end{align}
By integrating \eqref{mean value (-iR_ii)} and by the definition of $\gamma_z$ in Definition \ref{orismos theta kai gamma(z)} one has that,
\begin{align}\label{gamma <c}
    |\gamma_{z}|_{1-\frac{a}{2}+s}\leq C \Gamma \left(1-\frac{a}{2}\right) \E |R_{i,i}|^{a/2} \leq C \Gamma \left(1-\frac{a}{2}\right) \left(\E |R_{i,i}|^{2}\right)^{a/4}\leq \epsilon^{-a/4} C \Gamma \left(1-\frac{a}{2}\right). 
\end{align}
Where in the first inequality in \eqref{gamma <c} we used \eqref{mean value (-iR_ii)}, in the second we used Holder's inequality and in the third we used our Assumption \ref{Ypothesi gia local law}. So the first estimate in \eqref{prota 2 limma 7.18} is proven.

For the second estimate in \eqref{prota 2 limma 7.18} one has that for any $u \in S^{1}_{+}$ 
\begin{align}\label{re gamma_z}
    \operatorname{Re} \gamma_{z}(u)= \Gamma \left( 1- \frac{a}{2} \right) \E \operatorname{Re} (iR_{i,i}|u)^{a/2} \geq \Gamma \left( 1- \frac{a}{2} \right) \E \left( \left(\operatorname{Re} (iR_{i,i}|u)\right)^{a/2}  \right)\\ \geq  \Gamma \left( 1- \frac{a}{2} \right)  \E (\operatorname{Im} R_{i,i})^{a/2}\geq \Gamma\left( 1- \frac{a}{2} \right) \epsilon 
\end{align}
where in the first inequality in \eqref{re gamma_z} we used the fact that $\operatorname{Re}c^r \geq (\operatorname{Re}c)^r$ for any $c \in \mathbb{K}^+$ and $r \in (0,1)$, see the proof of Lemma 7.18 in \cite{aggarwal2021goe}, in the second inequality we used the fact that $\operatorname{Re}(c|u)\geq \operatorname{Re} (c)$ for any $c \in \mathbb{K}^+$ and $u \in S_{+}^{1}$ and in the third we used our Assumption \ref{Ypothesi gia local law}. Thus the second estimate in \eqref{prota 2 limma 7.18} is proven.

\end{proof}
Before presenting the proof of Theorem \ref{To theorima gia to local}, we need a last approximation result. 
\begin{lem}
There exists a constant $C=C(a,\epsilon,s)$ such that
\begin{equation}\label{limma 7.19}
\mathbb{P}\left(|\psi_{z}-\mathop{Y_{\gamma_{z}}}|_{1-a/2+s}>\frac{C \log^{C}(N)}{N^{s/2}\eta^{a/2}}\right)<C\exp\left(-\frac{\log^{2}(N)}{C}\right).  
\end{equation}
\end{lem}
\begin{proof}
The strategy of the proof is firstly to approximate $Y_{\gamma_z}$ by $Y_{\zeta}$ and then use Lemma \ref{fixed point gia Gi}.
\begin{itemize}
    \item For the approximation of $Y_{\gamma_z}$ and $Y_{\zeta}$: Let $C_1, C_2$ be the constants mentioned in Lemma \ref{7.17} and Corollary \ref{7.18}. Set $C=2\max\{C_1,C_2\}$. Moreover define the following sets
    \begin{align}
E_1=\left\{|\zeta-\gamma_{z}|_{1-a/2+s}>\frac{C\log^{C}(N)}{N^{s/2}\eta^{a/2}}\right\} 
\\ E_2= \left\{ \inf_{u \in S_{+}^{1}}\operatorname{Re}\zeta(u)<\frac{1}{C} \right\}
    \end{align}
   By Lemma \ref{7.17} and Corollary \ref{7.18} one has that
    \begin{align}\label{7.19 E_1,E_2}
        \mathbb{P} (E_1 \cup E_2)\leq C 
        \exp \left(- \frac{\log^2 N}{C} \right)
    \end{align}
    Set $F$ the complement event of $E_1 \cup E_2$. 
    
   So 
    \begin{align}\label{Y_z-Y_gammaz,F}
        \mathbf{1}\left(F\right) \left|Y_\zeta -Y_{\gamma_z}\right|_{1-\frac{a}{2}+s}\leq \mathbf{1}\left(F\right) C_1 |\zeta -\gamma_{z}|_{1-\frac{a}{2}+s} \bigg(1+ |\gamma_z|_{1-\frac{a}{2}+s}+|\zeta|_{1-\frac{a}{2}+s}    \bigg) \leq \mathbf{1}\left(F\right) \frac{C\log^{C}(N)}{N^{s/2}\eta^{a/2}}\left(1 +\frac{2}{C}  \right)
    \end{align}
    where in the first inequality of \eqref{Y_z-Y_gammaz,F} we used Lemma \ref{Y_f,Y_g} and Remark \ref{sigrisi sup. norm gia periergous xorous} ($C_1$ is the constant mentioned in Lemma \ref{Y_f,Y_g}) and the fact that $\gamma_z, \zeta 
    \mathbf{1}\left(F\right) \in H^{1/C}_{\frac{a}{2},1-\frac{a}{2}+s}$ by Corollary \ref{7.18} and the definition of the set $F$. For the second inequality we used again the definition of $F$,  Corollary \ref{7.18} and Lemma \ref{7.17}.

    Now working on the event $E_1\cup E_2$ we get that by Lemma \ref{F_h(g)} and Corollary \ref{7.18} we there exists a constant $C'>0$ such that
    \begin{align}\label{Y_gamma_z, 7.19}
   \mathbf{1}(E_1 \cup E_2) |Y_{\gamma_z}|_{1-\frac{a}{2}+s}\leq C' \eta^{-\frac{a}{2}}  (1+C) \mathbf{1}(E_1\cup E_2)
    \end{align}
    \item Note that similarly to the proof of \eqref{anisotita 1/Im()} one can prove that
    \begin{align}
       \left| \frac{1}{G_i+z}\right| \leq \frac{1}{\eta}
    \end{align}
    Thus, we can apply \eqref{sigrisi esoterikon ginomenon} to get that there exists a constant $C=C(a)$ such that
    \begin{align}\label{PSI_z,7.19}
        |\Psi_z|_{1-\frac{a}{2}+s}\leq C\eta^{-a/2}
    \end{align}
    So by Lemma \ref{fixed point gia Gi}, one has that
    \begin{align}\label{7.19 bound}
         |\psi_z-Y_{\gamma_z}|_{1-\frac{a}{2}+s}\leq  \E \mathbf{1}\left(F\right)|Y_{\zeta}-Y_{\gamma_z}|_{1-\frac{a}{2}}+ \E \mathbf{1}\left(E_1\cup E_2\right) |\Psi_{z}|_{1-\frac{a}{2}+s}+\E \mathbf{1}\left(E_1\cup E_2\right) |Y_{\gamma_z}|_{1-\frac{a}{2}+s}  
    \end{align}
\end{itemize}
Now \eqref{limma 7.19} is proven by combining \eqref{7.19 E_1,E_2}, \eqref{PSI_z,7.19}, \eqref{Y_gamma_z, 7.19}, \eqref{7.19 bound} and \eqref{Y_z-Y_gammaz,F}.
\end{proof}
Next, the proof of the main theorem of this subsection is presented.
\begin{proof}[Proof of Theorem \ref{To theorima gia to local}]
Note that, \eqref{fixed point equation} is a consequence of \eqref{migadikous ropes Ri kai Gi} and \eqref{limma 7.19}. Additionally \eqref{Kato Fragma ton Rjj mesa sto theorima 7.8} is already proven in \eqref{Kato Fragma ton Rjj}. Lastly, note that \eqref{kato fragma sto theorima 7.8} is a consequence of \eqref{prota 2 limma 7.18}.
So all that remains is to establish \eqref{fixed point Ropes} and \eqref{fixed point ropes v.2} in order to complete the proof. 
We will prove only \eqref{fixed point ropes v.2}. The proof of \eqref{fixed point Ropes} is similar and will be omitted.

To that end, define the sets $E_1$ and $E_2$ as in \eqref{7.19 E_1,E_2} and $F$ the complement event of $E_1 \cup E_2$.  
So by the first estimate in \eqref{r_pz comparison} and Remark \ref{sigrisi sup. norm gia periergous xorous} one has that
\begin{align}\label{1_F r_pz}
   \mathbf{1}\left(F\right)|r_{p,z}(\zeta)-r_{p,z}(\gamma_z)|\leq \mathbf{1}\left(F\right) C'|\gamma_z - \zeta|
\end{align}
for some constant $C'$. By the definition of the event $F$ and Lemma \ref{fixed point gia Gi}, we get the bound in \eqref{fixed point ropes v.2} on the event $F$.

On the event $E_1\cup E_2$ we can use the deterministic bound in Lemma \ref{F_h(g)} to get that
\begin{align}\label{R_pz, E_1,E_2}
    \mathbf{1}\left(E_1\cup
E_2\right)|r_{p,z}(\zeta)-r_{p,z}(\gamma_z)|\leq 2 C'' \eta^{-p} \mathbf{1}\left(E_1 \cup E_2\right)
\end{align}
for some other constant $C''$. Now the bound in \eqref{fixed point ropes v.2} on the event $\mathbf{1}\left(E_1\cup E_2\right)$ is a consequence of \eqref{7.19 E_1,E_2} and \eqref{R_pz, E_1,E_2}.
\end{proof}
\section{Universality for the least singular value after short time}\label{section universality for the least singular value}
At this section, universality of the least eigenvalue for the matrices $X+\sqrt{t}W$ is proven. More precisely :
\begin{thm}\label{universality for the short time}
Let $L_{N}$ be an $N \times N$ matrix with i.i.d. entries all following the Gaussian distribution with mean $0$ and variance $\frac{1}{N}$, independent from $H_{N}$. Then denote $W$ be the symmetrization of $L_{N}$. Let $\tilde{W}$ be an independent copy of $W$. Moreover for every matrix $Y$ denote $\lambda_{N}(Y)$ to be the smallest positive eigenvalue of $Y$. Then for all $a\in (0,2)$ for which local law, Theorem \ref{local law}, holds there exists $\delta_{\ref{universality for the short time}}=\delta_{\ref{universality for the short time}}(a)>0$ such that for all $r>0$
\begin{equation}
\left|\mathbb{P}(N\xi\lambda_{N}(X+\sqrt{s}W)\geq r)-\mathbb{P}(N\lambda_{N}(\tilde{W})|\geq r )\right|\leq \frac{1}{N^{\delta_{\ref{universality for the short time}}}}, 
\end{equation}
for all $s \in (N^{2\delta-\frac{1}{2}},N^{-2\delta})$. Note that $\xi$ is the constant defined in \eqref{semicirle law kai logos}. 
\end{thm}
The proof of Theorem \ref{universality for the short time} can be found in paragraph \ref{proof of universality after short time}.

In order to begin the proof we need the following definition.
\begin{defn}\label{ orismos free convolution }
For an $N \times N$ matrix $J$ with eigenvalues $\{\lambda_{i}(J)\}_{i \in [N]}$ we define the free additive convolution of $J$, with $s$ times the semicircle law, to be the probability measure with Stieltjes transform $\mathrm{m_{s,fc}}$, such that 
\[\mathrm{m_{s,fc}}(z)=\frac{1}{N}\sum_{i=1}^{N}\frac{1}{\lambda_{i}(J)-z-s m_{s,fc}(z)}.\]

It can be proven that the equation above has a unique solution. Moreover we denote by $\mathrm{\rho_{s,fc}(E)}$ the density of the free convolution given by $\mathrm{\rho_{s,fc}}(E)=\frac{1}{\pi}\lim_{\epsilon \rightarrow 0}\operatorname{Im}(\mathrm{m_{s,fc}}(E+i\epsilon)).$
\end{defn}
\begin{rem}\label{idiotites convolution}
For $z \in D_{C_{a},\delta}$, the set for which the local law holds in Theorem \ref{local law}, and $s \in (N^{\delta-\frac{1}{2}+\sigma},N^{-2\delta})$, one has that $|\mathrm{m_{s,fc}}(z)-m_{N,s}(z)|\leq \frac{1}{N\eta}$ with overwhelming probability, as is proven in Theorem 4.5 of \cite{che2019universality}. Here $\mathrm{m_{s,fc}}$ is the Stieltjes transform of the free additive convolution of $X$ with $s$ times the semicircle law and $m_{N,s}$ is the Stieltjes transform of the E.S.D of the matrix $X+\sqrt{s}W$, where $W$ is the symmetrization of a matrix with i.i.d. entries all following the Gaussian distribution with mean 0 and variance $\frac{1}{N}$. Moreover the following stability result is true, due to Lemma 4.1 of \cite{che2019universality},
\begin{equation}
   c \leq \operatorname{Im}(\mathrm{m_{s,fc}}(z))\leq C .
  \end{equation}
\end{rem}

In order to establish Theorem \ref{universality for the short time}, we wish to apply Theorem 3.2 in \cite{che2019universality} but we need to take into account Remark 7.6 of \cite{che2021universality}. So firstly we state the following.
\begin{lem}\label{piknotites}
Fix $s \in (N^{2\delta-\frac{1}{2}},N^{-2\delta})$ for appropriate small $\delta$. Then 
\[|\rho_{a}(x)-\mathrm{\rho_{s,fc}}(x)|\leq N^{-\frac{a\delta}{8}},\]
for $x \in (-\frac{1}{C_{a}},\frac{1}{C_{a}})$ and $a,\delta$ are parameters satisfying the assumptions of Theorem \ref{local law} and $C_{a}$ the constant mentioned in the statement of Theorem \ref{local law}. 
\end{lem}
\begin{proof}
The proof of the lemma is due to the local law Theorem \ref{local law} and similar to the proof of \cite{huang2015bulk}, Lemma 3.4, so it is omitted.
\end{proof}

\label{proof of universality after short time}\begin{proof}[Proof of Theorem \ref{universality for the short time}]
Firstly we apply Theorem 3.2 of \cite{che2019universality} to the sequence of matrices $\mathrm{\rho_{sc}}(0)X_{N}$. Note that due to Theorem \ref{local law}, the matrix $X$ satisfies the assumptions of Theorem 3.2 for $g=N^{\delta-\frac{1}{2}}$ and $G=N^{-\delta}$ with overwhelming probability, for any small enough $\delta>0$. So for all $s_{0},s_{1} \in (N^{2\delta-\frac{1}{2}},N^{-2\delta})$ such that $s_{0}=\frac{N^{\omega_{0}}}{N}$ , $s_{1}=\frac{N^{\omega_{1}}}{N}$ with $\omega_{1}<\frac{\omega_{0}}{2}<\frac{1}{2}$, there exists $\delta_{\ref{universality for the short time}}>0$ and a coupling of $\lambda_{N}(X+\sqrt{s_{1}+s_{0}}W)$ and $\lambda_{N}(\tilde{W}+\sqrt{s_{1}+s_{0}}\tilde{W}')$ such that,
\begin{equation}\label{arxiki sigkrisi short time universality}
  \left|\frac{\mathrm{\rho_{sc}}(0)}{\mathrm{\rho_{s_{0},fc}}(0)}\lambda_{N}^{o}(X+\sqrt{s_{1}+s_{0}}W)-\lambda_{N}^{o}(\tilde{W}+\sqrt{s_{1}+s_{0}}\tilde{W}')\right|\leq \frac{1}{N^{\delta_{\ref{universality for the short time}}+1}},
\end{equation}
where $\tilde{W},\tilde{W}'$ are independent copies of $W$. Moreover, by the properties of the Gaussian law, one has that $\tilde{W}+\sqrt{s_{1}+s_{0}}\tilde{W}' $ has the same law as $\sqrt{1+s_{1}+s_{0}}\tilde{W}''$, where $\tilde{W}''$ is again an independent copy of $W$. But by Slutsky's theorem one has that $\lim_{N \rightarrow \infty}N\lambda_{N}(\sqrt{1+s_{1}+s_{0}}\tilde{W}'') \mathop{=}\limits^{d} \lim_{N \rightarrow \infty}N\lambda_{N}(\tilde{W}'')$. So one has that for each $r>0,$ 
\begin{equation}\label{short time universality me ton tyxaio logo}
\left|\mathbb{P}(N\frac{\mathrm{\rho_{sc}}(0)}{\mathrm{\rho_{s_{0},fc}}(0)} \lambda_{N}(X+\sqrt{s_{1}+s_{0}}W)\geq r)-\mathbb{P}(N\lambda_{N}(\tilde{W})\geq r)\right|\leq \frac{1}{N^{\delta_{\ref{universality for the short time}}}},
\end{equation}
where we have violated the notation in \eqref{short time universality me ton tyxaio logo} by keeping the same constant $\delta_{\ref{universality for the short time}}$. Next, since Remark \ref{idiotites convolution} and Lemma \ref{piknotites} are true, one has that
\begin{equation}\label{short time universality me stathero logo}
\left|\mathbb{P}\left(N \xi \lambda_{N}(X+\sqrt{s_{1}+s_{0}}W)\geq r\right)-\mathbb{P}(N\lambda_{N}(\tilde{W})\geq r)\right|\leq \frac{1}{N^{\delta_{\ref{universality for the short time}}}}.
\end{equation}

Moreover for $s_{1},s_{2} \in (N^{2\delta-\frac{1}{2}},N^{-2\delta})$, such that $s_{1}<s_{2}$, one can apply Weyl's inequality, Lemma \ref{Weyl's inequality}, to get that
\begin{equation}\label{sigrisi meso Weyl}
  \lambda_{N}(X+\sqrt{s_{1}}W)-\lambda_{N}(X+\sqrt{s_{2}}W)\geq (s_{1}-s_{2}) \lambda_{\min}(W)\geq 0.
  \end{equation}
The first inequality of \eqref{sigrisi meso Weyl} comes from the bottom of Weyl's inequality, for the $\frac{N}{2}+1-$th eigenvalues of $X+\sqrt{s_{1}}W$ and $X+\sqrt{s_{2}}W$ when the eigenvalues are arranged in decreasing order. Note that in the notation we normally use, we have arranged the eigenvalues in decreasing order with respect to their absolute values. The second inequality comes from the fact that $\lambda_{\min}(W)$ is the negative of the maximum singular value of $L$.
  \\So \eqref{sigrisi meso Weyl} implies that if $s_{1}\leq s_{2}$ then 
  \begin{equation}\label{fthinon gia pertubation}
\lambda_{N}(X+\sqrt{s_{1}}W)\geq \lambda_{N}(X+\sqrt{s_{2}W}).
\end{equation}
Finally, fix $s \in (N^{2\delta-\frac{1}{2}},N^{-2\delta})$ and $s_{1}=\frac{N^{\omega_{1}}}{N},s_{2}=\frac{N^{\omega_{2}}}{N}$ parameters such that 
$$s_{1}-\frac{N^{\omega/2}}{N}>N^{2\delta-\frac{1}{2}},  \ \ \  s_{1} < s ,\ \ \ s_{2}-\frac{N^{\omega_{2}/2}}{N}\geq s\ \ \ \ \text{and}  \ \ \ s_{2}< N^{-2\delta}.$$
So by construction, one has that $\lambda_{N}(X+\sqrt{s_{1}}W)$ and $\lambda_{N}(X+\sqrt{s_{2}}W)$ are both universal in the sense of \eqref{short time universality me stathero logo} and $s_{1}<s<s_{2}$. So by \eqref{fthinon gia pertubation},
\[N \xi \lambda_{N}(X+\sqrt{s_{2}}W)\leq N \xi \lambda_{N}(X+\sqrt{s}W)\leq N \xi \lambda_{N}(X+\sqrt{s_{1}}W),\]
which implies Theorem \ref{universality for the short time}.
\end{proof}
\begin{cor}\label{universality gia t}
The least singular value of $X+\sqrt{t}W$ is universal in the sense of Theorem \ref{universality for the short time} , where $t$ is defined in Definition \ref{orismos t}.
\end{cor}
\begin{proof}
We just need to show that $t$ belongs to the interval $(N^{2\delta-\frac{1}{2}},N^{-2\delta})$, for any small enough $\delta>0$, and then apply Theorem \ref{universality for the short time}. Note that the latter claim is true due to the way $\nu$ is chosen in \eqref{statheres}, i.e., 
\[0<\nu(2-a)<\frac{1}{2},\]
and since $t$ is of order $N^{-\nu (2-a)}$.
\end{proof}
\section{Isotropic local law for the perturbed matrices at the optimal scale}\label{section gia pertubed local law}
At this point we have proven, in Theorem \ref{local law}, that some kind of regularity holds for the matrix $X$. Specifically we have proven that with high probability, the Stieltjes transform of $X$ converges to its deterministic limit, and its diagonal entries of its resolvent are logarithmically bounded, for complex numbers with imaginary parts of order just above $N^{-\frac{1}{2}}$. So, at this section we "justify" the reason why we have splitted the matrix $H$ into its "big" and "small" elements, i.e., the matrices $X$ and $A$, in Definition \ref{orismos b-removals}. More precisely, we prove that given the regularity properties of $X$ and after perturbing it by a Gaussian component, then the matrix becomes even more regular in some sense. Thus, what will remain to investigate is whether the "small" elements of $H$ preserve this regularity, which will be proven in the next section. 

Specifically, at this section we show that for any small $\delta>0$, the event $\delta-$dependent events  \begin{equation}
\left\{\sup_{\mathbb{D}_{C_a,\delta}}\sup_{i,j}|T_{i,j}(z)|\leq N^{\delta}\right\},
\end{equation}
hold with overwhelming probability. Here $\mathbb{D}_{C_{a},\delta}=\left\{E+i\eta: E \in \left(-\frac{1}{2C_{a}},\frac{1}{2C_{a}}\right),\eta \in \left[N^{\delta-1}, \frac{1}{4C_{a}}\right]\right\}$ and $C_{a}$ is the constant mentioned in Theorem \ref{To theorima gia to local}. This is stated in Corollary \ref{Veltisto Fragma gia to X+sqrtG}. 

In order to prove the latter, we will show a general result which can be used for a general class of matrices. So except from Corollary \ref{Veltisto Fragma gia to X+sqrtG}, the rest of this section is independent from the rest of the paper. The general result we prove in Theorem \ref{Theorima Local law gia pertubed} is an approximation of the resolvent of the symmetrization of a slightly perturbed by a Gaussian component matrix, which initially satisfies some regularity assumptions, Assumption \ref{Assumption for local pertubed }. This resolvent is approximated by a quantity which involves the free additive convolution of the initial matrix with the semicircle law and the eigenvectors of the initial matrix. This approximation is achieved at any direction on the sphere, so it is called isotropic local law.

The isotropic local law is an analogue of Theorem 2.1 in \cite{bourgade2017eigenvector} for our set of matrices, i.e., matrices perturbed by Gaussian factors with 0 at the diagonal blocks. In \cite{bourgade2017eigenvector} an isotropic local law is proven for matrices after perturbing them by a symmetric Brownian motion matrix. 

This kind of results demands precise computations for the resolvent entries. In our case the "target" matrix, with which we compare the resolvent, is a diagonal matrix who lives in $M_{N}(M_{2}(\C))$. This increases the complexity of the calculations from the symmetric case where the "target matrix" is diagonal, but eventually this increase is not that significant.

\subsection{Terminology}
Firstly we introduce the terminology of \cite{bourgade2017eigenvector}.

For any $N-$dependent random variables $Y_{1},Y_{2}$ we denote
\begin{enumerate}
   \item $Y_1 \preceq Y_2 $ if there exists a universal constant $C>0$ such that $|Y_1|\leq C Y_2$.
   \item $Y_1\preceq_{k} Y_2 $ if there exists a constant $C_{k}$ (which depends on some $k$) such that $|Y_1| \leq C_{k}Y_2$.
   \item $Y_1 \ll Y_2 $ if there   a positive constant $c$ such that $Y_1 N^{c}\leq Y_2$.
\end{enumerate}
\subsection{Statement of the main result of this section}
\begin{assumption}\label{Assumption for local pertubed }
Let V be a deterministic $N \times N$ matrix. Denote $\tilde{V}$ the symmetrization of $V$ and $m_{\tilde{V}}$ the Stieltjes transform of $\tilde{V}$. Assume that there exists a large constant $\mathop{a}>1$ such that
\begin{enumerate}
  \item $\parallel \tilde{V} \parallel_{op} \leq N^{a} $ .
  \item $a^{-1}\leq \operatorname{Im}(m_{\tilde{V}}(z))\leq a$, for all $z \in \{E+i\eta , \ \ E \in (E_{0}-r,E_{0}+r) \ \ ,h_{*} \leq \eta \leq 1\}$ for some $N-  $dependent constants $r,h_{*}$ such that:
  $\frac{1}{N}\ll h_{*}\ll r\leq 1$.
\end{enumerate}
\end{assumption}
Moreover, fix $\mathop{c}>0$ some arbitrary small constant and set \begin{equation}\label{orismos ths statheras y=N^(c-1)}
\psi=\frac{N^{\mathop{c}}}{N}.
\end{equation}
\begin{rem}
Note that the matrix $X$ satisfies with high probability the Assumptions \ref{Assumption for local pertubed } due to Theorem \ref{local law}, for $E_{0}=0$, $h_{*}=N^{\delta-\frac{1}{2}}$ for arbitrary small constant $\delta>0$, $r=\frac{1}{C}$ 
where $C$ is the constant mentioned in Theorem \ref{local law} and since for fixed large $D>0$, one can compute by \eqref{tailbound} that any given entry of $X$ has magnitude greater than $N^{\frac{2D+1}{a}}$ with probability less than $C N^{-2D-2}$, which implies that  

\begin{align}
    \mathbb{P}\left( \parallel X \parallel_{op} \geq N^{\frac{2 (D+3)}{a}} \right) \leq C N^{-2D}. 
\end{align}
\end{rem}
\begin{rem}\label{Singular value decomposition kai antidiagonios}
Let $V$ be a deterministic $N \times N$ matrix. Due to the singular value decomposition of $V$, there exist $J_{1},J_{2}$ two orthogonal $N \times N$ matrices, such that $\Sigma=J_{2}VJ_{1}$, where $\Sigma $ is a diagonal matrix with diagonal entries the singular values of $V$. Then denote $\tilde{V}$ the symmetrization of $V$ and set 
\[U=\begin{bmatrix}J_{1}^{T}&0\\ 0& J_{2} \end{bmatrix}.\]
Then it is true that
\[U\tilde{V}U^{T}=\begin{bmatrix}0&\Sigma \\ \Sigma&0
\end{bmatrix}.\]
Moreover, note that $U$ is orthogonal.
\end{rem}
\begin{defn}
Suppose $V$ is a deterministic matrix which satisfies the Assumption \ref{Assumption for local pertubed } for some N-dependent constants $h_{*},r$. Then for any $\mathop{k}\in (0,1)$, define the set 
$$\mathbb{D}_{\mathop{k}}=\left\{z=E+i\eta:E \in (E_{0}-(1-\mathop{k})r,E_{0}+(1-\mathop{k})r), \frac{\psi^{4}}{N}\leq \eta \leq 1-\mathop{k}r\right\}.$$
The parameter $\psi$ is defined in \eqref{orismos ths statheras y=N^(c-1)}.
\end{defn}
\begin{defn}
Recall the definition of the the Stieltjes transform of the Empirical spectral distribution of $\tilde{V}$ with $s-$times the semicircle law in Definition \ref{ orismos free convolution }. We will use the following notation
\[\mathrm{m_{s,fc}}(z)=\frac{1}{N}\sum_{\{i \in [N]\}\cup \{-i \in [N]\}}g_{i}(s,z), \ \ \ \text{with} \ \ \ \ \ g_{i}(s,z)=\frac{1}{\lambda_{i}-z-s \mathrm{m_{s,fc}}(z)}\]
and $\lambda_{i}$ are the eigenvalues of $\tilde{V}$ arranged in increasing order so that $\lambda_{i}=-\lambda_{-i}$.
\end{defn}
\begin{thm}\label{Theorima Local law gia pertubed}
Let $V$ be a deterministic matrix that satisfies the Assumptions \ref{Assumption for local pertubed }. Denote the matrix $$G(z,s)=(\tilde{V}+\sqrt{s}W-z\mathbb{I})^{-1}.$$ 
Here $W$ is the symmetrization of a matrix with i.i.d. entries, all following the Gaussian law with mean 0 and variance $\frac{1}{N}$. Moreover fix $U$ to be the orthogonal matrix constructed in Remark \ref{Singular value decomposition kai antidiagonios} for $V$. Moreover fix $\mathop{k}\in (0,1)$, $s: h_{*}\ll s \ll r$ and $q \in \R^{N}:\|q\|_{2}=1$. Then it is true that,
\begin{align}
    &\left|\langle q,G(s,z)q \rangle-\sum_{i=-N}^{N}\frac{1}{2}(g_{i}+g_{-i})(s,z)\langle u_{i},q\rangle^{2}-\sum_{i=1}^{N}(g_{i}-g_{-i})(s,z)\langle u_{i},q\rangle \langle u_{i+N},q \rangle \right|\\& \preceq \frac{\psi^{2}}{\sqrt{N\eta}}\operatorname{Im}\left(\sum_{i=1}^{N}(\langle u_{i},q\rangle^{2} +\langle u_{i+N},q \rangle^{2}) (g_{i}(s,z)+g_{-i}(s,z)) \right),
    \end{align}
with overwhelming probability, uniformly for all $z \in \mathbb{D}_{k}$. Here $u_{i} $ denote the columns of $U$.
\end{thm}
Set $C_{j}$, for $j \in \{1,2\}$, to be the $N \times N$ diagonal matrices with their $i-th$ diagonal element equal to $g_{i}+(-1)^{j+1}g_{-i}$. Fix the $2N \times 2N $ matrix,
\[C=\frac{1}{2}\begin{bmatrix} C_{1} & C_{2} \\C_{2}  & C_{1} \end{bmatrix}.\]
 In general, what Theorem \ref{Theorima Local law gia pertubed} states is that the matrix $G(z,s)$ can be well approximated by $U C U^*$, since $$\langle q,U C U^* q\rangle=\sum_{i=-N}^{N}\frac{1}{2}(g_{i}+g_{-i})(s,z)\langle u_{i},q\rangle^{2}+\sum_{i=1}^{N}(g_{i}-g_{-i})(s,z\langle u_{i},q\rangle \langle u_{i+N},q \rangle.$$
\\Moreover we can reduce the proof of Theorem \ref{Theorima Local law gia pertubed} to the diagonal case.
\begin{thm}\label{Theorima Local pertubed gia diagonious}
Fix $V=\text{diag}(v_{1},\cdots,v_{N})$ a diagonal matrix which satisfies Assumption \ref{Assumption for local pertubed }. Moreover set $W$ to be the symmetrization of a matrix $L$ with i.i.d. entries, all following the Gaussian law with 0 mean and $\frac{1}{N}$ variance. Define the resolvent $G(z,s)=(\tilde{V}+\sqrt{s}W-z\mathbb{I})^{-1}$. Fix  $\mathop{k} \in (0,1)$, $h_{*}\ll s \ll r$ and $q:\|q\|_{2}=1$. Then
\begin{align}
&\left|\langle q,G(s,z)q\rangle-\sum_{i=-N}^{N}\frac{1}{2}(g_{i}+g_{-i})(s,z)q_{i+N}^{2}-\sum_{i=1}^{N}(g_{i}-g_{-i})(s,z)q_{i}q_{i+N}\right|\\
&\preceq \frac{\psi^{2}}{\sqrt{N\eta}}\operatorname{Im}\left( \sum_{i=1}^{N}(q^{2}_{i}+q^{2}_{i+N})(g_{-i}(s,z)+g_{i}(s,z)\right),
\end{align}
holds with overwhelming probability uniformly for all $z \in \mathbb{D}_{\mathop{k}}$.
\end{thm}
The proof of Theorem \ref{Theorima Local pertubed gia diagonious} can be found in paragraph \ref{proof of theorem diagonious}.
\begin{proof}[Proof of Theorem \ref{Theorima Local law gia pertubed} assuming Theorem \ref{Theorima Local pertubed gia diagonious}]
Let $V$ be a general deterministic matrix with singular value decomposition $\Sigma=J_{2}VJ_{1}$ where $J_{1}$ and $J_{2}$ are orthogonal matrices. Define $U$ as in Remark \ref{Singular value decomposition kai antidiagonios}. Then
\[U(\tilde{V}+\sqrt{s}W)U^{T}=\begin{bmatrix}0 & \Sigma + J_{1}^{T}\sqrt{s}L^{T}J_{2}^{T} \\ \Sigma +
J_{2} \sqrt{s}L J_{1} & 0
\end{bmatrix}.\]
But $L$ is invariant under orthogonal transformation, so $J_{2}LJ_{1}$ has the same law as $L$. This implies that $U(\tilde{V}+\sqrt{s}W)U^{T}$ has the same law as $U\tilde{V}U^{T}+\sqrt{s}W$. Next, by the properties of the inner product, one has that
\[\langle q ,(\tilde{V}+\sqrt{s}W-z\mathbb{I})^{-1}q\rangle=\langle q ,U(U\tilde{V}U^{T}+UWU^{T}-z\mathbb{I})^{-1}U^{T}q\rangle= \langle U^{T}q ,(U\tilde{V}U^{T}+UWU^{T}-z\mathbb{I})^{-1}U^{T}q\rangle.\]
By a similar computation for $\langle q,U CU^* q\rangle$, one reduces the problem in bounding
$$\left|\langle q,(U\tilde{ V}U^{T}+\sqrt{s}W-z\mathbb{I})^{-1}q\rangle-\left(\sum_{i=1}^{N}q^{2}_{i}g_{i}(t,z)+q^{2}_{i+N}g_{-i}\right)-\sum_{i \in [N]}q_{i}q_{i+N}(g_{i}-g_{-i})(s,z)\right|,$$ which is true by a direct application of Theorem \ref{Theorima Local pertubed gia diagonious}.

\end{proof}
So it suffices to prove Theorem \ref{Theorima Local pertubed gia diagonious}, i.e., to consider $V$ to be diagonal. Moreover we have the following identities. 
\begin{rem}\label{Protasi gia peiragmenous diagonioys}
  Let $V$ be a deterministic diagonal matrix which satisfies Assumptions \ref{Assumption for local pertubed }. Adopt the notation of Theorem \ref{Theorima Local pertubed gia diagonious}. Then consider the following matrix $F=\{F_{i,j}\}_{i,j \in [N]}$, where 
  \[F_{i,j}:= \begin{bmatrix}
  0 & [V+\sqrt{s}L]_{i,j} \\ [V+\sqrt{s}L]_{j,i} & 0
  \end{bmatrix},  \text{ for all } i,j \in [N].\]
  Note that there exists a unitary matrix $S$, the product of permutation matrices, such that if we set $F=S^{T} (V+\sqrt{s}W) S$ then  $G(s,z)=S (F-z\mathbb{I})^{-1}S^{T} $ and
  \[(F-z\mathbb{I})_{i,j}^{-1}=\begin{bmatrix} G_{i,j} & G_{i+N,j} \\ G_{i,j+N} & G_{i+N,j+N}\end{bmatrix},\]
  where $G_{i,j}$ are the entries of $G(s,z)$.
  \end{rem}
 It is more convenient to work with the matrix $F$ and its resolvent as it can be thought as a full symmetric matrix in $\mathcal{M}_{N}(\mathcal{M}_{2}(\C))$, instead of a symmetric matrix with $0$ at the diagonal blocks in $\mathcal{M}_{2N}(\C)$.  
\subsection{Proof of Theorem \ref{Theorima Local pertubed gia diagonious}}
In this subsection we will prove Theorem \ref{Theorima Local pertubed gia diagonious}. First, we present some results from \cite{che2019universality}, necessary for the proof.
\begin{prop}[\cite{che2019universality},Theorem 4.5] 
Fix $s$ as in Theorem \ref{Theorima Local pertubed gia diagonious}, the parameter $\psi$ defined in \eqref{orismos ths statheras y=N^(c-1)} and $\mathop{k} \in (0,1)$. Then it is true that, 
\[|m_{s}(z)-\mathrm{m_{s,fc}}(z)|\leq \frac{\psi}{N\eta},\]
holds with overwhelming probability uniformly for all $z \in \mathbb{D}_{\mathop{k}}$. Here $m_{s}(z)$ is the Stieltjes transform of $\tilde{V} +\sqrt{s}W$.  
\end{prop}
\begin{lem}
Fix $s$ and $\mathop{k}$ as in Theorem \ref{Theorima Local pertubed gia diagonious}. Then uniformly for all $z \in \mathbb{D}_{\mathop{k}}$, there exists a constant $C>1$ such that:
\begin{equation}\label{fragma gia to free convolution}
 C^{-1}\leq |\mathrm{m_{s,fc}}(z)|\leq C,
\end{equation}
\begin{equation}\label{fragma gia to free conv kai gia gi}
  |\mathrm{m_{s,fc}}(z)|\leq \frac{1}{N} \sum_{i=1}^{N}|g_{i}|+|g_{-i}|\leq C\log(N).
\end{equation}

\end{lem}
\begin{proof}
These estimates can be found in \cite{che2019universality} Lemma 4.1 and Lemma 4.12. 
\end{proof}
Moreover the following estimates hold.
\begin{lem}
Fix $\mathbb{T} \subseteq [2N]$ such that $|\mathbb{T}| \leq \log(N)$, which consists of pairs of indeces $\{k,k+N\}$ for $k \in [N]$. Moreover set $(H+\sqrt{s}G)^{\mathbb{T}}$ the sub matrix of $H+\sqrt{s}G$ with the $i-th$ columns and row removed for all $i \in \mathbb{T}$ and  $G^{\mathbb{T}}(s,z)=\left((H+\sqrt{s}G)^{\mathbb{T}}-z\mathbb{I}_{2N-|\mathbb{T}|}\right)^{-1}$. Then the following estimates hold with overwhelming probability.
\begin{align}&\label{fragma gia diagonious bgazontas T}
  \left|G^{\mathbb{T}}_{i,i}-\frac{1}{2}(g_{i}+g_{-i})\right|\leq \left(|g_{i}|+|g_{-i}|\right)^{2} \frac{\psi s}{\sqrt{N\eta}}  \text{ for all } i \in [2N]\setminus \mathbb{T} ,
\\&\label{Fragma gia antidiagonio bgazontas T}
\left|G^{\mathbb{T}}_{i,N+i}-\frac{1}{2}(g_{i}-g_{-i})\right|\leq (|g_{i}|+|g_{-i}|)^{2} \frac{\psi s}{\sqrt{N\eta}} \text{ for all } i \in [2N]\setminus \mathbb{T} ,
  \\&\label{fragma gia ta ypolloipa stoixeia Bgazontas T}
\left|G^{\mathbb{T}}_{i,j}\right|\leq \frac{\min(|g_{i}|+|g_{-i}|,|g_{j}|+|g_{-j}|)\psi}{\sqrt{N\eta}}\leq \frac{\bigg((|g_{i}|+|g_{-i}|)(|g_{j}|+|g_{-j}|)\bigg)^{1/2}\psi}{\sqrt{N\eta}} \text{ for all } i,j \in [2N]\setminus \mathbb{T}.
\end{align}
\end{lem}
\begin{proof}
The first two estimates are proven by the Schur Complement formula and the bounds (4.69) and (4.89) from \cite{che2019universality}. The last bound is given in \cite{che2019universality}, equation (4.70) and (4.79). 
\end{proof}
Next, we present a bound for the diagonal and the anti-diagonal entries of $G(s,z)$.
\begin{lem}\label{fragma gia diagonio antidiagonio}
Adopt the notation of Theorem \ref{Theorima Local pertubed gia diagonious}. Then it is true that with overwhelming probability
\begin{align}
&\left|\langle q,G(s,z)q\rangle-\frac{1}{2}(\sum_{i=-N}^{N}q^{2}_{i}(g_{i}(s,z)+g_{-i}(s,z))-\sum_{i=1}^{N}q_{i}q_{i+N}(g_{i}-g_{-i})\right|\preceq \frac{\psi^{2}}{\sqrt{N\eta}}\operatorname{Im}\left(\sum_{i=-N}^{N}q_{i}^{2}\left(g_{i}(s,z)+g_{-i}(s,z)\right)\right)
\\&
+\frac{ \psi^{2} }{\sqrt{N\eta}} \operatorname{Im}\left(\sum_{i=1}^{N}(g_{i}+g_{-i})|(q^{2}_{i+N}+q_{i}^{2})\right) +\left|\sum_{i\neq j,i\neq N+j}G_{i,j}q_{i}q_{j}\right|.
\end{align}
\end{lem}
\begin{proof}
One has that,
\begin{equation}\label{esoteriko ginomeno gia G(s,z)}
\langle q,G(s,z)q\rangle=\sum_{i=1}^{2N}q^{2}_{i}G_{i,i}+2\sum_{i=1}^{N}G_{i,i+N}q_{i}q_{i+N}+\sum_{i \neq j,i\neq N+j}q_{i}q_{j}G_{i,j}.
\end{equation}
So for the first part on the right side of the equality in \eqref{esoteriko ginomeno gia G(s,z)},
one can apply \eqref{fragma gia diagonious bgazontas T}) to get that, 
\[\sum_{i=-N}^{N}q^{2}_{i+N}\left|G_{i,i}-\frac{1}{2}(g_{i}+g_{-i})\right|\leq \frac{s \psi}{\sqrt{N\eta}} \sum_{i=-N}^{N}q^{2}_{i+N}\left(|g_{i}|+|g_{-i}|\right)^{2}\leq \frac{2 s \psi}{\sqrt{N\eta}} \sum_{i=-N}^{N}q^{2}_{i+N}|g_{i}|^{2}+\frac{2 s \psi}{\sqrt{N\eta}}\sum_{i=-N}^{N}q_{i+N}|g_{-i}|^{2}.\]
Next, we can apply Proposition 2.8 from \cite{bourgade2017eigenvector} to get that with overwhelming probability,
\[\sum_{i=-N}^{N}q^{2}_{i+N}\left|G_{i,i}-\frac{1}{2}(g_{i}+g_{-i})\right| \preceq \frac{2 \psi }{\sqrt{N\eta}} \operatorname{Im}\left(\sum_{i=-N}^{N}(g_{i}+g_{-i})q^{2}_{i+N}\right). \]
Similarly, by \eqref{Fragma gia antidiagonio bgazontas T} one has that with overwhelming probability, 
\begin{align}
&\sum_{i=1}^{N}|q_{i}q_{i+N}|\left|G_{i,N+i}-\frac{1}{2}(g_{i}-g_{-i})\right|\preceq \frac{2 \psi }{\sqrt{N\eta}} \operatorname{Im}\left(\sum_{i=1}^{N}(g_{i}+g_{-i})|q_{i+N}^{2}+q_{i}^{2}|\right), \\& \sum_{i=1}^{N}|q_{i}q_{i+N}|\left|G_{N+i,i}-\frac{1}{2}(g_{i}-g_{-i})\right|\preceq \frac{2 \psi }{\sqrt{N\eta}} \operatorname{Im}\left(\sum_{i=1}^{N}(g_{i}+g_{-i})|q_{i+N}^{2}+q_{i}^{2}|\right).
\end{align}
\end{proof}
So in order to prove Theorem \ref{Theorima Local pertubed gia diagonious}, it suffices to prove that
\begin{equation}\label{ropes gia ta perierga antikeimena}
\E \mathop{Z^{2k}}\preceq_{k} \mathop{Y^{2k}} \ \ \ \text{ for all } k \in \N  ,
\end{equation}
where,
\begin{equation}\label{mh diagonia stoixeia gia pertubed pinakes klp}
\mathop{Z}:=\left|\sum_{i \neq j\mod {N}}q_{i}q_{j}G_{i,j}\right| \ \ \ \ \text{and} \ \ \ \ \mathop{Y}=\frac{\log N y }{\sqrt{N \eta}}\operatorname{Im}\left(\sum_{i=1}^{N}(g_{i}+g_{-i})q^{2}_{i}+q^{2}_{i+N}\right).
\end{equation}
By \eqref{ropes gia ta perierga antikeimena}, one can obtain Theorem \ref{Theorima Local pertubed gia diagonious} by Markov's inequality, which will imply \begin{align}
    \left|\sum_{i \neq j \mod{N}}q_{i}q_{j}G_{i,j}\right|\preceq \frac{\psi^{2}}{\sqrt{N\eta}}\operatorname{Im}\left(\sum_{i=1}^{N}(g_{i}+g_{-i})(q^{2}_{i}+q^{2}_{-i})\right)
    \end{align} with overwhelming probability. More precisely for any $D>0$ if we fix $k:\mathop{c}k\geq D$ and sufficient large $N$ such that $N^{\mathop{c}k-D}\geq C_{k}$. Here $\mathop{c}$ is the constant in the definition of $\psi$ in \eqref{orismos ths statheras y=N^(c-1)} and $C_{k}$ is implied in \eqref{ropes gia ta perierga antikeimena}. Thus, one can apply Markov's inequality in order to get that:
\[\mathbb{P}\left(\mathop{Z}\geq \frac{\psi}{\log(N)}\mathop{Y}\right)=\mathbb{P}\left(\mathop{Z^{2k}}\geq \left(\frac{\psi}{\log(N)}\mathop{Y}\right)^{2k}\right)\leq \frac{C_{k}\log^{2k}(N)}{N^{\mathop{c}2k}}\leq \frac{C_{k}}{N^{\mathop{c}k}}\leq \frac{1}{N^{D}}.\]

Next, we give an analysis for the moments of $\mathop{Z}$. Firstly, note that
\begin{equation}\label{eksisosi ropon me b mikro}
\E |\mathop{Z}|^{2k}=\sum_{\textbf{b}}q_{b_{1}}q_{b_{2}}q_{b_{3}}\cdots q_{b_{4k}}\mathbb{E}X_{b_{1},b_{2}}X_{b_{3},b_{4}}\cdots X_{b_{4k-1},b_{4k}},
\end{equation}
where the sum is taken over all $\textbf{b} \subseteq [2N]^{4k}$ such that $b_{2i-1}\neq b_{2i}\mod{N}$ and $X_{b_{2i-1},b_{2i}}=G_{b_{2i-1},b_{2i}}$ for $i \in [k]$ and $X_{b_{2i-1},b_{2i}}=\bar{G}_{b_{2i-1},b_{2i}}$ and $i \in [2k]\setminus[k]$. Furthermore, we can continue the analysis of the sum such that, 
\begin{equation}\label{eksisosi me B kefalaio}
\E |\mathop{Z}|^{2k}=\sum_{\mathbf{B}}\sum_{b_{i}=\{ B_{i},B_{i}+N \}}q_{b_{1}}q_{b_{2}}q_{b_{3}}\cdots q_{b_{4k}}\mathbb{E}X_{b_{1},b_{2}}X_{b_{3},b_{4}}\cdots X_{b_{4k-1},b_{4k}}.
\end{equation}
Now the sum is considered, firstly over all $\mathbf{B} \subseteq [i\mod{N}]^{4k}$ with the restriction that $B_{2i-1}\neq B_{2i}$ and then over the possible $b_{i}= k$ such that $k\in [N]$  or $ k \in B_{i}$ or $b_{i}=k+N$ for $k\in [N]$ and $ k \in B_{i}$. Next, for every summand in \eqref{eksisosi me B kefalaio} set  $\mathbb{T}=\cup_{b_{i}\in \mathbf{B}, b_{i} \in [N] } \{b_{i},b_{i}+N\}$. Moreover set the diagonal block matrices
\[\{\mathrm{M_{fs,s}^{(\mathbb{T})}}\}_{ii \in \mathbb{T}}=\begin{bmatrix} \mathrm{m_{s,fc}} & 0 \\ 0 & \mathrm{m_{s,fc}}
\end{bmatrix}\]
and 
\[M^{\mathbb{T}}_{ii \in [T]}=\begin{bmatrix} m^{(\mathbb{T})} & 0 \\ 0 & m^{(\mathbb{T})}
\end{bmatrix},\]
where $m^{(\mathbb{S})}(z)$ is the trace of the resolvent of $(V+\sqrt{s}W)^{(\mathbb{S})}$ divided by $2N$. Here $\mathbb{S}$ is any subset of $[2N]$ and $(V+\sqrt{s}W)^{(\mathbb{S})}$ is the minor of $(V+\sqrt{s}W)$ with rows and columns not included in $\mathbb{S}$. Moreover set 
\[\Phi_{i,i}=\begin{bmatrix}-z & \lambda_{i} \\ \lambda_{i} & -z \end{bmatrix}\]
and 
\[W'_{i,j}=\begin{bmatrix}0 & w_{i,j} \\ w_{j,i} & 0 \end{bmatrix}.\]

Adopting the notation of Proposition \ref{Protasi gia peiragmenous diagonioys} and since $G(s,z)=S(F-z\mathbb{I})^{-1}S^{T}$, one can apply Schur complements formula to get that 
\begin{align}
&(F-z\mathbb{I})^{-1}_{i, j \in \mathbb{T}}=(\Phi_{i \in \mathbb{T}}+\sqrt{s}W'_{i,j \in \mathbb{T}}-s (W')^{*}_{i,j : i \in \mathbb{T},j \in [2N]\setminus \mathbb{T}}(S^{T}G(s,z)S)^{(\mathbb{T})}W_{i\in [2N]\setminus \mathbb{T}, j \in \mathbb{T}}' )^{-1}
\\&=(D-E^{1}-E^{2}-E^{3})^{-1},
\end{align}
where 
\begin{align}
&D=\Phi_{i \in \mathbb{T}}-s\mathrm{M^{(\mathbb{T})}_{fc,s}}, \text{ 
 } E^{1}=s(M^{\mathbb{T}}-\mathrm{M_{fs,s}^{\mathbb{T}}}), \text{  } E^{2}=-\sqrt{s}W', \\& E^{3}=s (W')^{*}_{i,j : i \in \mathbb{T},j \in [2N]\setminus \mathbb{T}}(S^{T}G(s,z)S)^{(\mathbb{T})}W_{i\in [2N]\setminus \mathbb{T}, j \in \mathbb{T}}-M^{\mathbb{T}}).
\end{align}
\\Next, we wish to estimate the operator norm of the matrix $ED^{-1}$. We will show that,
\begin{equation}\label{anisotita me 2.16 gia ta idioidanismata}
  |ED^{-1}|_{op}\leq C_{k}\frac{\psi}{\sqrt{N\eta}},
\end{equation}
with overwhelming probability. Here $E=\sum_{i=1}^{3}E^{i}$.
\\More precisely, firstly note that $D$ is a $2N \times 2N$ dimensional matrix with 0 at the all non-diagonal $2 \times 2 $ blocks and with diagonal blocks equal to  
\[D_{i,i}=\begin{bmatrix} -z-\mathrm{m_{s,fc}}(z) & \lambda_{i} \\ \lambda_{i} & -z-\mathrm{m_{s,fc}}(z) \end{bmatrix}.\]
So the inverse of $D$ will preserve the same structure. Thus, we can compute that:
\[D^{-1}_{i,i}=\frac{1}{2}\begin{bmatrix}g_{i}+g_{-i} && \text{  } g_{i}-g_{-i} \\ g_{i}-g_{-i} && \text{  }g_{i}+g_{-i} \end{bmatrix}.\]
Moreover since $\operatorname{Im}(z+s \mathrm{m_{s,fc}}(z))\succeq (s+\eta)$, we get that $|g_{i}|\preceq \frac{1}{s+\eta}$, $\text{ for all } i : |i| \in [N]$. All these imply that all the entries of $D^{-1}$ are bounded by $\frac{1}{s+\eta}$ up to some universal constant. So it is implied that
\begin{equation}\label{anisotita 2.16 gia D^-1}
|D^{-1}|_{op}\preceq_{k} \frac{1}{s+\eta} .
\end{equation}
Next, similarly to the proof of (2.16) in \cite{bourgade2017eigenvector} one can prove that
\begin{equation}\label{anisotita gia 2.16 gia E}
|E|_{op}\preceq_{k} (s+\eta)\frac{\psi}{\sqrt{N\eta}}.
\end{equation}

So after combining \eqref{anisotita 2.16 gia D^-1} and \eqref{anisotita gia 2.16 gia E}, we get \eqref{anisotita me 2.16 gia ta idioidanismata}. Set $\mathcal{A}$ to be the event where \eqref{anisotita me 2.16 gia ta idioidanismata} holds with overwhelming probability. Then it is true that for appropriately large $N$
\begin{equation}\label{fragma gia to endexomeno A}
\mathbb{P}(\mathcal{A}^{c})\leq N^{-(4\mathop{a}+6)k}
,\end{equation}
where $\mathop{a}$ is given in the Assumptions \ref{Assumption for local pertubed }. Next by Taylor's expansion on the event $\mathcal{A}$ one has 
\begin{equation}\label{Taylor gia pinakes sthn apodeiksi toy local}
(F-z\mathbb{I})^{-1}_{i,j \in \mathbb{T}}=(D-E)^{-1}=\sum_{l=0}^{f-1}D^{-1}(ED^{-1})^{l}+(D-E)^{-1}(ED)^{-f},
\end{equation}
where $f$ can be chosen to be arbitrary large. We choose $f=\left\lceil \frac{8 k (\mathop{a}+1)}{\mathop{c}} \right\rceil$, where $\mathop{c}$ is mentioned in the definition of $\psi$ in \eqref{orismos ths statheras y=N^(c-1)} and $\mathop{a}$ is mentioned in the Assumption \ref{Assumption for local pertubed }. Moreover since all the non diagonal $2 \times 2$ blocks of $D^{-1}$ are 0, we can ignore the case of $l=0$ in \eqref{Taylor gia pinakes sthn apodeiksi toy local}, since we are interested in the elements of $G_{i,j \in \mathbb{T}}$, such that $b_{i}\neq b_{i+1}\text{mod}{N}$. Moreover set $X_{b_{i},b_{i+1}}^{l}=( D^{-1}(ED^{-1})^{l})_{b_{i},b_{i+1}}$ and $X_{b_{i},b_{i+1}}^{\infty}=((D-E)^{-1}(ED)^{f})_{b_{i},b_{i+1}}$.

So in order to prove \eqref{ropes gia ta perierga antikeimena}, firstly we need to bound $\mathop{Y}$ from below. Note that similarly to \cite{bourgade2017eigenvector} (2.13) one has
\begin{equation}\label{Kato fragma gia to im(gq)}
\operatorname{Im}\left(\sum_{i=-N}^{N}q^{2}_{i+N}g_{i}\right)=\sum_{i=-N}^{N}\frac{(\eta+s\operatorname{Im}\left(\mathrm{m_{s,fc}}(z)\right)q_{i}^{2}}{|\lambda_{i}(0)-z-\mathrm{m_{s,fc}}(z)|^{2}}\succeq \frac{\eta}{N^{2\mathop{a}}},
\end{equation}
due to the fact that $z \in \mathbb{D}_{\mathop{k}}$ , Assumption \ref{Assumption for local pertubed } and \eqref{fragma gia to free conv kai gia gi}. So it is easily implied that
\begin{equation}\label{kato fragma gia Y}
  \mathop{Y}\succeq \frac{\eta \psi \log(N)}{N^{2\mathop{a}}\sqrt{N\eta}}.
\end{equation}
Returning to the analysis of equation \eqref{eksisosi me B kefalaio}, one has that 
\begin{align}
&\sum_{\mathbf{B}}\sum_{b_{i}=\{ B_{i},B_{i}+N \}}q_{b_{1}}q_{b_{2}}q_{b_{3}}\cdots q_{b_{4k}}\mathbb{E}X_{b_{1},b_{2}}X_{b_{3},b_{4}}\cdots X_{b_{4k-1},b_{4k}}
\\&=\sum_{\mathbf{B}}\sum_{b_{i}\{ B_{i},B_{i}+N \}}q_{b_{1}}q_{b_{2}}q_{b_{3}}\cdots q_{b_{4k}}\mathbb{E}\prod_{i=1}^{2k}X_{b_{2i-1},b_{2i}}\mathbf{1}\left(\mathcal{A}\right)+\mathcal{O}(N^{2k}\eta^{-2k})\mathbb{P}(\mathcal{A}^{c}),\end{align}
where the second part on the right hand side of the equation comes from the fact that $X_{b_{i},b_{i+N}}$ are uniformly bounded by $\eta^{-1}$ and the fact that $|q|_{1}\leq N^{1/2}$ since $|q|_{2}=1$. So one has that
\[\eta^{-2k}\mathbb{P}(\mathcal{A}^{c})\left|\sum_{i \neq j}q_{i}q_{j}\right|^{2k}= \eta^{-2k}\mathbb{P}(\mathcal{A}^{c})\left(\sum_{i=1}^{2N}|q_{i}| \sum_{j\neq i}|q_{j}|\right)^{2k}\preceq \eta^{-2k}\mathbb{P}(\mathcal{A}^{c})\left(\sum_{i \in [2N]}N^{1/2}|q_{i}|\right)^{2k}\preceq N^{2k}\eta^{-2k}\mathbb{P}(\mathcal{A}^{c}). \]
Next by \eqref{fragma gia to endexomeno A},\eqref{kato fragma gia Y} and the fact that $z \in \mathbb{D}_{\mathop{k}}$ one has that
\[N^{2k}\eta^{-2k}\mathbb{P}(\mathcal{A}^{c})\leq N^{-2k-2}\eta^{-2k}\leq \frac{\eta^{2k}}{N^{2k}}\leq \mathop{Y^{2k}}.\]
So, we have proven that we can restrain to the event that $\mathcal{A}$ holds. Returning again to the analysis of the sum in the form \eqref{eksisosi ropon me b mikro} one has that
\begin{align}
&\sum_{\textbf{b}}q_{b_{1}}q_{b_{2}}q_{b_{3}}\cdots q_{b_{4k}}\mathbb{E}X_{b_{1},b_{2}}X_{b_{3},b_{4}}\cdots X_{b_{4k-1},b_{4k}}\mathbf{1}\left(\mathcal{A}\right)=\sum_{\textbf{b}}q_{b_{1}}q_{b_{2}}q_{b_{3}}\cdots q_{b_{4k}}\E \mathbf{1}\left(\mathcal{A}\right) \prod_{i=1}^{2k}\sum_{l=1}^{f-1} X^{(l)}_{b_{2i-1},b_{2i}} \\&+\sum_{\textbf{b}}q_{b_{1}}q_{b_{2}}q_{b_{3}}\cdots q_{b_{4k}}\sum_{i=1}^{2k}\E \mathbf{1}\left(\mathcal{A}\right) X^{(\infty)}_{2i-1,2i}\prod_{j \leq i-1}X_{b_{2j-1},b_{2j}}\prod_{j\geq i+1}\left(X_{b_{2j-1},b_{2j}}- X^{(\infty)}_{b_{2j-1},b_{2j}}\right) .
\end{align}
We will show that the second part of the right hand side of the equation is negligible on the event $\mathcal{A}$.
Note that since $|(D-E)^{-1}|_{op}\preceq_{k}\frac{1}{\eta}$ and since \eqref{anisotita me 2.16 gia ta idioidanismata} holds in $\mathcal{A}$ we get that 
\[|X^{(\infty)}_{b_{i},b_{i+1}}| \preceq_{k}\frac{1}{\eta}\left(\frac{\psi}{\sqrt{N\eta}}\right)^{f}.\]
All these, imply that 
\begin{equation}\label{fragma gia ta apeira ED^(-1)}
  \left|\sum_{\textbf{b}}q_{b_{1}}q_{b_{2}}q_{b_{3}}\cdots q_{b_{4k}}\sum_{i=1}^{2k}\E \mathbf{1}\left(\mathcal{A}\right) X^{(\infty)}_{2i-1,2i}\prod_{j \leq i-1}X_{b_{2j-1},b_{2j}}\prod_{j\geq i+1}(X_{b_{2j-1},b_{2j}}- X^{(\infty)}_{b_{2j-1},b_{2j}})\right|\preceq_{k} \frac{N^{2k}}{\eta}\left(\frac{\psi}{\sqrt{N\eta}}\right)^{f}.
\end{equation}
The $N^{2k}$ factor in \eqref{fragma gia ta apeira ED^(-1)}, comes from bounding the quantity $|\sum_{i \neq j}q_{i}q_{j}|$. By the way $f$ is chosen, we get that 
\[\frac{N^{2k}}{\eta}\left(\frac{\psi}{\sqrt{N\eta}}\right)^{f}\preceq_{k} \left(\frac{\eta}{N^{2\mathop{a}}}\right)^{2k}\left(\frac{\psi}{\sqrt{N\eta}}\right)^{2k}\preceq_{k} \mathop{Y^{2k}} .\]
So, the remaining quantity in the sum we need to bound is
\[\sum_{1 \leq l_{1},l_{2},\cdots,l_{2k}\leq f-1}\sum_{\textbf{b}}q_{b_{1}}q_{b_{2}}q_{b_{3}}\cdots q_{b_{4k}}\E \mathbf{1}\left(\mathcal{A}\right) \prod_{i=1}^{2k} X^{(l_{i})}_{b_{2i-1},b_{2i}}.\]
Moreover due to Cauchy--Schwarz inequality one can show that
\[\left|\E \mathbf{1}\left(\mathcal{A}\right) \prod_{i=1}^{2k} X^{(l_{i})}_{b_{2i-1},b_{2i}}\right|\leq \left|\E \prod_{i=1}^{2k} X^{(l_{i})}_{b_{2i-1},b_{2i}}\right|+\mathbb{P}(\mathcal{A}^{c})\left|\E \prod_{i=1}^{2k} X^{(l_{i})}_{b_{2i-1},b_{2i}}\right|^{2}.\]

So, we will work with the right hand side of the last inequality, meaning we won't focus anymore on the event $\mathcal{A}$. Moreover we will focus on the first summand of the right hand side of the inequality, since the second one can be treated analogously.

Next, we can transform the previously mentioned quantity in a more appropriate form. Firstly, note for each $i \in [k]$:
\[X_{b_{2i-1},b_{2i}}^{(l_{i})}=\sum_{\mathbf{a^{(i)}}}(D^{-\mathbf{1}_{j=1)}}ED^{-1})_{a^{(i)}_{j},a^{(i)}_{j+1}}\]
and similarly for $i \in [2k]\setminus [k]$
\[\bar{X}_{b_{2i-1},b_{2i}}^{(l_{i})}=\sum_{\mathbf{a^{(i)}}}(D^{-\mathbf{1}_{j=1)}}ED^{-1})_{a^{(i)}_{j},a^{(i)}_{j+1}},\]
where the sum is taken over all $\mathbf{a^{(i)}} \subseteq \mathbb{T}^{l_{i}+1}$, i.e., all the $l_{i}$-tuples with the restriction that $a^{(i)}_{1}=b_{2i}$ and $a^{(i)}_{l_{i}+1}=b_{2i+1}$. So, since this is true for all $i \in [2k]$ one can show:
\[\prod_{i=1}^{2k} X^{(l_{i})}_{b_{2i-1},b_{2i}}=\sum_{\mathbf{a}}\prod (D^{-\mathbf{1}_{(j=1)}}ED^{-1})_{a^{(i)}_{j},a^{(i)}_{j+1}}',\]
where the sum is taken over all $\mathbf{a}=(\mathbf{a^{1}},\mathbf{a^{2}}\cdots \mathbf{a^{2k}})$ and for $i \in [2k]\setminus[k] $, the $'$ denotes the conjugate.

Next set
\[\mathbf{E}_{[a_{j}^{(i)}],[a_{j+1}^{(i)}] }= \begin{bmatrix} E_{[a_{j}^{(i)}],[a_{j+1}^{(i)}]} & E_{[a_{j}^{(i)}]+N,[a_{j+1}^{(i)}]} \\ E_{[a_{j}^{(i)}],[a_{j+1}^{(i)}]+N} & E_{[a_{j}^{(i)}]+N,[a_{j+1}^{(i)}]+N}  \end{bmatrix},\]
where $[a_{j}^{(i)}]$ is the least positive integer which is equal to $a_{j}^{(i)}$ $\mod(N)$. Moreover set 
\[x(a_{j}^{(i)})= \mathbf{1}\left\{a_{j}^{(i)}=[a_{j}^{(i)}]+N\right\}+1 .\]

Furthermore, since $D^{-1}$ consists of zero at the non diagonal $2 \times 2$ blocks, one has that for $j\neq 1$ 
\begin{align}
    &(ED^{-1})_{a_{j}^{(i)},a_{j+1}^{(i)}}=\left( \mathbf{E}_{[a_{j}^{(i)}],[a_{j+1}^{(i)}]]}D^{-1}_{[a_{j+1}^{(i)}],[a_{j+1}^{(i)}]}\right)_{x(a_{j}^{(i)}),x(a_{j+1}^{(i)})}\\=&\left(\mathbf{E}_{[a^{(i)}_{j}],[[a^{(i)}_{j+1}]} \right)_{x(a_{j}^{(i)}),1}\left( D^{-1}_{[a_{j+1}^{(i)}],[a_{j+1}^{(i)}]}\right)_{1,x(a_{j+1}^{(i)})}
    +\left(\mathbf{E}_{[a^{(i)}_{j}],[[a^{(i)}_{j+1}]} \right)_{x(a_{j}^{(i)}),2}\left( D^{-1}_{a_{j+1}^{(i)},a_{j+1}^{(i)}}\right)_{2,x(a_{j+1}^{(i)})} 
    \end{align}
and similarly for $j=1$ 
\[(D^{-1}ED^{-1})_{a_{1}^{(i)},a_{2}^{(i)}}=\sum_{l=1}^{2}\sum_{m=1}^{2} \left(D^{-1}_{[a_{1}^{(i)}],[a_{1}^{(i)}]}\right)_{[a_{1}^{(i)}],l}\left(\mathbf{E}_{[a_{1}]^{(i)},[a_{2}]^{2}} \right)_{l,m} \left(D^{-1}_{[a_{2}^{(i)}],[a_{2}]^{(i)}}\right)_{m,[a_{2}^{(1)}]}.\]
So it is implied that
\begin{align}
    &\E \prod_{i=1}^{2k} X^{(l_{i})}_{b_{2i-1},b_{2i}}=\sum_{\mathbf{a}} \sum_{\mathbf{c}}\E \prod_{i=1}^{2k}\left( D^{-1}_{[a_{1}^{(i)}],[a_{1}^{(i)}]}\right)_{x(a_{1}^{(i)}),c^{i}_{1}}'\left(\mathbf{E}_{[a_{1}^{(i)}],[a_{2}^{(i)}]} \right)_{c^{(i)}_{1},c^{(i)}_{2}}' \left(D^{-1}_{[a_{2}^{(i)}],[a_{2}^{(i)}]}\right)_{c^{(i)}_{2},x(a_{2}^{(i)})}' \boldsymbol{\cdot}\\& \boldsymbol{\cdot} \prod_{j \neq 1} \left(\mathbf{E}_{[a^{(i)}_{j}],[[a^{(i)}_{j+1}]} \right)'_{x(a_{j}^{(i)}),c_{j+1}^{(i)}}\left( D^{-1}_{[a_{j+1}^{(i)}],[a_{j+1}^{(i)}]}\right)'_{c_{j+1}^{(i)},x(a_{j+1}^{(i)})}.
    \end{align}
 Here $\mathbf{c}$ is any subset of $\{1,2\}^{ (\sum_{i =1}^{2k}l_{i})+1}$.
 So since the entries of $D^{-1}$ are deterministic and for all $a_{j}^{(i)}$ the entries of $D^{-1}_{[a_{j}^{(i)}],[a_{j}^{[i]}]}$ are bounded by $|g_{[a_{j}^{(i)}]}|+|g_{-[a_{j}^{(i)}]}|$, it is true that 
 \[\left|\E \prod_{i=1}^{2k} X^{(l_{i})}_{b_{2i-1},b_{2i}}\right|\leq \left|\sum_{\mathbf{a}}
 \prod_{i,j}\left(|g_{[a_{j}^{(i)}]}|+|g_{-[a_{j}^{(i)}]}|\right) 
 \sum_{\mathbf{c}}  \E \prod_{i=1}^{2k}\left(\mathbf{E}_{[a_{1}^{(i)}],[a_{2}^{(i)}]} \right)_{c^{(i)}_{1},c^{(i)}_{2}}' \prod_{j \neq 1} \left(\mathbf{E}_{[a^{(i)}_{j}],[[a^{(i)}_{j+1}]} \right)'_{x(a_{j}^{(i)}),c_{j+1}^{(i)}}\right|.\]
 Next we will show an important inequality, necessary to estimate the expectation of the products in the previous equations.
 \begin{lem}\label{lima gia ginomena Eaij }
 It is true that for each array $(a^{i}_{j})$ with entries in $\mathbb{T}$,
\begin{align}
 &\left|\sum_{\mathbf{c}}  \E \prod_{i=1}^{2k}\left(\mathbf{E}_{[a_{1}^{(i)}],[a_{2}^{(i)}]} \right)_{c^{(i)}_{1},c^{(i)}_{2}}' \prod_{j \neq 1} \left(\mathbf{E}_{[a^{(i)}_{j}],[[a^{(i)}_{j+1}]} \right)'_{x(a_{j}^{(i)}),c_{j+1}^{(i)}}\right| \\ &\preceq_{k} \frac{\left(\psi \log(N)\right)^{\sum l_{i}} \left(s+\eta\right)^{\sum_{i}l_{i}}}{\left(N\eta\right)^{\sum_{i=1}l_{i}/2}}X\left([a_{1}^{(1)}], [a_{2}^{(1)}],[a_{2}^{1}],[a_{3}^{1}]\cdots,[a^{(1)}_{l_{1}+1}],[a^{(2)}_{1}]\cdots,[a^{(2)}_{l_{2}+1}]\cdots,[a^{(2k)}_{l_{2k}}]\right),
 \end{align}
where $X(\cdot)$ is the indicator function that indicates if every element in the array appears an even number of times. 
 \end{lem}
 \begin{proof}
Note that by the definition of the matrices, one has that $$\mathbf{E}_{[a^{(i)}_{j}],[a^{(i)}_{j+1}]}=\mathbf{E}^{1}_{[a^{(i)}_{j}],[a^{(i)}_{j+1}]}+\mathbf{E}^{2}_{[a^{(i)}_{j}],[a^{(i)}_{j+1}]}+\mathbf{E}^{3}_{[a^{(i)}_{j}],[a^{(i)}_{j+1}]}.$$ 
Moreover, after conditioning on the matrix $W'_{\mathbb{T},\mathbb{T}}$, the matrices $\mathbf{E}^{1},\mathbf{E}^{2},\mathbf{E}^{3}$ are independent since $\mathbf{E}^{1}$ is dependent only on $(F-z\mathbb{I})^{(\mathbb{T})}$ and is diagonal and deterministic, $\mathbf{E}^{2}$ depends only on $W'_{\mathbb{T},\mathbb{T}}$ and $\mathbf{E}^{3}$ depends on $G^{\mathbb{T}}$ and $W'_{[2N]\setminus \mathbb{T},\mathbb{T}} $. We will use the notation $\E_{\mathbb{T}}$ for the conditional expected value. So in order to prove Lemma \ref{lima gia ginomena Eaij } it suffices to show that 
\begin{equation}\label{Anisotita Eaij 1}
\begin{aligned}
  \left|\sum_{\mathbf{c}}   \E_{\mathbb{T}} \prod_{i=1}^{2k}\left(\mathbf{E}^{1}_{[a_{1}^{(i)}],[a_{2}^{(i)}]} \right)_{c^{(i)}_{1},c^{(i)}_{2}}' \prod_{j \neq 1} \left(\mathbf{E}^{1}_{[a^{(i)}_{j}],[[a^{(i)}_{j+1}]} \right)'_{x(a_{j}^{(i)}),c_{j+1}^{(i)}}\right| \preceq_{k} \left(\frac{\psi s}{N\eta } \right)^{\sum l_{i}} X\left( \left( [a^{i}_{j}] \right)_{ i \in [2k],j \in [2l_{i}]}\right)\leq \\ \leq X\left( \left( [a^{i}_{j}]\right)_{i \in [2k],j \in [l_{i}]}\right) \frac{(\psi \log(N))^{\sum l_{i}} (s+\eta)^{\sum_{i}l_{i}}}{(N\eta)^{\sum_{i=1}l_{i}/2}},
  \end{aligned}
\end{equation}
\begin{equation}\label{Anisotita Eaij 2}
\begin{aligned}
 \left|\sum_{\mathbf{c}}   \E_{\mathbb{T}} \prod_{i=1}^{2k}\left(\mathbf{E}^{2}_{[a_{1}^{(i)}],[a_{2}^{(i)}]} \right)_{c^{(i)}_{1},c^{(i)}_{2}}' \prod_{j \neq 1} \left(\mathbf{E}^{2}_{[a^{(i)}_{j}],[[a^{(i)}_{j+1}]} \right)'_{x(a_{j}^{(i)}),c_{j+1}^{(i)}}\right| \preceq_{k} \left(\frac{s}{N } \right)^{\sum l_{i}/2} X\left( \left( [a^{i}_{j}]\right)_{i \in [2k],j \in [l_{i}]}\right)\leq \\ 
 \leq X (\left( [a^{i}_{j}]\right)_{i \in [2k],j \in [l_{i}]}) \frac{(\psi \log(N))^{\sum l_{i}} (s+\eta)^{\sum_{i}l_{i}}}{(N\eta)^{\sum_{i=1}l_{i}/2}} , 
 \end{aligned}
\end{equation}
\begin{equation}\label{Anisotita Eaij 3}
\begin{aligned}
 \left|\sum_{\mathbf{c}}   \E_{\mathbb{T}} \prod_{i=1}^{2k}\left(\mathbf{E}^{3}_{[a_{1}^{(i)}],[a_{2}^{(i)}]} \right)_{c^{(i)}_{1},c^{(i)}_{2}}' \prod_{j \neq 1} \left(\mathbf{E}^{3}_{[a^{(i)}_{j}],[[a^{(i)}_{j+1}]} \right)'_{x(a_{j}^{(i)}),c_{j+1}^{(i)}}\right| \preceq_{k} \left(\frac{s\psi \log(N)}{\sqrt{N\eta}} \right)^{\sum l_{i}/2} X( \left( [a^{i}_{j}]\right)_{i \in [2k],j \in [l_{i}]})\leq \\ X \left(\left( [a^{i}_{j}]\right)_{i \in [2k],j \in [l_{i}]}\right) \frac{(\psi \log(N))^{\sum l_{i}} (s+\eta)^{\sum_{i}l_{i}}}{(N\eta)^{\sum_{i=1}l_{i}/2}}.  
\end{aligned}
\end{equation}
 This is true, since if we assume \eqref{Anisotita Eaij 1},\eqref{Anisotita Eaij 2},\eqref{Anisotita Eaij 3} hold for any array then
\begin{align}
&\left|\sum_{\mathbf{c}}  \E_{\mathbb{T}} \prod_{i=1}^{2k}\left(\mathbf{E}_{[a_{1}^{(i)}],[a_{2}^{(i)}]} \right)_{c^{(i)}_{1},c^{(i)}_{2}}' \prod_{j \neq 1} \left(\mathbf{E}_{[a^{(i)}_{j}],[[a^{(i)}_{j+1}]} \right)'_{x(a_{j}^{(i)}),c_{j+1}^{(i)}}\right|\\& 
\leq \sum_{P^{1},P^{2},P^{3}}\left|\sum_{\mathbf{c}}\prod_{y \in [3]}\E_{\mathbb{T}} \prod_{i=1}^{2k} \mathbf{1}\left\{a^{(i)}_{1},a^{(i)}_{2} \in P^{y}\right\}\left(\mathbf{E}^{y}_{a_{1}^{(i)}],[a_{2}^{(i)}]} \right)_{c^{(i)}_{1},c^{(i)}_{2}}' \prod_{j \neq 1} \mathbf{1}\left\{a_{j}^{i},a_{j}^{i} \in P^{y}\right\} \left(\mathbf{E}^{y}_{[a^{(i)}_{j}],[[a^{(i)}_{j+1}]} \right)'_{x(a_{j}^{(i)}),c_{j+1}^{(i)}}\right|\\& \preceq_{k} \sum_{P^{1},P^{2},P^{3}} \prod_{y \in [3]} X\left([a^{i}_{j}]:a_{j}^{i} \in P^{y} \right) \left((s+\eta)\frac{\psi \log(N)}{\sqrt{N\eta}}\right)^{|P^{y}|} \leq \sum_{P^{1},P^{2},P^{3}} X([a_{j}^{i}] ) \left((s+\eta)\frac{\psi \log(N)}{\sqrt{N\eta}}\right)^{\sum l_{i}} \\& \preceq_{k}X([a_{j}^{i}] ) \left((s+\eta)\frac{\psi \log(N)}{\sqrt{N\eta}}\right)^{\sum l_{i}|},   
\end{align}
where the sum is taken over all 3-partitions $P^{1},P^{2},P^{3}$ of the set $\{i,j:i \in [2k],j\in [l_{i}+1]\}$ such that if $a_{j}^{i} \in P^{y}$ then $a_{j+1}^{i} \in P^{y}$ for each $j \in [l_{i}]\cap (2\N+1)$, $i \in [2k]$ and $ y \in [3]$. So the number of these partitions depends only on $k$ which implies the last inequality. 

For the first inequality \eqref{Anisotita Eaij 1} note that the matrix $\mathbf{E}^{1}$ is diagonal and its diagonal entries are bounded by $s|\mathbb{T}|\frac{\psi }{N\eta}$ due to Theorem 4.5 of \cite{che2019universality} and the interlacing properties of the minors of the eigenvalues. So it is implied that $|\mathbf{E}^{1}_{[a_{j}^{i}],a_{j+1}^{i}]}|_{op}\preceq \mathbf{1}\left\{a^{j}_{i}=a_{j+1}^{i}\right\} s\frac{\psi }{N\eta}$. So

\begin{align}
& \sum_{\mathbf{c}}   \left|\E_{\mathbb{T}} \prod_{i=1}^{2k}\left(\mathbf{E}^{1}_{[a_{1}^{(i)}],[a_{2}^{(i)}]} \right)_{c^{(i)}_{1},c^{(i)}_{2}}' \prod_{j \neq 1} \left(\mathbf{E}^{1}_{[a^{(i)}_{j}],[[a^{(i)}_{j+1}]} \right)'_{x(a_{j}^{(i)}),c_{j+1}^{(i)}}\right|\leq \sum_{\mathbf{c}}   \E_{\mathbb{T}} \prod_{i=1}^{2k}\left|\mathbf{E}^{1}_{[a_{1}^{(i)}],[a_{2}^{(i)}]}  \prod_{j \neq 1} \mathbf{E}^{1}_{[a^{(i)}_{j}],[[a^{(i)}_{j+1}]}\right|_{op} \\& \preceq \sum_{\mathbf{c}} \prod_{i,j} \mathbf{1}\left\{a^{i}_{j}=a^{i}_{j+1}\right\} s \frac{\psi}{N\eta} \leq \sum_{\mathbf{c}} \prod_{i,j} \mathbf{1}\left\{[a^{i}_{j}]=[a^{i}_{j+1}]\right\} s \frac{\psi}{N\eta} \preceq_{k} \left(\frac{\psi s}{N\eta } \right)^{\sum l_{i}} X\left( \left( [a^{i}_{j}] \right)_{ i \in [2k],j \in [l_{i}+1]}\right). 
\end{align}
 
For the second inequality \eqref{Anisotita Eaij 2} by the way $\mathbf{E}^{2}$ was defined, one can compute that 
 \begin{align}
    &\left|\sum_{\mathbf{c}}   \E_{\mathbb{T}} \prod_{i=1}^{2k}\left(\mathbf{E}^{2}_{[a_{1}^{(i)}],[a_{2}^{(i)}]} \right)_{c^{(i)}_{1},c^{(i)}_{2}}' \prod_{j \neq 1} \left(\mathbf{E}^{2}_{[a^{(i)}_{j}],[[a^{(i)}_{j+1}]} \right)'_{x(a_{j}^{(i)}),c_{j+1}^{(i)}}\right|\\&
 =\left|\sum_{c_{1}^{i} \in \{1,2\}} \E_{\mathbb{T}} \prod_{i=1}^{2k}\left(\mathbf{E}^{2}_{[a_{1}^{(i)}],[a_{2}^{(i)}]} \right)_{c^{(i)}_{1},[c^{(i)}_{1}+1]_{2} } \prod_{j \neq 1} \left(\mathbf{E}^{2}_{[a^{(i)}_{j}],[[a^{(i)}_{j+1}]} \right)_{x(a_{j}^{(i)}),[x(a_{j}^{(i)})+1]_{2}}\right|, 
 \end{align}
 where $[a]_{2}$ the least positive integer such that it is equal to $a\mod 2$. Note that $\left( \mathbf{E}^{2}_{[a_{j}^{i}],[a_{j+1}^{i}]}\right)_{1,2}=-\sqrt{s} w_{a_{j}^{i},a^{i}_{j+1}}$ and $\left(\mathbf{E}^{2}_{[a_{j}^{i}],[a_{j+1}^{i}]}\right)_{2,1}=-\sqrt{s} w_{a^{i}_{j+1},a_{j}^{i}}$. Moreover, since the non zero entries of $W'$ are independent, symmetric, normal random variables with variance $\frac{1}{N}$, the product is non-zero only if every pair $([a^{i}_{j}],[a^{i}_{j+1}])$ in the product appears an even number of times. All these imply that
 \[\sum_{c_{1}^{i} \in \{1,2\}} \left|\E_{\mathbb{T}} \prod_{i=1}^{2k}\left(\mathbf{E}^{2}_{[a_{1}^{(i)}],[a_{2}^{(i)}]} \right)_{c^{(i)}_{1},[c^{(i)}_{1}+1]_{2} } \prod_{j \neq 1} \left(\mathbf{E}^{2}_{[a^{(i)}_{j}],[[a^{(i)}_{j+1}]} \right)_{x(a_{j}^{(i)}),[x(a_{j}^{(i)})+1]_{2}}\right|\preceq_{k}\sum_{c^{i}_{1} \in \{1,2\}} (\frac{s}{N})^{\sum l_{i}/2} X([a^{i}_{j}]),\]
 which implies \eqref{Anisotita Eaij 2}. The constant which is implied in the last inequality can be chosen to be $2k \prod_{j=1}^{\sum l_{i}} \E (\sqrt{N}w_{1,1})^{2j} $, which is a large constant depending only on k since $\sqrt{N}w_{1,1} \sim N(0,1)$.
 
 For the third inequality, \eqref{Anisotita Eaij 3}, one can show that
 \begin{align}
&\sum_{\mathbf{c}}   \E_{\mathbb{T}} \prod_{i=1}^{2k}\left(\mathbf{E}^{3}_{[a_{1}^{(i)}],[a_{2}^{(i)}]} \right)_{c^{(i)}_{1},c^{(i)}_{2}}' \prod_{j \neq 1} \left(\mathbf{E}^{3}_{[a^{(i)}_{j}],[[a^{(i)}_{j+1}]} \right)'_{x(a_{j}^{(i)}),c_{j+1}^{(i)}} \\&
 = \sum_{\{(c^{i}_{2j-1},c^{i}_{2j}) \in \mathbf{c}:a^{i}_{2j-1} \neq a^{i}_{2j} \}}\E_{\mathbb{T}} \left( \prod_{i=1}^{2k}\left(\mathbf{E}^{3}_{[a_{1}^{(i)}],[a_{2}^{(i)}]} \right)_{c^{(i)}_{1},c^{(i)}_{2}}'+\mathbf{1}\left\{a_{1}^{(i)}=a_{2}^{i}\right\}\left( \mathbf{E}^{3}_{[a_{1}^{(i)}],[a_{2}^{(i)}]} \right)_{[c^{(i)}_{1}+1]_{2},[c^{(i)}_{2}+1]_{2}}'\right)\boldsymbol{\cdot} \\&
 \boldsymbol{\cdot} \left( \prod_{j \neq 1} \left(\mathbf{E}^{3}_{[a^{(i)}_{j}],[[a^{(i)}_{j+1}]} \right)_{x(a_{j}^{(i)}),c^{i}_{j+1}}+\mathbf{1}\left\{a^{i}_{j}=a^{i}_{j+1}\right\} \left(\mathbf{E}^{3}_{[a^{(i)}_{j}],[[a^{(i)}_{j+1}]} \right)_{[x(a_{j}^{(i)})+1]_{2},[c^{i}_{j+1}+1]_{2}}\right)
 \\&=\sum_{\mathbf{c}}   \E_{\mathbb{T}} \prod_{i=1}^{2k}\left(\tilde{\mathbf{E}}^{3}_{[a_{1}^{(i)}],[a_{2}^{(i)}]} \right)_{c^{(i)}_{1},c^{(i)}_{2}}' \prod_{j \neq 1} \left(\tilde{\mathbf{E}}^{3}_{[a^{(iF)}_{j}],[[a^{(i)}_{j+1}]} \right)'_{x(a_{j}^{(i)}),c_{j+1}^{(i)}},
 \end{align}
 where 
 \[\left(\tilde{\mathbf{E}}^{3}_{[i],[j]}\right)_{l,m}=\begin{cases}
 \left(\mathbf{E}^{3}_{[i],[j]}\right)_{l,m}, & [i]\neq[j]  \text{ or } l\neq m \\ 
 \left(\mathbf{E}^{3}_{[i],[j]}\right)_{l,m} +\frac{s}{N} \mathbf{1}\left\{l=1\right\}\sum_{i \in [N]\cap [2N]\setminus \mathbb{T}}G^{\mathbb{T}}_{ii} +\frac{s}{N} \mathbf{1}\left\{l=2\right\}\sum_{i \in [N]\cap [2N]\setminus \mathbb{T}}^{N}G_{i+N,i+N}^{\mathbb{T}}, & \text{else}
 \end{cases}\]
 So by construction one can compute that
 \begin{equation}\label{Allagmeno E3}
 \left(\tilde{\mathbf{E}}^{3}\right)_{[i],[j]}= s\begin{bmatrix} \sum_{f,k \notin \mathbb{T}\cup 2N\setminus[N]}(w_{j,f}w_{k,i}-\mathbf{1}_{\{[i]=[j]\}}\mathbf{1}_{\{f=k\}}\frac{1}{N})G^{\mathbb{T}}_{k+N,f+N} & \sum_{f,k \notin \mathbb{T}\cup 2N\setminus[N]}w_{f,j}w_{k,i}G^{\mathbb{T}}_{k,f+N} \\ \sum_{f,k \notin \mathbb{T}\cup 2N\setminus[N]}w_{j,f}w_{i,k}G^{\mathbb{T}}_{k+N,f} & \sum_{f,k \notin \mathbb{T}\cup 2N\setminus[N]}(w_{f,j}w_{i,k}-\mathbf{1}_{\{[i]=[j]}\mathbf{1}_{\{f=k\}}\frac{1}{N})G^{\mathbb{T}}_{k,f} \end{bmatrix}
\end{equation}
As a result
\begin{align}
&\E_{\mathbb{T}}\sum_{\mathbf{c}}\prod_{i=1}^{2k}\left(\tilde{\mathbf{E}}^{3}_{[a_{1}^{(i)}],[a_{2}^{(i)}]} \right)_{c^{(i)}_{1},c^{(i)}_{2}}' \prod_{j \neq 1} \left(\tilde{\mathbf{E}}^{3}_{[a^{(i)}_{j}],[[a^{(i)}_{j+1}]} \right)'_{x(a_{j}^{(i)}),c_{j+1}^{(i)}}=\\
&s^{\sum l_{i}}\sum_{\mathbf{c}}\sum_{\beta_{1}^{1},\beta^{1}_{2}\cdots..\beta_{2l_{2k}}^{2k} \notin \mathbb{T}}\E_{\mathbb{T}}\prod_{i=1}^{2k}\left[ \left(W'_{[a^{i}_{1}],[\beta^{i}_{1}]}\right)_{[c^{i}_{1}+1]_{2},c^{i}_{1}}\left(W'_{[a^{i}_{2}],[\beta^{i}_{2}]}\right)_{c^{i}_{2},[c^{i}_{2}+1]_{2}}-\frac{\mathbf{1}_{a_{1}^{i}=\beta_{1}^{i}}\mathbf{1}_{a_{2}^{i}=\beta_{2}^{i}}}{N}\right]\boldsymbol{\cdot}
\\&
\boldsymbol{\cdot} \prod_{j \neq 1} \left[ \left(W'_{[a^{i}_{j}],[\beta^{i}_{j}]}\right)_{[x(a_{j}^{i})+1]_{2},x(a_{j}^{i})}\left(W'_{[a^{i}_{j+1}],[\beta^{i}_{j+1}]}\right)_{c^{i}_{j},[c^{i}_{j+1}+1]_{2}}-\frac{\mathbf{1}_{a_{j}^{i}=\beta_{j}^{i}}\mathbf{1}_{a_{j+1}^{i}=\beta_{j+1}^{i}}}{N}\right] \boldsymbol{\cdot}
\\&
\boldsymbol{\cdot} \prod_{i=1}^{2k}\left(G^{\mathbb{T}}_{[\beta_{1}^{i}],[\beta_{2}^{i}]}\right)_{[c_{1}^{i}+1]_{2},[c_{2}^{i}+1]_{2}}\prod_{j\neq 1} \left(G^{\mathbb{T}}_{[\beta_{j}^{i}],[\beta_{j+1}^{i}]}\right)_{[x(a_{j}^{i})+1]_{2},[c_{j}^{i}+1]_{2}}.
\end{align}

Next set $\mathcal{G}$ to be the graph with vertices $\{[\beta_{1}^{1}],[\beta_{2}^{1}],\cdots\cdots[\beta_{2l_{2k}}^{2k}]\}$ and with edges $([\beta^{i}_{2j-1}],[\beta^{i}_{2j}])$. Set $\rho(\mathcal{G})$ the indicator function that every vertex of $\mathcal{G}$ is adjacent to at least two edges, $v=\{[\beta_{1}^{1}],[\beta_{2}^{1}],\cdots[\beta_{2l_{2k}}^{2k}]\}$, $\gamma_{r \in [v]}$ the non-repeating vertices of $\mathcal{G}$, $d_{r}$ the multiplicity of $\gamma_{r}$ and $o$ the number of self loops in $\mathcal{G}$. So, by \eqref{Fragma gia antidiagonio bgazontas T}, \eqref{fragma gia diagonious bgazontas T} and \eqref{fragma gia ta ypolloipa stoixeia Bgazontas T} one has that with overwhelming probability 
\begin{align}
&\rho(\mathcal{G})\left|\prod_{i=1}^{2k}\left(G^{\mathbb{T}}_{[\beta_{1}^{i}],[\beta_{2}^{i}]}\right)_{[c_{1}^{i}+1]_{2},[c_{2}^{i}+1]_{2}}\prod_{j\neq 1} \left(G^{\mathbb{T}}_{[\beta_{j}^{i}],[\beta_{j+1}^{i}]}\right)_{[x(a_{j}^{i})+1]_{2},[c_{j}^{i}+1]_{2}}\right|
\\&\preceq_{k} \rho(\mathcal{G}) \frac{\psi^{\sum l_{i}}}{\sqrt{N\eta}^{\sum l_{i} -o}} \prod_{r \in [v]}\left(|g_{\gamma_{r}}|^{d_{r}/2}+|g_{-\gamma_{r}}|^{d_{r}/2}\right).
\end{align}
 Thus one can show similarly to the proof of (2.40) in \cite{bourgade2017eigenvector} that the following holds with overwhelming probability
\begin{equation}\label{anisotita gia ta Gt lima gia fragma toy E3}
  \rho(\mathcal{G})\sum_{\beta_{j}^{i} \notin \mathbb{T}} \left|\prod_{i=1}^{2k}\left(G^{\mathbb{T}}_{[\beta_{1}^{i}],[\beta_{2}^{i}]}\right)_{[c_{1}^{i}+1]_{2},[c_{2}^{i}+1]_{2}}\prod_{j\neq 1} \left(G^{\mathbb{T}}_{[\beta_{j}^{i}],[\beta_{j+1}^{i}]}\right)_{[x(a_{j}^{i})+1]_{2},[c_{j}^{i}+1]_{2}}\right|\preceq_{k}\frac{(\psi \log(N))^{\sum l_{i}}N^{\sum l_{i}/2}}{\eta^{\sum l_{i}/2}}.
\end{equation}
Next we need to bound the quantity
\begin{align}
&\E_{\mathbb{T}}\prod_{i=1}^{2k}\left[ \left(W'_{[a^{i}_{1}],[\beta^{i}_{1}]}\right)_{[c^{i}_{1}+1]_{2},c^{i}_{1}}\left(W'_{[a^{i}_{2}],[\beta^{i}_{2}]}\right)_{c^{i}_{2},[c^{i}_{2}+1]_{2}}-\frac{\mathbf{1}_{a_{1}^{i}=\beta_{1}^{i}}\mathbf{1}_{a_{2}^{i}=\beta_{2}^{i}}}{N}\right] \boldsymbol{\cdot}
\\& \boldsymbol{\cdot} \prod_{j \neq 1} \left[ \left(W'_{[a^{i}_{j}],[\beta^{i}_{j}]}\right)_{[x(a_{j}^{i})+1]_{2},x(a_{j}^{i})}\left(W'_{[a^{i}_{j+1}],[\beta^{i}_{j+1}]}\right)_{c^{i}_{j},[c^{i}_{j+1}+1]_{2}}-\frac{\mathbf{1}_{a_{j}^{i}=\beta_{j}^{i}}\mathbf{1}_{a_{j+1}^{i}=\beta_{j+1}^{i}}}{N}\right].
\end{align}

Note, that in order for the product to be different than 0, every pair $([a_{i}^{j}],[b_{i}^{j}])$ must appear an even number of times. Moreover in order for the product  to be different than 0, for each $i,j$ the number of consecutive pairs $([a^{m}_{r}],[\beta^{m}_{r}])$ and $([a^{m}_{r+1}],[\beta^{m}_{r+1}])$ for $m \in [2k]$ and $r \in [l_{m}]$, such that exactly one of them is equal to $[a_{j}^{i},\beta_{j}^{i}]$, must be also even. Furthermore, for each $i,j$, if such pairs do not exist, then the number of consecutive pairs which are both equal to $[a_{i}^{j},\beta_{i}^{j}]$ must be at least 2,
or else the product would be 0. The latter is true since either the square of a centered Gaussian random variable minus its variance would appear, either the product of two independent centered Gaussian random variables would appear.

So it is implied that in order for the product above to not be zero, it is demanded that $\rho(\mathcal{G})=1$ and $X([a_{i}^{j}])=1$. Here $\mathcal{G}$ is the graph which is associated with $\beta_{i}^{j}$. So by a trivial bounding in the moments of Gaussian random variables one can show that
\begin{align}\label{anisotita gia ta W' }
 &\bigg| \E_{\mathbb{T}}\prod_{i=1}^{2k}\left[ \left(W'_{[a^{i}_{1}],[\beta^{i}_{1}]}\right)_{[c^{i}_{1}+1]_{2},c^{i}_{1}}\left(W'_{[a^{i}_{2}],[\beta^{i}_{2}]}\right)_{c^{i}_{2},[c^{i}_{2}+1]_{2}}-\frac{\mathbf{1}_{a_{1}^{i}=\beta_{1}^{i}}\mathbf{1}_{a_{2}^{i}=\beta_{2}^{i}}}{N}\right]\boldsymbol{\cdot}  \\& \boldsymbol{\cdot} \prod_{j \neq 1} \left[ \left(W'_{[a^{i}_{j}],[\beta^{i}_{j}]}\right)_{[x(a_{j}^{i})+1]_{2},x(a_{j}^{i})}\left(W'_{[a^{i}_{j+1}],[\beta^{i}_{j+1}]}\right)_{c^{i}_{j},[c^{i}_{j+1}+1]_{2}}-\frac{\mathbf{1}_{a_{j}^{i}=\beta_{j}^{i}}\mathbf{1}_{a_{j+1}^{i}=\beta_{j+1}^{i}}}{N}\right]\bigg| \preceq_{k} N^{-\sum l_{i}} \rho(\mathcal{G})X([a^{i}_{j}]).
\end{align}  

Thus, the proof of the lemma is complete after combining \eqref{anisotita gia ta W' } and \eqref{anisotita gia ta Gt lima gia fragma toy E3}.
 \end{proof}
 We are now ready to present the proof of Theorem \ref{Theorima Local pertubed gia diagonious}.
 \begin{proof}[Proof of Theorem \ref{Theorima Local pertubed gia diagonious}]\label{proof of theorem diagonious}
      Note that Lemma \ref{lima gia ginomena Eaij } holds for every sequence of indexes. In our case though, by construction, every term in $[a^{i}_{j}]$ appears a non-zero even number of times since they appear consecutive times for $j \neq 1,l_{i}+1 $. So one has that $X([a_{i}^{j}]_{i \in [2k],j \in [l_{i}+1]})=X([a_{i}^{j}]_{i \in [2k],j \in \{1, l_{i}+1\}})=X([\mathbf{B)}]$. So by a direct application of Lemma \ref{lima gia ginomena Eaij }
 \begin{align}&
\sum_{\mathbf{a}} \sum_{\mathbf{c}}\bigg|\E \prod_{i=1}^{2k}\left( D^{-1}_{[a_{1}^{(i)}],[a_{1}^{(i)}]}\right)_{x(a_{1}^{(i)}),c^{i}_{1}}'\left(\mathbf{E}_{[a_{1}^{(i)}],[a_{2}^{(i)}]} \right)_{c^{(i)}_{1},c^{(i)}_{2}}' \left(D^{-1}_{[a_{2}^{(i)}],[a_{2}^{(i)}]}\right)_{c^{(i)}_{2},x(a_{2}^{(i)})}' \mathbf{\cdot}\\& \mathbf{\cdot} \prod_{j \neq 1} \left(\mathbf{E}_{[a^{(i)}_{j}],[[a^{(i)}_{j+1}]} \right)'_{x(a_{j}^{(i)}),c_{j+1}^{(i)}}\left( D^{-1}_{[a_{j+1}^{(i)}],[a_{j+1}^{(i)}]}\right)'_{c_{j+1}^{(i)},x(a_{j+1}^{(i)})}\bigg|
 \\& \preceq_{k} \sum_{\mathbf{a}}
 \prod_{i,j}\left(|g_{[a_{j}^{(i)}]}|+|g_{-[a_{j}^{(i)}]}|\right)\frac{(\psi \log(N))^{\sum l_{i}} (s+\eta)^{\sum_{i}l_{i}}}{(N\eta)^{\sum_{i=1}l_{i}/2}}X(\mathbf{B}).
 \end{align}
 As a result
 \begin{align}
& \sum_{1 \leq l_{1},l_{2},\cdots,l_{2k}\leq f-1}\sum_{\mathbf{B}}\sum_{b_{i}=\{ B_{i},B_{i}+N \}}|q_{b_{1}}||q_{b_{2}}||q_{b_{3}}|\cdots|q_{b_{4k}}|\left|\E \prod_{i=1}^{2k} X^{(l_{i})}_{b_{2i-1},b_{2i}}\right|\\&
 \preceq_{k}\sum_{1 \leq l_{1},l_{2},\cdots,l_{2k}\leq f-1}\sum_{\mathbf{B}}\frac{(\psi \log(N))^{\sum l_{i}} (s+\eta)^{\sum_{i}l_{i}}}{(N\eta)^{\sum_{i=1}l_{i}/2}}X(\mathbf{B})\sum_{b_{i}=\{ B_{i},B_{i}+N \}}|q_{b_{1}}||q_{b_{2}}||q_{b_{3}}|\cdots|q_{b_{4k}}| \sum_{\mathbf{a}}
 \prod_{i,j}\left(|g_{[a_{j}^{(i)}]}|+|g_{-[a_{j}^{(i)}]}|\right)
 \\&\label{endiameso fragma prin ta partitions gia isotropic local law}
\preceq \sum_{1 \leq l_{1},l_{2},\cdots,l_{2k}\leq f-1}\sum_{\mathbf{B}}\frac{(\psi \log(N))^{\sum l_{i}} (s+\eta)^{\sum_{i}l_{i}}}{(N\eta)^{\sum_{i=1}l_{i}/2}}X(\mathbf{B})\prod_{i=1}^{4k}\left(|q_{B_{i}}|+|q_{B_{i}+N}|\right)\sum_{\mathbf{A}}
 \prod_{i,j}\left(|g_{[a_{j}^{(i)}]}|+|g_{-[a_{j}^{(i)}]}|\right),
 \end{align}
 where the sum now is considered over all $\mathbf{A} \subseteq \mathbf{B}^{\sum l_{i}+2k}$  with the restriction that $[a^{i}_{1}]=B_{2i-1}$ and $[a^{i}_{l_{i}+1}]=B_{2i}$.
 
 Moreover note that the array $[a_{j}^{i}]$ defines a partition on the set $\{(i,j): i \in [2k], j \in [l_{i}+1]\}$ such that $(i,j)$ belongs to the same block of the partition with $(i',j')$ if and only if $a^{i}_{j}=a_{i'}^{j'}$. Furthermore, denote $n=|\mathbf{B}|$, $d_{i}$ the number of times the $i-th$ element of $\mathbf{B}$, which we denote with $\gamma_{i}$, appears without repetition and $r_{i}$ such that $r_{i}+d_{i}$ is the number of times the $i-th$ element of $\mathbf{B}$ appears in $\mathbf{A}$. Note that since we are interested in the sequences that $X(\mathbf{B})=1$, it is implied that $d_{i}$ are all even. So it is true that
 \[\sum d_{i}=2k, \ \ \ \ 2k + \sum r_{i}= \sum l_{i}.\]

Moreover, notice that each induced partition mentioned before, uniquely determines the quantities $d_{i},l_{i}$ and each block of the partition has at least two elements, since $X(\mathbf{B})=1$. So we can modify the sum, into first summing over all partitions $P$ and then over all $\mathbf{A}$-possible choices in the partition. Note that $\mathbf{B}$ is completely described by the set $\mathbf{A}$. So one has that
 \begin{align}&
 \eqref{endiameso fragma prin ta partitions gia isotropic local law}= \sum_{1\leq l_{1},l_{2}\cdots,l_{2k}\leq f-1}\frac{(\psi \log(N))^{\sum l_{i}} (s+\eta)^{\sum_{i}l_{i}}}{(N\eta)^{\sum_{i=1}l_{i}/2}} \sum_{P} \sum _{\mathbf{A} \sim P} X\left(\{[a^{i}_{j}]\}_{i \in [2k], j \in \{1,l_{i}+1\}}\right) \boldsymbol{\cdot} \\&
 \boldsymbol{\cdot}  \prod_{i=1}^{2k}\left(|q_{[a^{i}_{1}]}|+|q_{[a^{i}_{1}]+N}|\right)\left(|q_{[a^{i}_{l_{i}+1}]}|+|q_{[a^{i}_{l_{i}+1}]+N}\right)\prod_{_{i,j}}\left(|g_{[a_{j}^{(i)}]}|+|g_{-[a_{j}^{(i)}]}|\right) \preceq_{k}\\&
 \preceq_{k} \sum_{1\leq l_{1},l_{2}\cdots,l_{2k}\leq f-1} \frac{(\psi \log(N))^{\sum l_{i}} (s+\eta)^{\sum_{i}l_{i}}}{(N\eta)^{\sum_{i=1}l_{i}/2}} \sum_{P} \sum _{1 \leq \gamma_{1},\gamma_{2}\cdots,\gamma_{n}\leq N} \prod_{i=1}^{n}(|q_{\gamma_{i}}|^{d_{i}}+|q_{\gamma_{i}+N}|^{d_{i}})(|g_{-\gamma_{i}}|^{d_{i}+r_{i}}+|g_{\gamma_{i}}|^{d_{i}+r_{i}})
 \\&
 \preceq_{k} \sum_{1\leq l_{1},l_{2}\cdots,l_{2k}\leq f-1} \sum_{P}\frac{(\psi \log(N))^{\sum l_{i}} (s+\eta)^{\sum_{i}l_{i}}}{(N\eta)^{\sum_{i=1}l_{i}/2}}\frac{\operatorname{Im}\left(\sum (q^{2}_{i}+q^{2}_{i+N})(g_{i}+g_{-i})^{2k}\right)}{(s+\eta)^{\sum l_{i}}}\preceq_{k}\mathop{Y^{2k}},
 \end{align}
where in the last inequality we used the fact that $\psi \log(N) (\sqrt{N\eta})^{-1}\leq 1$, the fact that $\sum l_{i}\geq 2k$, and the fact that both the number of partitions and the number of possible $l_{i}$ are bounded by constants depending only on $k$. For the second to last inequality, we used Proposition 2.18-inequality (2.38) in \cite{bourgade2017eigenvector}, the facts that $d_{i}\geq 2$ and that $\sum d_{i}=2k$.
\\This finishes the proof of Theorem \ref{Theorima Local pertubed gia diagonious}
\end{proof}

\subsection{Bounding the pertubed matrices at the optimal scale}
At this subsection we are going to essentially bound the entries of the resolvent $G(s,z)$ at the optimal scale $\operatorname{Im}(z)=N^{\epsilon-1}$, for all matrices $\tilde{V}$  that are initially bounded by an $N-$dependent parameter. Next, we will apply this result to the matrix $X$, which is initially bounded due to \eqref{fragma gia Rii pano se ola ta z} with high probability. Thus we will prove that the matrix $X$, after slightly perturbing it, has essentially bounded resolvent entries at the optimal scale $N^{\delta-1}$, for any small enough, positive $\delta$. 
\begin{prop}\label{analogo protashs 3.9 apo goe}
Let $V$ be an $N \times N$ matrix and consider $\tilde{V}$ the symmetrization of $D$. Suppose that  $\tilde{V}$ satisfies Assumption \ref{Assumption for local pertubed } for some parameters $h_{*},r$ at energy level $E_{0}=0$ and that there exists an $N-$dependent parameter $B \in (0,\frac{1}{h_{*}})$ such that $\max_{j}|(\tilde{V}-zI)^{-1}_{j,j}|\leq B$. Then for any 
$\delta>0$ and $s : N^{\delta} h_{*} \leq s \leq r N^{-\delta}$, it is true that for any $D>1$ there exists $C=C(\delta,D)$ such that 
\[\mathbb{P}\left(\sup_{\mathbb{D}}\sup_{i,j}|G(s,z)| \geq B N^{\delta}\right) \leq C N^{-D},\]
where $G(s,z)=(\tilde{V}+\sqrt{s}W-z\mathbb{I})^{-1}$, $W$ is the symmetrization of an i.i.d. Gaussian matrix with centered entries and variance $\frac{1}{N}$ and $\mathbb{D}=\{E+i\eta: E \in (-\frac{r}{2},\frac{r}{2}),\eta \in [N^{ \delta -1},1- \frac{r}{2}]\}$. 
\begin{proof}
By a direct application of Theorem \ref{Theorima Local law gia pertubed} for $q_{k}=\mathbf{1}\left\{i=k\right\}$ for any $k \in [2N]$ (without loss of generality suppose $k \in [N]$), one has that with overwhelming probability uniformly on $\mathbb{D}$ it is true that,
\begin{align}
&|G_{k,k}(s,z)|\preceq \sum_{i=-N}^{N}\left(|g_{i}|+|g_{-i}|\right) \langle u_{i+N}(0) ,q_{k}\rangle^{2} + \sum_{i=1}^{2N} \left(|g_{i}|+|g_{-i}|\right) \left|\langle u_{i}(0),q_{k}\rangle\right| \left|\langle u_{i+N}(0),q_{k}\rangle\right|\\&
+ \frac{N^{ \delta/2}}{\sqrt{N\eta}}\operatorname{Im}(\sum_{i=1}^{N}(g_{i}+g_{-i})(\langle u_{i},q_{k}\rangle+\langle u_{i+N},q_{k}\rangle) . 
\end{align}
Note that by definition the $k-th$ element of each of the columns/rows of $U$ is 0 for all the columns/rows with index larger that $N$. Moreover by definition $N^{\delta}\leq N\eta$. So it is implied that the above bound becomes 
\[|G_{k,k}(s,z)| \preceq \sum_{i=1}^{N}(|g_{i}+|g_{-i}|) |u_{k,i}|.\]
Furthermore, due to Schur's complement formula, one can prove,  as in Lemma \ref{iid R_i, Schur}, that
\[G_{k,k}=z \tilde{G}_{k,k}(z^{2}),\]
where $\tilde{G}$ is the resolvent of the matrix $D^{T}D$. Moreover, one may compute that $$D^{T}D=\left((U)_{i \in [N],j \in [N]}\right)^{T} \Sigma^{2}(U)_{i \in [N],j \in [N]},$$where $\Sigma$ is the diagonal matrix with the singular values of $D$. So it is true that
\[\left|\sum_{i=1}^{N}u_{k,i}^{2}\frac{1}{\lambda_{i}^{2}-z^{2}}\right|=\left| \tilde{G}_{k,k}(z^{2})\right|= \left|\frac{1}{z}G_{k,k}(z)\right| \preceq \frac{B}{|z|} ,\] 
 with overwhelming probability uniformly on $z \in \{z=E+i\eta: E \in (-r,r), h_{*}\leq \eta\leq 1\}$. Thus if we consider the sets $\mathcal{A}_{m}(0)=(2^{m-1}h_{*},2^{m}h_{*})\cup(-2^{m}h_{*},-2^{m-1}h_{*}) $ and set $z=i\eta_{*}$ it is true that 
 \begin{equation}\label{fragma gia idioidanismata}
 \max_{j \in [N]}\left|\sum_{\lambda_{i}\in \mathcal{A}_{m}}u_{k,i}^{2}\right|\preceq \min \{B\eta^{2}_{*}2^{4m},1 \},
 \end{equation}
 where the bounding by $1$ is true due to the fact that the eigenvectors are considered normalized. After this observation the proof continues in a completely analogous way to the Proof of Proposition 3.9 in \cite{aggarwal2021goe} and so it is omitted.

 \end{proof}
\end{prop}
\begin{cor}\label{Veltisto Fragma gia to X+sqrtG}
Adopt the notation of Section \ref{section local}. Let $\mathcal{A}$ be the set mentioned in Theorem \ref{To theorima gia to local}. For all $a \in (0,2) \setminus \mathcal{A}$ consider the matrix $X+\sqrt{t}W$, where $t=t(N)$ is defined in Definition \ref{orismos t}. Set $\{T\}_{i,j \in [2N]}(z)$ the resolvent of $X+\sqrt{t}W$ at $z$. Then it is true
that for any $D>0$ and $\delta>0$, there exists a constant $C'=C'(a,\nu,\rho,\delta,D)$ such that
\begin{equation}\label{anisotita gia veltisto fragma X+sqrt(t)G}
\mathbb{P}\left(\sup_{\mathbb{D}_{\delta}}\sup_{i,j}|T_{i,j}(z)|\geq N^{\delta}\right) \leq C'N^{-D} ,
\end{equation}
where $\mathbb{D}_{C_{a},\delta}=\{E+i\eta: E \in (-\frac{1}{2C_{a}},\frac{1}{2C_{a}}),\eta \in [N^{\delta-1}, \frac{1}{4C_{a}}]\}$ where $C_{a}$ is the constant mentioned in Theorem \ref{To theorima gia to local}. 
\end{cor}
\begin{proof}
Due to \eqref{fragma gia Rii pano se ola ta z}, \eqref{local law mesa sto theorima} and since $t$ belongs to the desired interval $(N^{2\delta-\frac{1}{2}},N^{-2\delta})$, as is mentioned in the proof of Corollary \ref{universality gia t}, the proof of Corollary \ref{Veltisto Fragma gia to X+sqrtG} is just an application of Proposition \ref{analogo protashs 3.9 apo goe} to our set of matrices.
\end{proof}
\begin{rem}\label{eigenvector delocalization gia X+sqrtG}
 Note that bounding the entries of the resolvent of $X+\sqrt{t}W$ as we did in Corollary \ref{Veltisto Fragma gia to X+sqrtG} at scale $N^{\delta-1}$, implies the complete eigenvector delocalization in the sense of Theorem \ref{TO THEORIMA}. The proof of the latter claim is well-known and can be found in the proof of Theorem 6.3 in \cite{huang2015bulk}.
\end{rem}
\section{Establishing universality of the least singular value and eigenvector delocalization}\label{section pou apodiknietai to theorima}
Thus far, we have proven both universality of the least singular value, Corollary \ref{universality gia t}, and complete eigenvector de-localization, Remark \ref{eigenvector delocalization gia X+sqrtG}, for the matrix $X+\sqrt{t}W$ in the sense of Theorem \ref{TO THEORIMA}. What we need to prove next, is that the transition from $X+\sqrt{t}W$ to $X+A$ is smooth enough to preserve both the eigenvector delocalization and universality of the least singular value.
A first step to that direction is Theorem \ref{theorima 3.15 goe}, whose proof is more or less the same as its symmetric counterpart in \cite{aggarwal2021goe}. Furthermore what we manage, is to extend Theorem 3.15 of \cite{aggarwal2021goe} to its "integrating analogue" in Proposition \ref{protasi gia q gia kathe gamma}, which is not very difficult given Theorem 
\ref{theorima 3.15 goe}.
Proposition \ref{protasi gia q gia kathe gamma} is the milestone for the comparison of the least positive eigenvalues of $X+A$ and $X+\sqrt{t}W$.

Firstly, we will use a convenient decomposition of the elements of $H$ in order to express the dependence of the "small" and the "large" entries of $H$ with Bernoulli random variables.
\begin{defn}
Define the following random variables for $i,j \in [2N]: |i-j|\geq N$
\[\psi_{i,j}=\mathbb{P}\left(|h_{i,j}|\geq N^{-\rho}\right), \ \ \ \ \ x_{i,j}=\frac{\mathbb{P}\left(|h_{i,j}|\in (N^{-\nu}, N^{-\rho}) \right)}{\mathbb{P}\left(|h_{i,j}|\leq N^{-\rho}\right)}\]
and
\[\mathbb{P}[a_{i,j}\in I)=\frac{\mathbb{P}(h_{i,j} \in I \cap (-N^{-\nu},N^{-\nu})]} {\mathbb{P}(|H_{i,j}| <N^{-\nu}))}, \ \ \ \ \mathbb{P}(c_{i,j} \in I)= \frac{\mathbb{P}(h_{i,j} \in (-\infty,N^{-\rho})\cup (N^{\rho}, \infty ) \cap I)}{\mathbb{P}(|h_{i,j}| \geq N^{-\rho})},\]
\[\mathbb{P}(b_{i,j} \in I)=\frac{ \mathbb{P}(|h_{i,j}|\in (N^{-\nu},N^{-\rho})\cap h_{i,j} \in I)}{\mathbb{P}(|h_{i,j}| \in [N^{-\nu},N^{-\rho}])},\]
for any interval $I$, subset of $\R$. 

Moreover we define each bunch of 
\begin{align}
&\{x_{i,j}\}_{i,j \in 2N:|i-j|\geq N},\{ \psi_{i,j}\}_{i,j \in [2N]: |i-j| \geq N},\{a_{i,j}\}_{i,j \in 2N:|i-j|\geq N},
\\& \{b_{i,j}\}_{i,j \in 2N:|i-j|\geq N},\{c_{i,j}\}_{i,j \in 2N:|i-j|\geq N}
\end{align}
to be independent up to symmetry and independent amongst them for different indexes $i,j$. 
\end{defn}
\begin{defn}
Define the following matrices
\begin{align}
&A_{i,j}=\begin{cases}(1-\psi_{i,j})(1-x_{i,j})a_{i,j}, & i,j: |i-j|\geq N \\ 0, & \text{otherwise}
\end{cases}
\\&B_{i,j}=\begin{cases}(1-\psi_{i,j})x_{i,j}b_{i,j}, & i,j: |i-j|\geq N \\ 0, & \text{otherwise}\end{cases}
\\& C_{i,j}=\begin{cases}\psi_{i,j}c_{i,j}, & i,j: |i-j|\geq N \\ 0, & \text{otherwise}
\end{cases}
\\&\Psi_{i,j}=\begin{cases}\psi_{i,j}, & i,j: |i-j|\geq N \\ 0, & \text{otherwise}
\end{cases}
\end{align}
Note that by definition $H=A+B+C$ and $X=B+C$.
\end{defn}
Next we define the way to quantify the transition from $X+\sqrt{t}W$ to $X+A$.
\begin{defn}
Define the matrices
\[H^{\gamma}=\gamma A+\sqrt{t}(1-\gamma^{2})^{1/2}W +X, \ \ \ \ \ \ \ \ \gamma \in [0,1]\]
and $G^{\gamma}(z)=(H^{\gamma}-z\mathbb{I})^{-1}$.
\end{defn}
\subsection{Green function Comparison}
Next we present a comparison theorem for the resolvent entries of $H^{\gamma}$.
\begin{thm}\label{theorima 3.15 goe}
Let $a,b,\rho,\nu$ be constants that satisfy \eqref{statheres}. Additionally suppose that $a \in (0,2)\setminus \mathcal{A}$ as in Theorem \ref{local law}. Moreover let $F:\R \rightarrow \R $ such that 
\begin{equation}\label{Ypothesi gia F}
\sup_{|x| \leq 2N^{\epsilon}}\left|F^{(\mu)}(x)\right| \leq N^{C_{0}\epsilon}, \ \ \ \ \sup_{|x| \leq 2 N^{2} }\left|F^{(\mu)}(x)\right| \leq N^{C_{0}},
\end{equation}
for some absolute constant $C_{0}>0$, for some integer $n=n(a,b,\rho,\nu,C_{0})$ sufficiently large and any $ \epsilon>0$ and $\mu \in [n]$. Furthermore fix $z=E+i\eta$ for $E \in \R$ and $\eta \geq N^{-2}$. Moreover for any matrix $\Psi$ denote $\E_{\Psi}$ the conditional expectation with respect to $\Psi$. Set
\begin{align}
&\Xi(z) = \sup_{\gamma \in [0,1]}\max_{\mu \in [n]}\max_{i,j \in [2N]}\E_{\Psi} \left|F^{(\mu)} \operatorname{Im}(G^{\gamma}_{i,j}(z) )\right|,
\\&\Omega_{0}(z,\epsilon)= \left\{ \sup_{i,j}|G^{\gamma} _{i,j}(z)| \leq N^{\epsilon}\right\} \text{ ,   } Q_{0}(z,\epsilon)=1- \mathbb{P}_{\Psi}\left(\Omega_{0}(z,\epsilon)\right).
\end{align}
Then there exist $\epsilon=\epsilon(a,b,\rho,\nu)$ and $\omega=\omega(a,b,\rho,\nu)$ such that for any matrix $\Psi$ with at most $N^{1+a\rho+\epsilon}$ non-zero entries, there exists a constant $C=C(a,\nu,\rho)$ so that
\begin{equation}\label{anisotita apo theorima 3.15 goe}
  \sup_{\gamma \in [0,1]}\left|\E_{\Psi}F\left(\operatorname{Im}(G_{i,j}^{\gamma}(z))\right)-\E_{\Psi}F\left(\operatorname{Im}(G_{i,j}^{0}(z))\right)\right|\leq C N^{-\omega}\left(\Xi(z)+1\right)+CQ_{0}(z,\epsilon)N^{C+C_{0}}, \text{ for all } i,j \in [2N].
\end{equation}

A similar bound to \eqref{anisotita apo theorima 3.15 goe} can be proven, if one replaces $\operatorname{Im}(G^{\gamma}_{i,j}(z))$ and $\operatorname{Im}(G_{i,j}^{0}(z))$ with $\operatorname{Re}(G_{i,j}^{\gamma}(z)))$ and $\operatorname{Re}(G_{i,j}^{0}(z)))$ respectively.
\end{thm}
\begin{proof}
The proof is similar to the proof of Theorem 3.15 in \cite{aggarwal2021goe}. Next we give a short description of the main ideas behind the proof. We will do so only for the imaginary parts $\operatorname{Im}(G^{\gamma}_{i,j}(z))$. The proof for the real parts $\operatorname{Re}(G^{\gamma}_{i,j}(z))$ is completely analogous.
\\Fix $z \in \C$ and $F :\R \rightarrow \R$ satisfying the hypothesis of Theorem \ref{theorima 3.15 goe}.
\\Firstly note that since $G^{\gamma}=G^{\gamma}(H^{\gamma}-z\mathbb{I})G^{\gamma}$, it is true that
\[\frac{d}{d \gamma}G^{\gamma}=\frac{d}{d \gamma}G^{\gamma}(H^{\gamma}-z\mathbb{I})G^{\gamma}+G^{\gamma}(H^{\gamma}-z\mathbb{I})\frac{d}{d \gamma}G^{\gamma}+G^{\gamma}\frac{d}{d \gamma}(H^{\gamma}-z\mathbb{I})G^{\gamma}\]
So it is implied that
\begin{equation}\label{isotita paragogon pinakon}
 -\frac{d}{d\gamma}G^{\gamma}=G^{\gamma}\frac{d}{d\gamma}(H^{\gamma})G^{\gamma},
\end{equation}
where the derivative $\frac{d}{d\gamma}$ is considered in every entry. So by \eqref{isotita paragogon pinakon} and Leibniz integral rule, it is true that
\[\left| \frac{d}{d\gamma}\E_{\Psi} G^{\gamma}_{i,j}\right|=\left|\sum_{p,q \in [2N]:|p-q|\geq N}\E_{\Psi}  
 G^{\gamma}_{i,p}\left(A_{p,q}-\frac{\gamma t^{1/2}w_{p,q}}{(1-\gamma^{2})^{1/2}}\right)G_{q,j}^{\gamma}\right|.\]
 Thus, in order to prove \eqref{anisotita apo theorima 3.15 goe} it is sufficient to show that there exists a constant $C=C(a,\nu,\rho)>0$ such that $\text{ for all } \gamma \in (0,1)$  
 \begin{equation}\label{Paragogos anisotitas 3.15 goe}
\sum_{p,q \in [2N]:|p-q|\geq N} \left|\E_{\Psi}  
 \operatorname{Im}(G^{\gamma}_{i,p}G_{q,j}^{\gamma})\left(A_{p,q}-\frac{\gamma t^{1/2}w_{p,q}}{(1-\gamma^{2})^{1/2}}\right)F'(\operatorname{Im}(G_{i,j}^{\gamma})\right| \leq   \frac{C}{(1-\gamma^{2})^{1/2}}(N^{-\omega}(\Xi+1)+ Q_{0}N^{C+C_{0}})
 \end{equation}
 and then integrate over any interval of the form $(0,\gamma')$ with $\gamma' \in (0,1]$.
The proof of \eqref{Paragogos anisotitas 3.15 goe} is completely analogous to the proof of Proposition 4.4 in \cite{aggarwal2021goe}. So we will give a sketch of the proof. Firstly fix $p,q \in [2N]:|p-q|\geq N$ and set the matrices 
\begin{align}
    &D_{a,b}=\begin{cases}H^{\gamma}_{a,b}, & (i,j)\notin \{(p,q),(q,p)\} \\ X_{p,q}, & \text{else}
\end{cases}
\\&
E_{a,b}=\begin{cases}H^{\gamma}_{a,b}, & (a,b)\notin \{(p,q),(q,p)\} \\ C_{p,q}, & \text{else}
\end{cases}.
\end{align}
Moreover set 
\begin{align}
&\Gamma=H^{\gamma}-D, \ \ \ \ \ \Lambda=D-E 
\\&R=(D-z\mathbb{I})^{-1}, \ \ \ \ \ U=(E-z\mathbb{I})^{-1}.
\end{align}
So by Lemma \ref{BASIKES ANISOTITES GIA RESOVLENT}, and as we have mentioned in the proof of Theorem \ref{Theorima Local pertubed gia diagonious}, one can apply Taylor's Theorem for matrices to get that 
\begin{equation}\label{taylor expantion gia pinakes apodiksi Goe 3.15}
  G^{\gamma}-R=-R\Gamma R+(R \Gamma)^{2} R -(R \Gamma)^{3}G^{\gamma}.
\end{equation} 
Moreover, by a Taylor expansion for the function $F'$, it is true that for some $\zeta_{0} \in [\operatorname{Im}(G^{\gamma}_{i,j}),\operatorname{Im}(R_{i,j})]$ and $\zeta=\operatorname{Im}(R_{i,j})-\operatorname{Im}(G^{\gamma}_{i,j})$,
\begin{equation}\label{taylor expation apodiksi Goe 3.15} 
  F'(\operatorname{Im}(G^{\gamma}_{i,j})= F'(\operatorname{Im} R_{i,j})+\zeta F^{(2)}(\operatorname{Im} R_{i,j})+\frac{\zeta^{2}}{2}F^{(3)}(\operatorname{Im} R_{i,j})+\frac{\zeta^{3}}{6} F^{(4)}(\zeta_{0}),
\end{equation}
where we have denoted $F^{(l)}(x)=\frac{d^{l}}{dx^{l}}F(x)$ for all $l \in \N$. 

So by combining \eqref{taylor expantion gia pinakes apodiksi Goe 3.15} and \eqref{taylor expation apodiksi Goe 3.15}, one can notice that each of the (p,q)-summand in \eqref{Paragogos anisotitas 3.15 goe} can be viewed as a sum of finite number of monomials of $A_{p,q}$ and $t^{1/2}w_{p,q}$ with coefficients depending on the matrices $R$ and $G^{\gamma}$.
These monomials can be categorized into the following cases:
\begin{enumerate}
  \item The product of even degree of terms, i.e., $\prod_{r=1}^{s} \xi^{k_{r}}_{i,j}$ such that $\sum k_{r}$ is even and $\xi_{i_{r},j_{r}}^{k_{r}}$ is equal either to $\left((R\Gamma)^{k_{r}}R\right)_{i_{r},j_{r}}$, either equal to $\operatorname{Im}(\left((R\Gamma)^{k_{r}}R\right)_{i_{r},j_{r}})$, either to $\operatorname{Re}(\left((R\Gamma)^{k_{r}}R\right)_{i_{r},j_{r}})$  for some $s \in \N$ and $k_{r} \in \{0,1,2\}$. Then for any $m \in \{1,2,3\}$ it is true that
  \[\E_{\Psi}F^{(m)}\left(\operatorname{Im}(R_{i,j})\right)\left(A_{p,q}-\frac{\gamma t^{1/2}w_{p,q}}{(1-\gamma^{2})^{1/2}}\right) \prod_{r=1}^{s}\xi_{i_{r},j_{r}}=0,\]
  which is a consequence to the independence of the matrix $R$ from $A_{p,q},w_{p,q}$ after further conditioning on the matrix $X$, and the symmetry of the random variables $A_{p,q},w_{p,q}$, which has a consequence that every odd moment of them is 0. \item The terms that contain $F^{(4)}(\zeta_{0})$ can be bounded by a Taylor expansion similarly to Lemma 4.7 in \cite{aggarwal2021goe}. More precisely one can show that
  \[\left|\E_{\Psi}\operatorname{Im}(G^{\gamma}_{i,p}G^{\gamma}_{q,j})\left(A_{p,q}-\frac{\gamma t^{1/2}w_{p,q}}{(1-\gamma^{2})^{1/2}}\right)\zeta^{3}F^{(4)}(\zeta_{0})\right|\leq N^{-2} \frac{C}{(1-\gamma^{2})^{1/2}}(N^{-\omega}(\Xi+1)+ Q_{0}N^{11+C_{0}}),\]
  for parameters $\omega>\epsilon_0>0$, such that 
  \begin{align}&\label{(4.25)}
      \epsilon_0:= \frac{a}{100}\min \{(4-a)\nu-1, (2-a)\nu - a\rho, \nu-\rho, \frac{\rho}{2},1\},
      \\&
      \omega:= \min\{(a-2\epsilon_0)\rho-15\epsilon_0, (2-a)\nu -a\rho-15\epsilon_0, (4-\alpha)\nu-1-10\epsilon_0, (4-2a)\nu-15\epsilon_0\}
  \end{align}
These parameters also appear in (4.25) of \cite{aggarwal2021goe}. 
  \item Analogously to the previous bound, one can prove that for the $s-$ products of $\xi_{i,j}^{k_{r}}$, when $s \in \{1,2,3,4\}$, $k_{r} \in \{1,2,3\} $ and $\sum k_{r}\geq 3$, it holds that for any $m \in \{1,2,3\}$,
  \begin{equation}\label{taylor gia monodiastati apodiksi 3.15}
  \left|\E_{\Psi}F^{(m)}(\operatorname{Im}(R_{i,j}))\left(A_{p,q}-\frac{\gamma t^{1/2}w_{p,q}}{(1-\gamma^{2})^{1/2}}\right) \prod_{r=1}^{s}\xi_{i_{r},j_{r}}\right| \leq N^{-2} \frac{C}{(1-\gamma^{2})^{1/2}}(N^{-\omega}(\Xi+1)+ Q_{0}N^{11+C_{0}}).
  \end{equation}
  \item The remaining terms are the monomials of $2-$ degree. So it can be proven that,
  \begin{align}
&\left|\E_{\Psi} (\operatorname{Im}(R\Gamma R)_{i,p} R_{q,j}) F^{'}(\operatorname{Im}(R_{i,j})\left(A_{p,q}-\frac{\gamma t^{1/2}w_{p,q}}{(1-\gamma^{2})^{1/2}}\right)\right| \\&
\leq N^{-2} \frac{C}{(1-\gamma^{2})^{1/2}}\left[(N^{-\omega}(\Xi+1)+ Q_{0}N^{11+C_{0}})+N^{a\rho +3 \epsilon_0-1}t \Xi(\psi_{p,q}+\mathbf{1}\left\{p=q\right\}) \right],
\\&
  \left|\E_{\Psi}\operatorname{Im}(R_{i,p}R_{q,j})\operatorname{Im}(R\Gamma R)_{i,j}F^{(2)}(\operatorname{Im} R_{i,j})\left(A_{p,q}-\frac{\gamma t^{1/2}w_{p,q}}{(1-\gamma^{2})^{1/2}}\right)\right| 
  \\&
  \leq N^{-2} \frac{C}{(1-\gamma^{2})^{1/2}}\left[(N^{-\omega}(\Xi+1)+ Q_{0}N^{11+C_{0}})+N^{a\rho +3 \epsilon_0-1}t \Xi(\psi_{p,q}+\mathbf{1}\left\{p=q\right\}) \right]
  .\end{align}
  The proof of these inequalities is a consequence of further comparison between the entries of the matrices $R$ and $U$, similar to the one which was done for the matrices $G^{\gamma}$ and $R$ before.
\end{enumerate}

So after summing over all possible (p,q) and taking into account that $t \sim N^{(a-2)\nu}$ and that there are at most $N^{1+a\rho+\epsilon}$ non-zero entries of $\Psi$  with overwhelming probability, see the proof of Corollary \ref{anisotita gia sinartiseis twn metasximatismon stieltjes ws pros gamma}, one has that \eqref{Paragogos anisotitas 3.15 goe} holds, which finishes the proof.
\end{proof}
In what follows, set $C_{a}$ the constant mentioned in Theorem \ref{local law}.
So due to Theorem \ref{theorima 3.15 goe} one can prove the following.
\begin{prop}\label{Protasi 3.17 goe} Let $a,b,\nu ,\rho$ as in \eqref{statheres}. Moreover fix $\varsigma>0$ arbitrary small. Then for each $\delta>0$ and $D>0$ there exists a constant $C=C(a,\rho,\nu,b)$ such that
\begin{equation}\label{veltisto fragma ton G^gamma gia kathe gamma}
\mathbb{P}\left(\sup_{\gamma \in [0,1]}\sup_{E \in [-\frac{1}{2C_{a}},\frac{1}{2C_{a}}]} \sup_{\eta \geq N^{\varsigma}-1} \max_{i,j} |G^{\gamma}_{i,j}(E+i\eta)|\geq N^{\delta}\right)\leq C N^{-D}.  
\end{equation}
The constant $C_a$ in \eqref{veltisto fragma ton G^gamma gia kathe gamma} is the constant mentioned in Theorem \ref{local law}.
\end{prop}
\begin{proof}
The proof is based on Theorem \ref{theorima 3.15 goe} and is similar to the proof of Proposition 3.17 in \cite{aggarwal2021goe}, so we will just describe the key ideas behind the proof.
The proof is done in steps. Set $p=\left\lceil \frac{D+30}{\delta} \right\rceil$ and consider the function $F_{2p}(x)=|x|^{2p}+1$. Note that $F_{2p}$ satisfies the hypothesis of Theorem \ref{theorima 3.15 goe}. Moreover by Corollary \ref{Veltisto Fragma gia to X+sqrtG} there exists a constant $C'=C'(a,b,\nu,\rho)$ such that
\[\mathbb{P}\left(\sup_{E \in [-\frac{1}{2C_{a}},\frac{1}{2C_{a}}]} \sup_{\eta \geq N^{\varsigma}-1} \max_{i,j} |G^{0}_{i,j}(E+i\eta)|\right)\leq C' N^{-D}.\]
Fix $\epsilon_0$ and $\omega$, the constants from the application of Theorem \ref{theorima 3.15 goe} for the function $F_{p}$. Moreover define the quantities
\[\mathcal{B}(\delta,\eta)=\mathbb{P}\left(\sup_{\gamma \in [0,1] } \max_{i,j} |G^{\gamma}(E+i\eta)|\geq N^{\delta}\right),\]
for $E \in \left[-\frac{1}{2C_{a}},\frac{1}{2C_{a}}\right]$ and $\eta \geq N^{\varsigma-1}$. Set $s=\frac{\epsilon}{4}$. Then one can show that there exists a constant $A=A(\delta,D)$ such that,
\begin{equation}\label{Lemma 4.3 goe}
  \mathcal{B}(\delta,\eta)\leq A N^{A}\mathcal{B}(\frac{\epsilon_0}{2},N^{\sigma}\eta) +A N^{-D},
\end{equation}
which can be proven by (i) integrating over $\Psi$ in the conclusion of Theorem \ref{theorima 3.15 goe} for $F_{p}$ after using \eqref{plithos megalon stoixoion tou pinaka}, (ii) Corollary \ref{Veltisto Fragma gia to X+sqrtG}, (iii) Markov's Inequality applied for  $\gamma \in N^{-20}\Z\cap (0,1) $ and (iv) the deterministic estimates in the end of the proof of Lemma 4.3 in \cite{aggarwal2021goe}. 

Thus in order to conclude, one can use induction over all $k \in \left[-1,\left\lceil \frac{1-\varsigma}{s} \right\rceil\right]$ to show that
  \[\mathcal{B}\left(\frac{\epsilon_0}{2},N^{-k\sigma}\right)\leq A N^{-D}\]
and then extend to all $E \in \left[-\frac{1}{3C},\frac{1}{3C}\right]$ and $\eta \geq N^{\varsigma -1}$ by deterministic estimates of the form $|G^{\gamma}(z)-G^{\gamma}(z')|\leq N^{6}|z-z'|$ for an appropriately chosen grid.
\end{proof}
\begin{cor}\label{anisotita gia sinartiseis twn metasximatismon stieltjes ws pros gamma}
  Fix $F: \R \rightarrow \R $ such that it satisfies the assumption of Theorem \ref{theorima 3.15 goe} and $E \in \left[\frac{-1}{3C},\frac{1}{3C}\right]$ and $\eta \geq N^{\varsigma-1}$, for an arbitrary small $\varsigma>0$. Then there exists a constant $c=c(a,b,\nu,\rho,C_{0})$ and a large constant $C=C(a,b,\nu,\rho)$ such that 
  \[\sup_{\gamma \in [0,1]}\left|\E F(\operatorname{Im}(G_{i,j}^{\gamma}(z)))-\E F(\operatorname{Im}(G_{i,j}^{0}(z)))\right|\leq C N^{-c},  \text{ for all } i,j \in [2N].\]
\end{cor}
\begin{proof}
Firstly note that due to Chernoff bound there exists a constant $C'$ such that
\begin{equation}\label{plithos megalon stoixoion tou pinaka}
\mathbb{P}\left(|(i,j):H_{i,j} \in [N^{-\rho},\infty)| \notin \left(\frac{N^{1+a\rho }}{C'},C' N^{1+a\rho}\right)\right)\leq C' \exp\left(\frac{-N}{C'}\right).
\end{equation}
 Set $\Omega= \left\{(i,j):|H_{i,j}| \in [N^{-\rho},\infty)| \in \left(\frac{N^{1+a\rho }}{C'},C' N^{1+a\rho}\right)\right\}$. Moreover by the deterministic estimate $|G^{\gamma}_{i,j}|\leq \eta^{-1} \leq  N$ and the hypothesis for $F$ one has that $|F(\operatorname{Im} (G_{i,j}^{\gamma})| \leq N^{C_{0}}$ and hence, 
\[\left|\E F(\operatorname{Im}(G_{i,j}^{\gamma}(z)))- F(\operatorname{Im}(G_{i,j}^{0}(z)))\right|\leq \left|\E \mathbf{1}\left(\Omega\right) F(\operatorname{Im}(G_{i,j}^{\gamma}(z)))- F(\operatorname{Im}(G_{i,j}^{0}(z)))\right| + N^{C_{0}}C'\exp\left(-\frac{N}{C'}\right). \]
Note that on the set $\Omega$ we can apply Theorem \ref{theorima 3.15 goe}. Moreover by Proposition \ref{Protasi 3.17 goe}, one has that $Q_{0}(z,\epsilon)\leq C N^{-D}$ for any $D>0$ and similarly show that
\[\Xi \leq N^{C_{0}}Q_{0}(z)+C N^{C_{0}\epsilon}. \]
So the proof is complete after choosing an appropriately large $D>0$.
\end{proof}
Next, we extend the comparison result in such way that we can  use in order to approximate the gap probability.
\begin{prop}\label{protasi gia q gia kathe gamma}
  Fix parameters $a,b,\rho,\nu$ as in \eqref{statheres}. Let $q:\R\rightarrow \R$ a $C^{\infty}$ function with all its derivatives bounded by an absolute constant $M$ greater than $1$. Then for any $\eta \geq N^{-2} $ and any positive sequence $r(N)$ such that $\lim r(N)=r>0$ there exist constants $\omega=\omega(a,\rho,\nu,b,r)$, $\epsilon=\epsilon(a,\rho,\nu,b,r)$ and $C=C(a,\rho,\nu,b,r)$ such that 
  \begin{equation}\label{Anisotita gia q sthn geniki periptosi }
  \sup_{\gamma \in [0,1]} \left|\E q\left(\int_{-\frac{r(N)}{N}}^{\frac{r(N)}{N}}\sum_{i=1}^{2N}\operatorname{Im} G_{i,i}^{\gamma}(y+i\eta) dy \right) - \E q\left(\int_{-\frac{r(N)}{N}}^{\frac{r(N)}{N}}\sum_{i=1}^{2N}\operatorname{Im} G_{i,i}^{0}(y+i\eta) dy \right) \right| \leq C\left( M N^{-\omega}+M N^{C} \mathcal{Q}(\epsilon,\eta)\right),
  \end{equation}
  where 
  $$\mathcal{Q}(\epsilon,\eta)=\mathbb{P}\left(\sup_{E \in [-\frac {1}{3C},\frac{1}{3C}]}\max_{i,j} |G_{i,j}^{\gamma}(E+i\eta)|\geq N^{\epsilon}\right).$$
  Moreover if we suppose that $\eta \geq N^{\varsigma-1}$, for arbitrary small $\varsigma>0$, then there exists a constant $c=c(a,\rho,\nu,b)$ such that, 
  \begin{equation}\label{anisotita gia q gia h > N^d-1}
  \sup_{\gamma \in [0,1]} \left|\E q\left(\int_{-\frac{r(N)}{N}}^{\frac{r(N)}{N}}\sum_{i=1}^{2N}\operatorname{Im} G_{i,i}^{\gamma}(y+i\eta) dy \right) - \E q\left(\int_{-\frac{r(N)}{N}}^{\frac{r(N)}{N}}\sum_{i=1}^{2N}\operatorname{Im} G_{i,i}^{0}(y+i\eta) dy \right) \right| \leq C N^{-c} .
  \end{equation}
\end{prop}
\begin{proof}
For simplicity we will assume $r$ is a constant. The proof of \eqref{Anisotita gia q sthn geniki periptosi } is similar to the proof of Theorem \ref{theorima 3.15 goe}. Next we highlight the differences.

Note that similarly to the proof of Corollary \ref{anisotita gia sinartiseis twn metasximatismon stieltjes ws pros gamma}, it is sufficient to prove that for any matrix $\Psi$ with at most $N^{1+a\rho}$ non-zero entries it is true that,
\[\sup_{\gamma \in [0,1]} \left|\E_{\Psi} q\left(\int_{-\frac{r}{N}}^{\frac{r}{N}}\sum_{i=1}^{2N}\operatorname{Im} G_{i,i}^{\gamma}(y+i\eta) dy \right) - \E_{\Psi} q\left(\int_{-\frac{r}{N}}^{\frac{r}{N}}\sum_{i=1}^{2N}\operatorname{Im} G_{i,i}^{0}(y+i\eta) dy \right) \right| \leq C \left( M N^{-\omega}+M N^{C} \mathcal{Q}(\epsilon)\right) .\]

Furthermore one can compute the derivative of the previous quantity with respect to $\gamma$, as in Theorem \ref{theorima 3.15 goe}. Thus by Leibniz integral rule, Fubini Theorem and \eqref{isotita paragogon pinakon} it is true that
\[\left|\frac{d}{d\gamma}\E_{\Psi}q\left(\int_{-\frac{r}{N}}^{\frac{r}{N}}\sum_{i=1}^{2N}\operatorname{Im} G_{i,i}^{\gamma}(y+i\eta) dy \right)\right|\leq \]
\[\leq \int_{-\frac{r}{N}}^{\frac{r}{N}}\sum_{i=1}^{2N} \sum_{p,q \in [2N]:|p-q|\geq N} \left| \E_{\psi}q' \left(\int_{-\frac{r}{N}}^{\frac{r}{N}}\sum_{i=1}^{2N}\operatorname{Im} G_{i,i}^{\gamma}(y+i\eta) dy \right) \operatorname{Im}(G^{\gamma}_{i,p}G^{\gamma}_{q,i})\left(A_{p,q}-\frac{\gamma t^{1/2}w_{p,q}}{(1-\gamma^{2})^{1/2}}\right) dy \right| . \]
As a result it is sufficient to prove that for any $y \in \left(-\frac{r}{N},\frac{r}{N}\right)$, 
\begin{align}\label{anisotita gia theorima 3.15 C-apeiro periptosi}
\sum_{p,q \in [2N]:|p-q|\geq N} \left|\E_{\Psi} q' \left(\int_{-\frac{r}{N}}^{\frac{r}{N}}\sum_{i=1}^{2N}\operatorname{Im} G_{i,i}^{\gamma}(y+i\eta) dy \right) \operatorname{Im}(G^{\gamma}_{i,p}G^{\gamma}_{q,i})\left(A_{p,q}-\frac{\gamma t^{1/2}w_{p,q}}{(1-\gamma^{2})^{1/2}}\right)\right|
\\ \leq \frac{C}{(1-\gamma^{2})^{1/2}}M(N^{-\omega}+\mathcal{Q}(\epsilon,\eta)N^{C}).
\end{align}

But the proof of \eqref{anisotita gia theorima 3.15 C-apeiro periptosi} is similar to the proof of \eqref{Paragogos anisotitas 3.15 goe}, with the main difference located in the Taylor expansion which now instead of being applied as in \eqref{taylor expation apodiksi Goe 3.15}, it will be applied for the quantities $\int_{-\frac{r}{N}}^{\frac{r}{N}}\sum_{i=1}^{2N}\operatorname{Im} G^{\gamma}_{i,i}(y+i\eta)dy$ and $\int_{-\frac{r}{N}}^{\frac{r}{N}}\sum_{i=1}^{2N}\operatorname{Im} R_{i,i}(y+i\eta)dy$ for fixed $p,q$. But eventually, this does not affect the proof since each (p,q)-summand can again be expressed into monomials of $A_{p,q},w_{p,q}$, which do not depend on the parameters $\eta$ and $y$, and since the quantity $Q_{0}(\epsilon,\eta)$ is replaced by $\mathcal{Q}(\epsilon,\eta)$ in the bound.

Moreover, if we assume $\eta \geq N^{\varsigma-1}$, then $\mathcal{Q}(\epsilon,\eta)$ is smaller than $N^{-D}$ for any $D$ and for sufficient large $N$. Thus similarly to the proof of \eqref{anisotita gia sinartiseis twn metasximatismon stieltjes ws pros gamma}, one can prove \eqref{anisotita gia q gia h > N^d-1} by \eqref{Anisotita gia q sthn geniki periptosi }.

\end{proof}
Moreover, we wish to prove that the righthand side of \eqref{Anisotita gia q sthn geniki periptosi } tends to 0 as $N$ tends to infinity for $\eta= \mathcal{O}(\frac{1}{N^{1+\varsigma}})$, below the natural scale. This is achieved via the following lemma.
\begin{lem}[\cite{bauerschmidt2017local},Lemma 2.1]
Let $Y$ be an $N \times N $ matrix. Set the following quantity
$$\Gamma(Y,E+i\eta)= \max\{1, \max_{i,j}|(Y-(E+i\eta)\mathbb{I})^{-1}_{i,j}|)\} .$$
Then for any $M\geq 1$ and $\eta >0$ the following deterministic inequality holds 
\begin{equation}\label{anisotita gia geniki periptosi pinakon me mikro h}
\Gamma\left(Y,E+i\frac{\eta}{M}\right) \leq M \Gamma(Y,E+i\eta).
\end{equation}
\end{lem}
\begin{cor}
Fix $\varsigma$ and $\delta$ arbitrary small positive numbers. Set $\eta_{1}=N^{-\varsigma/2-1}$. Then by \eqref{anisotita gia geniki periptosi pinakon me mikro h} and \eqref{veltisto fragma ton G^gamma gia kathe gamma}
one has that for any $D>0$ and for sufficient large $N$, it is true that 
\begin{equation}\label{veltisto fragma gia mikra h me megali pithanotita}
  \mathbb{P}\left(\sup_{\gamma \in [0,1]}\sup_{E \in [-\frac{1}{3C},\frac{1}{3C}]} \max_{i,j} |G^{\gamma}_{i,j}(E+i\eta_{1})|\geq N^{\delta+\varsigma}\right)\leq C N^{-D}. 
  \end{equation}
So in the setting of Proposition \ref{protasi gia q gia kathe gamma}, it is implied that there exist two positive constants $C=C(a,b,\nu,\rho,r)$, $c=c(a,b,\rho,\nu)$ such that 
\begin{equation}\label{anisotita Capeiro sinartiseon gia kathe gamma gia mikro h}
\sup_{\gamma \in [0,1]} \left|\E q\left(\int_{-\frac{r(N)}{N}}^{\frac{r(N)}{N}}\sum_{i=1}^{2N}\operatorname{Im} G_{i,i}^{\gamma}(y+i\eta_{1}) dy \right) - \E q\left(\int_{-\frac{r(N)}{N}}^{\frac{r(N)}{N}}\sum_{i=1}^{2N}\operatorname{Im} G_{i,i}^{0}(y+i\eta_{1}) dy \right) \right| \leq C N^{-c}.
\end{equation}
\end{cor}

\subsection{Approximation of the gap probability}
The goal of this subsection is to approximate the gap probability, i.e., the probability that there are no eigenvalues in an interval, by $C^{\infty}$ functions of the Stieltjes transform as in Proposition \ref{protasi gia q gia kathe gamma}. In order to prove the latter, we use similar tools as in Section 5 of \cite{aggarwal2021goe}. First, define the following quantities for any $r>0,\gamma>0$ and $\eta \geq0$.  
\begin{equation} \label{orismos free convolution me trace}
  \begin{aligned}
   &X_{x}(t)=\mathbf{1}\left\{t \in (-x,x)\right\}, \text{ for all } x \in \R
   \\& \theta_{\eta}(x)=\frac{1}{\pi}\operatorname{Im}\left(\frac{1}{x-i\eta}\right)=\frac{1}{\pi} \frac{\eta}{\eta^{2}+x^{2}},  \text{ for all } x \in \R
   \\& 
   \operatorname{Tr} \ X_{r}\star \theta_{\eta}(H^{\gamma})= \frac{1}{\pi}\int_{-\frac{r}{2N}}^{\frac{r}{2N}} \sum_{i=1}^{2N} \operatorname{Im} G_{i,i}^{\gamma}(x+i\eta)dx=\frac{1}{\pi}\int_{-\frac{r}{2N}}^{\frac{r}{2N}}\sum_{i=1}^{2N}\frac{\eta}{\left(\lambda_{i}(H^\gamma)-x\right)^2+\eta^2}dx  .
  \end{aligned}
\end{equation}
Moreover for any $N \times N$ matrix $Y$ with eigenvalues denoted by $\lambda_{i}(Y)$ and for any $E_{1},E_{2},E \in \R$ such that $E_{1}\leq E_{2}$ and $E >0$, we denote
\begin{equation}
\begin{aligned}
 i_{N}(Y,E_{1},E_{2})=\#\{i \in [N]: \lambda_{i}(Y) \in (E_{1},E_{2})\}, 
 \\ i_{N}(Y,E)=\#\{i \in [N]:\lambda_{i}(Y) \in (-E,E)\}.
\end{aligned}
\end{equation}
Moreover set $\{\lambda_{i}^{\gamma}\}_{i \in [2N]}$ the eigenvalues of $H^{\gamma}$ arranged in increasing order.
\begin{lem}\label{plithos idiotimon se diastima}
For any $\gamma \in [0,1]$ and $I \subseteq \left(-\frac{1}{3C},\frac{1}{3C}\right)$ such that $|I|=N^{\varsigma/2-1}$, it is true that 
\[\left|\{i \in [2N]: \lambda_{i}^{\gamma} \in I\}\right|\leq 2 |I| N^{1+\varsigma/2}\]
with overwhelming probability. 
\end{lem}
\begin{proof}
For the convenience of notation, suppose that 
 $$I=(E-\eta,E+\eta)$$.

 Moreover set the event
\begin{align}
\Omega_\eta=\left\{\sup_{\gamma \in [0,1]}\sup_{E \in [-\frac{1}{3C},\frac{1}{3C}]}\max_{i,j} |G^{\gamma}_{i,j}(E+i\eta)|\geq N^{\varsigma/2}\right\}.
\end{align}
 By \eqref{Protasi 3.17 goe}, $\Omega^{c}_{\eta}$ holds with overwhelming probability.  Then  
\begin{align}
&\mathbf{1}\left(\Omega^{c}_{\eta}\right)N^{\varsigma/2}\geq \frac{\mathbf{1}\left(\Omega^{c}_{\eta}\right)}{2N} \sum_{i=1}^{2N}\operatorname{Im}(G_{i,i}^{\gamma}(E+i\eta))\geq \frac{\mathbf{1}\left(\Omega^{c}_{\eta}\right)}{2N} \sum_{ \lambda_{i} \in I}\operatorname{Im}(G_{i,i}^{\gamma}(E+i\eta))=\frac{\mathbf{1}\left(\Omega^{c}_{\eta}\right)}{2N}\sum_{ \lambda_{i} \in I} \frac{\eta}{(\lambda_{i}-E)^{2}+\eta^{2}}\\&\geq \frac{\mathbf{1}\left(\Omega^{c}_{\eta}\right)}{2N|I|} |\{i \in [2N]: \lambda_{i}^{\gamma} \in I\}|
\end{align}
\end{proof}
Next fix $\epsilon>0$ arbitrary small and $r \in \R$. Set 
$$\eta_{1}=N^{-1-99\epsilon}, \ \ \ \ \ l=N^{-1-3\epsilon}, \ \ \ \ \ \ l_{1}=lN^{2\epsilon}.$$
\begin{lem}\label{proti sigrisi i,trX}
For any $\gamma \in [0,1]$, it is true that there exists an absolute constant $C$ such that with overwhelming probability
\[\left|i_{2N}\left(H^{\gamma},\frac{r}{N}\right)- \operatorname{Tr} \  X_{r}\star \theta_{\eta_{1}}(H^{\gamma})\right|\leq C \left(N^{-2\epsilon}+ i_{2N}\left(H^{\gamma},-\frac{r}{N}-l,-\frac{r}{N}+l\right)+i_{2N}\left(H^{\gamma},\frac{r}{N}-l,\frac{r}{N}+l\right)\right). \]
\end{lem}
\begin{proof}
Firstly note that by elementary computation as in (6.10) of \cite{erdHos2012rigidity} one has that, 
\begin{equation}\label{anisotita me xaraktiristiki kai convolution}
|X_{\frac{r}{N}}(x)- X_{r}\star \theta_{\eta_{1}}(x)|\leq C\eta_{1}\left(\frac{2r}{N d_{1}(x)d_{2}(x)}+\frac{X_{r/N}(x)}{d_{1}(x)+d_{2}(x)}\right),
\end{equation}
where $C$ is some absolute constant, $d_{1}(x)=\left|\frac{r}{N}+x\right|+\eta_1$ and $d_{2}=\left|\frac{r}{N}-x\right|+\eta_1$. Moreover note that the right hand side of \eqref{anisotita me xaraktiristiki kai convolution} is always bounded by an absolute constant and is $\mathcal{O}(\eta_{1}/l)$ if $\min d_{i}\geq l$. Thus by Lemma \ref{plithos idiotimon se diastima} one has that with overwhelming probability

\begin{align}
&\left|i_{2N}\left(H^{\gamma},\frac{r}{N}\right)- \operatorname{Tr} \  X_{r}\star \theta_{\eta_{1}}(H^{\gamma})\right|\leq C \left(\operatorname{Tr} \ (f_{1}( H^{\gamma})+\operatorname{Tr}(f_{2}(H^{\gamma})) +\frac{\eta}{l}i_{2N}\left(H^{\gamma},\frac{-r}{N}+l, \frac{r}{N}-l\right)\right) \\&   +C\left(i_{2N}\left(H^{\gamma},-\frac{r}{N}-l,-\frac{r}{N}+l\right)+i_{2N}\left(H^{\gamma},\frac{r}{N}-l,\frac{r}{N}+l\right)\right),
\end{align}
where 
$$f_{1}(x)=\mathbf{1}\left\{x \leq -E-l\right\} \frac{2r\eta_{1}}{Nd_{1}(x)d_{2}(x)},\ \ \ \ f_{2}(x)=\mathbf{1}\left\{x \geq E+l\right\} \frac{2r \eta_{1}}{Nd_{1}(x)d_{2}(x)}.$$

So in order to complete the proof we need to show that the first term on the right side of the inequality is of order $N^{-2\epsilon}$. Note that due to Lemma \ref{plithos idiotimon se diastima} and the fact that the length of the interval $\left(-\frac{r}{N}-l,\frac{r}{N}+l\right)$ is smaller than $N^{\varsigma-1}$ for any $\varsigma>0$ one has that $\frac{\eta_1}{l}i_{2N}\left(H^{\gamma},-\frac{r}{N}-l,\frac{r}{N}+l\right) \leq N^{-2\epsilon}$ with overwhelming probability.

Moreover after splitting the interval $\left(-\frac{1}{3C},-\frac{r}{N} -l\right)$ into intervals with length $\mathcal{O}(N^{-1})$, like in \cite{che2019universality} (5.61) and since $$f_{1}(x)\leq \frac{\eta_{1}}{|E-x|}\mathbf{1}\left\{x \in \left(\frac{-1}{3C},-\frac{r}{N}-l\right)\right\}+ \frac{\eta_{1}}{|E-x|}\mathbf{1}\left\{x \in \left(-\infty,\frac{-1}{3C}\right)\right\}, $$
one can show that $\operatorname{Tr} f_{1}(H^{\gamma}) \preceq N^{-2\epsilon}$. Similar bound can be proven for $\operatorname{Tr}f_{2}(H^{\gamma})$.
\end{proof}
\begin{lem}\label{5.12}
For any $\gamma \in [0,1]$ there exists an absolute constant $C$ such that
\begin{align}\label{trX_r ,i}
\operatorname{Tr} \ X_{r}\star \theta_{\eta_{1}-l_{1}}(H^{\gamma}) -CN^{-\epsilon}\leq i_{2N}\left(H^{\gamma},\frac{r}{N}\right)\leq \operatorname{Tr} \ X_{r}\star \theta_{\eta_{1}+l_{1}}(H^{\gamma})+CN^{-\epsilon}.
\end{align}
\end{lem}
\begin{proof}
We will prove the second inequality of \eqref{trX_r ,i}. The proof of the first inequality is similar.

First note that by definition, one has that for any $y_1\geq y_2 >0$,
\begin{align}\label{i(y)<}
    &i_{2N}(H^{\gamma},y_2) \leq i_{2N}( H^{\gamma},y_1)
\\\label{Tr(X_r)<}& \operatorname{Tr} \  X_{y_2}\star \theta_{\eta_{1}}(H^{\gamma})\leq  \operatorname{Tr} \ X_{y_1}\star \theta_{\eta_{1}}(H^{\gamma}))
\end{align}
So we get that with overwhelming probability
\begin{align}\label{i_2N < tr}
    &i_{2N}\left(H^{\gamma},\frac{r}{N}\right) \leq \frac{1}{l_1}\int_{\frac{r}{N}}^{\frac{r}{N}+l_1} i_{2N}\left(H^{\gamma},\frac{r}{N}+y\right)d_y\\& \leq \frac{1}{l_1}\left(  \int_{\frac{r}{N}}^{\frac{r} {N}+l_1}\operatorname{Tr} \  X_{Ny}\star \theta_{\eta_{1}}(H^{\gamma})+ C \left(N^{-2\epsilon}+ i_{2N}(H^{\gamma},y-l,-y+l)+i_{2N}(H^{\gamma},-y-l,-y+l)\right)dy\right)
    \\& \leq \operatorname{Tr} \ X_{r}\star \theta_{\eta_{1}+l_{1}}(H^{\gamma})+CN^{-\epsilon}
\end{align}
In the first inequality of \eqref{i_2N < tr} we used \eqref{i(y)<}, in the second we used Lemma \ref{proti sigrisi i,trX} and in the third we used \eqref{Tr(X_r)<} for the first term in the sum and Lemma \ref{plithos idiotimon se diastima} for the second. 
\end{proof}
Next we proceed as in Lemma 5.13 of \cite{che2019universality}. Set $\tilde{q}(x): \R \rightarrow \R_{+}$ be a $C^{\infty}$, even function, with all its derivatives bounded by a constant $M$, such that
\begin{itemize}
  \item $\tilde{q}(x)=0$ for $ x \in \left(- \infty, \frac{-2}{9}\right)\cup \left(\frac{2}{9}, \infty\right)$
  \item $\tilde{q}(x)=1$ for $x \in \left(\frac{-1}{9},\frac{1}{9}\right)$
  \item $\tilde{q}(x)$ is decreasing on $\left(\frac{1}{9},\frac{2}{9}\right)$.
\end{itemize}
In the following Lemma we prove the approximation of the gap probability of $H^{\gamma}$ by function of the form appearing in \eqref{anisotita Capeiro sinartiseon gia kathe gamma gia mikro h}.
\begin{lem}\label{anisotita poy sigrini to gap probability me Capeiro}
For any $\gamma \in [0,1]$ and $D>0$ it is true that
\begin{align}\label{q<i<q}
\E \tilde{q}\left(\operatorname{Tr} \ X_{r}\star \theta_{\eta_{1}+l_{1}}(H^{\gamma})\right)-N^{-D}\leq \mathbb{P}\left(i_{2N}\left(H^{\gamma},\frac{r}{N}\right)=0\right)\leq \E \tilde{q}\left(\operatorname{Tr}\ X_{r}\star \theta_{\eta_{1}-l_{1}}(H^{\gamma})\right)+N^{-D} .
\end{align}
\end{lem}
\begin{proof}
By Lemma \ref{5.12} and for large enough $N$, it is true that if $i_{2N}\left(H^{\gamma},\frac{r}{N}\right)=0$ then $\operatorname{Tr}\ X_{r}\star \theta_{\eta_{1}-l_{1}}(H^{\gamma})\leq \frac{1}{9}$ with overwhelming probability. This implies that for any large $D>0$ and for $N$ sufficiently large one has that,
\begin{align}
&\mathbb{P}\left(i_{2N}\left(H^{\gamma},\frac{r}{N}\right)=0 \right) \leq \mathbb{P}\left(\operatorname{Tr}\ X_{r}\star \theta_{\eta_{1}-l_{1}}(H^{\gamma})\leq \frac{1}{9} \right)+ N^{-D} \leq \mathbb{P}\left(\operatorname{Tr}\ X_{r}\star \theta_{\eta_{1}-l_{1}}(H^{\gamma})\leq \frac{2}{9} \right) +N^{-D}
\\\label{Markov q}&= \mathbb{P}\left[ \tilde{q}\left( \operatorname{Tr}\ X_{r}\star \theta_{\eta_{1}-l_{1}}(H^{\gamma})\right)\geq 1 \right]+N^{-D}\leq \E \tilde{q}\left( \operatorname{Tr}\ X_{r}\star \theta_{\eta_{1}-l_{1}}(H^{\gamma})\right) +N^{-D} 
\end{align}
In \eqref{Markov q}, we used the Markov inequality for the random variable $\tilde{q}\left( \operatorname{Tr}\ X_{r}\star \theta_{\eta_{1}-l_{1}}(H^{\gamma})\right)$. So we have proven the second inequality of \eqref{q<i<q}.

For the first, note that again by Lemma \ref{5.12}, with overwhelming probability it is true that if $\tilde{q}\left( \operatorname{Tr}\ X_{r}\star \theta_{\eta_{1}+l_{1}}(H^{\gamma})\right)\leq \frac{2}{9}$ then $i_{2N}(H^{\gamma},\frac{r}{N})\leq  CN^{-\epsilon}+ \frac{2}{9}$. Thus,
\begin{align}
    &\E \tilde{q}\left( \operatorname{Tr}\ X_{r}\star \theta_{\eta_{1}+l_{1}}(H^{\gamma})\right)\leq \mathbb{P}\left[ \tilde{q}\left( \operatorname{Tr}\ X_{r}\star \theta_{\eta_{1}+l_{1}}(H^{\gamma})\right)\leq \frac{2}{9} \right]\\& \leq  \mathbb{P}\left[i_{2N}\left(H^{\gamma},\frac{r}{N}\right)\leq  CN^{-\epsilon}+ \frac{2}{9}\right]+N^{-D}= \mathbb{P}\left[i_{2N}\left(H^{\gamma},\frac{r}{N}\right)=0\right]+N^{-D}.
\end{align}
\end{proof}
\subsection{Proof of Theorem \ref{TO THEORIMA}}\label{subsection pou deixnoume to theorima}
At this subsection we prove Theorem \ref{TO THEORIMA}. Fix $r\in (0,\infty)$ and $\epsilon>0$ small enough. Set $\eta_{1}=N^{-1-\epsilon}$ and $l=N^{-1-99\epsilon}$. Furthermore $r\in (0,\infty)$. Let $\tilde{q}(x)$ denote the function defined before Lemma \ref{anisotita poy sigrini to gap probability me Capeiro}.
\begin{itemize}
  \item For the first part of Theorem \ref{TO THEORIMA} note that due to \eqref{anisotita Capeiro sinartiseon gia kathe gamma gia mikro h} and Lemma \ref{anisotita poy sigrini to gap probability me Capeiro} one has that there exist constants $C=C(r)>0$ and $c>0$, such that for large enough $D>0$ it is true that
\begin{align}
&\E \tilde{q}\left(\operatorname{Tr}\ X_{r}\star \theta_{\eta_{1}+l}(H^{0})\right)-N^{-D} -CN^{-c}\leq \E \tilde{q}\left(\operatorname{Tr}\ X_{r}\star \theta_{\eta_{1}+l}(H^{1})\right)-N^{-D}\\&
\leq \mathbb{P}\left(i_{2N}\left(H^{1},\frac{r}{N}\right)=0\right)\leq \E \tilde{q}\left(\operatorname{Tr}\ X_{r}\star \theta_{\eta_{1}-l}(H^{1})\right)+N^{-D} \leq \E \tilde{q}\left(\operatorname{Tr}\ X_{r}\star \theta_{\eta_{1}-l}(H^{0})\right)+CN^{-c}+N^{-D} .
\end{align}
\quad Next note that by the definition of the symmetrization of a matrix, the gap probability is actually the tail distribution of the smallest singular value, i.e.,
\[\mathbb{P}\left(i_{2N}\left(H^{1},\frac{r}{N}\right)=0\right)=\mathbb{P}\left(s_{1}(D_{N})\geq \frac{r}{N}\right).\]
\quad Moreover note that the limiting distribution of the least singular value of a Gaussian matrix is $1-\exp(-r^{2}/2-r)$ as mentioned in Theorem 1.3. of \cite{tao2010random}. Let $L_{N}$ be a matrix with i.i.d. entries all following the Gaussian law with mean $0$ and variance $N^{-1}$. Set $s_{1}(L_{N})$ the least singular value of $L_{N}$. Let $W_{N}$ be the symmetrization of $L_{N}$. As before one can notice that
\[\mathbb{P}\left(i_{2N}\left(E_{N},\frac{r}{N}\right)=0\right)=\mathbb{P}\left(s_{1}(L_{N})\geq \frac{r}{N}\right).\]
\quad So after another application of Lemma \ref{anisotita poy sigrini to gap probability me Capeiro} for the matrix $H^{0}$ and Corollary \ref{universality gia t} for $r'=r \xi^{-1}$, where $\xi$ is defined in \eqref{semicirle law kai logos}, one has that there exists a small constant $\tilde{c}>0$ and a large constant such that 
\[\mathbb{P}\left(N s_{1}(L_{N})\geq r-N^{-\epsilon} \right)- CN^{-\tilde{c}} \leq \mathbb{P}\left(\xi N s_{1}(D_{N})\geq r\right) \leq \mathbb{P}\left(N s_{1}(L_{N})\geq r+N^{-\epsilon} \right) + CN^{-\tilde{c}},\]
which implies universality of the least singular value for $D_{N}$ multiplied by $N\xi$. 
\item For the proof of the second part, it is well-known that bounding the entries of the resolvent implies the complete eigenvector delocalization. So by \eqref{veltisto fragma ton G^gamma gia kathe gamma}, one can prove the complete eigenvector delocalization as in Theorem 6.3 in \cite{huang2015bulk}. 
\end{itemize}

\begin{acks}[Acknowledgments]
I would like to thank Prof. Dimitris Cheliotis for his useful comments and suggestions while I was working on this paper.
Moreover I would like to express my gratitude to the anonymous referees for their useful comments.
\end{acks}

%\nocite{*} 
\bibliographystyle{IEEEtran}

\bibliography{template}
\addcontentsline{toc}{chapter}{References}
\end{document}